\documentclass[11pt]{article}

\usepackage{geometry}
\usepackage{latexsym}
\usepackage{mathrsfs}
\usepackage{amsfonts}
\usepackage{amssymb}
\usepackage{subfigure}    %
\usepackage{graphicx}
\usepackage{tikz}
\usetikzlibrary{arrows.meta,positioning,calc,fit,backgrounds}
\usepackage{epstopdf}
\usepackage{float}
\usepackage{multirow}
\usepackage{booktabs}
\usepackage{longtable}
\usepackage{threeparttable}
\usepackage{bm}
\usepackage{amsmath}
\usepackage{amsthm}
\usepackage{color}
\usepackage[subnum]{cases}
\usepackage{rotating}
\usepackage{enumitem}
\usepackage{accents}
\usepackage{algorithm}
\usepackage{algorithmic}
\usepackage{mathtools}
\usepackage{caption}
\usepackage{url}
\usepackage[numbers]{natbib}
\makeatletter
\@addtoreset{equation}{section}
\makeatother

\makeatletter
\def \wideubar{\underaccent{{\cc@style\underline{\mskip15mu}}}}
\def \widebar{\accentset{{\cc@style\underline{\mskip10mu}}}}
\makeatother

\definecolor{lightgreen}{rgb}{0.1,0.5,0.1}
\usepackage[colorlinks=true,
breaklinks=true,
bookmarks=true,
urlcolor=black,
citecolor=lightgreen,
linkcolor=lightgreen,
bookmarksopen=false,
draft=false]{hyperref}

\graphicspath{{images/}}

\newcommand{\R}{\mathbb{R}}
\newcommand{\cX}{\mathcal{X}}
\newcommand{\cY}{\mathcal{Y}}

\newcommand{\tb}[1]{\textbf{#1}}
\newcommand{\coleq}{\coloneqq}
\newcommand{\inprod}[2]{\langle#1,\,#2\rangle}
\DeclareMathOperator*{\argmin}{argmin}

\DeclareMathOperator{\Diag}{Diag}

\definecolor{otclarification}{rgb}{0,0,0}
\newcommand{\otclarify}[1]{\textcolor{otclarification}{#1}}

\begin{document}

\newtheorem{property}{Property}[section]
\newtheorem{proposition}{Proposition}[section]
\newtheorem{definition}{Definition}[section]
\newtheorem{lemma}{Lemma}[section]
\newtheorem{corollary}{Corollary}[section]
\newtheorem{theorem}{Theorem}[section]
\newtheorem{remark}{Remark}[section]
\newtheorem{problem}{Problem}[section]
\newtheorem{example}{Example}[section]
\newtheorem{assumption}{Assumption}
\renewcommand*{\theassumption}{\Alph{assumption}}

\title{A Logarithmic-dual-update Bregman ADMM   for Optimal Transport}
\author{
Di Hou\thanks{Department of Mathematics, National
University of Singapore, Singapore 119076 ({\tt dihou@u.nus.edu}).}
\and
Kim-Chuan Toh\thanks{Department of Mathematics, National
University of Singapore, Singapore 119076 ({\tt mattohkc@nus.edu.sg}).}
}
\date{}\maketitle
\begin{abstract}
{
This paper studies an entropy Bregman alternating direction method of
multipliers (BADMM) with logarithmic dual update for solving discrete
optimal transport (OT) problems. The algorithm uses simple row and
column normalizations without inner iterations, allowing efficient
CPU and GPU implementations. However, the convergence of this method for general OT problems
has not been established. Starting from a strictly positive feasible
primal point and a zero multiplier, here we establish the $R$-linear convergence
of the primal--dual sequence for general OT costs with any fixed
multiplier stepsize in $(0,2)$ and a sufficiently large fixed penalty,
even when the primal limit lies on the boundary. Our analysis further
characterizes the primal limit as the entropy Bregman projection of the
initial point onto the optimal solution set and establishes strict
complementarity of the limiting OT primal--dual solution. For separable
costs on two-dimensional grids, we exploit the separable structure  to
accelerate matrix--vector products and reduce storage requirements.
Numerical experiments show competitive performance of our BADMM for general cost
matrices with substantial speedups over state-of-the-art solvers, including
HOT and HALO, for squared Euclidean costs on two-dimensional-grid marginals.
The largest tested problem, with $2048\times2048 \approx 4.19$M grid points in each
marginal, is solved in approximately 100 seconds on a single GPU.
}
\end{abstract}
\section{Introduction}
\subsection{Optimal transport}
We consider the optimal transport (OT) problem
\begin{equation}\label{eq:gw}
\min\{\inprod CP\mid P\in\Pi(a,b)\},\qquad
\Pi(a,b)\coleq\{P\geq0\mid Pe_n=a,\ P^\top e_m=b\},
\end{equation}
where $C\in\R^{m\times n}$ is the cost matrix, $a\in\R^m$ and $b\in\R^n$
are positive vectors with sums one, and $m,n\geq2$.
We write $\Pi\coleq\Pi(a,b)$ and use $e_d$ for the all-ones vector in $\R^d$.
OT compares probability distributions by minimizing a linear cost
subject to prescribed row and column sums. OT and its regularized variants
have applications in image retrieval~\cite{rubner2000earth},
document comparison~\cite{kusner2015word},
domain adaptation~\cite{courty2017domain}, and
geometric data processing~\cite{solomon2015convolutional}.
For more applications, see~\cite{peyre2019computational}.
These applications motivate efficient algorithms for large transport
problems. As a linear programming (LP) problem, OT can be solved
by general-purpose LP solvers~\cite{zanetti2025hipo}
or specialized network simplex implementations such as POT~\cite{POT}.
Further approaches for unregularized OT and general LPs include Bregman
proximal point methods~\cite{yang2022bregman,chu2023ieppa,xie2020ipot},
splitting and primal--dual methods~\cite{mai2022drot,lu2024pdot,chen2025hprlp,applegate2025pdlp,zhang2025hot},
Newton methods~\cite{li2020snipal,hou2024sparse}, and multiscale
interior-point methods~\cite{sun2026multiscale}.
Recent solvers further exploit GPU parallelism~\cite{wang2014bregman,lu2025cupdlpx}
and hierarchical problem reduction~\cite{xia2026halo,xia2026hello}.

Entropic regularization offers a computationally attractive
alternative through the simple row and column scalings of the
Sinkhorn--Knopp algorithm~\cite{sinkhorn1967concerning,cuturi2013sinkhorn}.
However, a fixed positive regularization parameter generally changes the optimal
solution and its transport cost~\cite{weed2018entropic}, while small
parameters can cause numerical overflow or underflow and slow convergence
in direct scaling implementations. Log-domain stabilization alleviates
these numerical difficulties~\cite{schmitzer2019stabilized} but does not
remove the regularization bias. This motivates methods that solve
the original problem while retaining the simple computations associated
with entropy.

To avoid regularization bias, one can use entropy as a Bregman proximal term instead of
adding a fixed regularization term to the objective
\cite{teboulle1992entropic,chen1993proximal,yang2022bregman,xie2020ipot}.
The resulting Bregman proximal point algorithm targets the original OT
problem without requiring the proximal parameter to tend to zero.
Equivalently, it applies the exponential method of multipliers to the
OT dual problem~\cite[Section~2]{tseng1993exponential}.
However, each proximal step requires an entropy projection
satisfying both marginal constraints, which generally involves an inner
loop of row and column normalizations.
\subsection{Algorithmic framework}\label{sec:ot-framework}
To retain simple normalization steps without inner iterations,
we consider a Bregman alternating direction method of multipliers (ADMM)
with a logarithmic multiplier update.
To separate the two marginal constraints, define
\begin{equation*}
\cX\coleq\{X\geq0\mid Xe_n=a\},\qquad
\cY\coleq\{Y\geq0\mid Y^\top e_m=b\}.
\end{equation*}
Introducing two primal variables $X,Y$ and imposing $X-Y=0$, we
rewrite problem~\eqref{eq:gw} equivalently as
\begin{equation}\label{eq:ot-split}
\min\{\inprod CX\mid X\in\cX,\ Y\in\cY,\ X-Y=0\}.
\end{equation}
Let $H(P)\coleq\sum_{i,j}P_{ij}(\log P_{ij}-1)$ denote the negative entropy,
with the convention $0\log0=0$.
For $P\geq0$ and $W>0$, the Bregman distance induced by $H$ is defined by
\begin{equation*}
D_H(P;W)\coleq H(P)-H(W)-\inprod{\nabla H(W)}{P-W}
=\sum_{i,j}\left(P_{ij}\log\frac{P_{ij}}{W_{ij}}-P_{ij}+W_{ij}\right),
\end{equation*}
where $\nabla H(W)=\log W$. We use the same notation $D_H$ for the
corresponding Bregman distance between vectors.
Let $Z\in\R^{m\times n}$ be the dual variable associated with $X-Y=0$.
We use a Bregman ADMM (BADMM) with logarithmic dual update:
each primal update minimizes the Lagrangian over one marginal constraint
set, with an entropy Bregman penalty relative to the other primal variable.
Given $Y^0>0$ with $(Y^0)^\top e_m=b$ and $Z^0\in\R^{m\times n}$,
fix a penalty $\rho>0$ and multiplier stepsize $s>0$.
For $k\geq0$, compute
\begin{equation}\label{eq:ot-bregman-updates}
\begin{aligned}
X^{k+1}&=\argmin\{\inprod{C+Z^k}X+\rho D_H(X;Y^k)\mid X\in\cX\},\\
Y^{k+1}&=\argmin\{-\inprod{Z^k}Y+\rho D_H(Y;X^{k+1})\mid Y\in\cY\},\\
Z^{k+1}&=Z^k+s\rho(\log X^{k+1}-\log Y^{k+1}).
\end{aligned}
\end{equation}
Since $\nabla H(W)=\log W$, the multiplier correction is the difference
of the entropy gradients at $X^{k+1}$ and $Y^{k+1}$.
The two primal subproblems have explicit solutions obtained by row
and column normalization, respectively, so no inner iterations are needed.
Section~\ref{sec:ot-algorithm} gives the full construction and uses
a cyclic reindexing of these updates for the analysis.

The simple updates permit efficient computation in the following two ways.
\begin{enumerate}[leftmargin=*,label=(\roman*)]
\item When we choose the initial multiplier to be $Z^0=0$, the iteration
admits a three-vector implementation. Each sweep requires two
matrix--vector products and entrywise operations, without forming the
subsequent primal or multiplier matrices.
The primal and multiplier iterates are therefore represented by
only $\mathcal{O}(m+n)$ vector entries. The matrix-vector products and entrywise operations
are amenable to parallel computation on CPUs and GPUs.
Section~\ref{sec:ot-fast-computation} gives the complete vector recurrence
and its relation to Algorithm~\ref{alg:ot}.

\item For a product initial point and separable costs on product grids,
the matrix-vector products can be evaluated by successive products along
each grid coordinate. On two $L\times L$ grids, this reduces the cost of
each matrix-vector product from $\mathcal{O}(L^4)$ to $\mathcal{O}(L^3)$ and OT matrix storage from
$\mathcal{O}(L^4)$ to $\mathcal{O}(L^2)$. Thus, the same iteration can handle larger grid
problems without storing the full matrix.
Section~\ref{sec:ot-structured-products} develops this implementation
using the separable cost structure~\cite{peyre2019computational}.
\end{enumerate}

Related variants of this Bregman ADMM framework have been studied.
For example, Wang and Banerjee~\cite{wang2014bregman} apply Bregman ADMM
to OT with a linear multiplier update, while Khoo, Tang and
Toh~\cite{khoo2025bregman} use the same logarithmic multiplier update
for the Bethe variational problem. The Bethe analysis relies on uniform
positive lower bounds on the iterates, whereas our OT limits may lie on
the boundary. For $s=1$, $Z^0=0$ and matching starting points and
iteration indices, our iteration admits an equivalent entropy Bregman
Douglas--Rachford splitting (BDRS) formulation~\cite[Theorem~3.3 and Eq.~(5.12)]{ma2025bdrs}.
The convergence analysis in~\cite{ma2025bdrs} does not cover OT.
Ma~\cite{ma2026scaling} subsequently established convergence for matrix
scaling under a projective contraction condition, while counterexamples
in~\cite{ma2026failure} rule out unconditional convergence.
These results motivate identifying conditions that guarantee convergence
for general OT costs.

Establishing such convergence guarantees requires controlling the marginal
residuals and the logarithmic multiplier updates. This is particularly
delicate when the iterates approach the boundary, where the entropy
gradient is singular.
We provide convergence guarantees for general OT data:
from $X^0\in\operatorname{ri}\Pi$ and $Z^0=0$, the primal--dual sequence
has finite length and converges $R$-linearly for any fixed $s\in(0,2)$ and
a sufficiently large fixed $\rho$ depending on $(C,a,b,X^0,s)$.
We further identify the primal limit as the entropy Bregman projection
of the initial point $X^0$ onto the optimal solution set and prove strict
complementarity of the OT primal--dual solution determined by the limit.

\subsection{Contributions}
The main contributions of this paper are summarized as follows.
\begin{enumerate}[leftmargin=*]
\item We develop an entropy Bregman ADMM with logarithmic dual
update and a multiplier stepsize $s\in(0,2)$ for discrete OT problems.
The method admits efficient CPU and GPU implementations based on simple
row and column normalizations. For separable costs on product grids, we
further exploit the separable structure to accelerate the matrix--vector
products and reduce storage requirements.

\item We establish finite length and convergence of the
primal--dual sequence for general OT data. Starting from a strictly
positive feasible primal point and a zero multiplier, the result holds
for any fixed $s\in(0,2)$ and a sufficiently large fixed penalty depending
on the data, the initial point and $s$. Moreover, we characterize the
primal limit as the entropy Bregman projection of the initial point onto
the optimal solution set.
Thus, BADMM converges to the same primal
solution as the exact entropy Bregman proximal point algorithm started
from the same point, even when optimal solutions lie on the boundary.

\item We further prove that the OT primal--dual solution determined
by the limit is strictly complementary. Under the same conditions, we
establish $R$-linear convergence of the primal--dual sequence without
requiring a unique OT optimal solution.

\item We conduct extensive numerical experiments and demonstrate
competitive CPU and GPU performance of our BADMM for general cost matrices. For squared Euclidean costs on
two-dimensional grids, our GPU implementation outperforms the
state-of-the-art solvers HALO and HELLO. It solves the largest
$2048\times2048$ DOMDEC problem in under 100 seconds, about
24 times faster than HALO~\cite{xia2026halo}, and is approximately
84 times faster on average than HELLO~\cite{xia2026hello} on the
$512\times512$ DOTmark problems.
\end{enumerate}

As an additional side contribution, in Appendix~\ref{sec:ot-support-identification}, we derive a computable primal--dual gap test that certifies
variables vanishing at every optimal solution.
Under the preceding convergence conditions, the test identifies all
such variables in finitely many iterations, even though the iterates
themselves remain strictly positive. Fixing the certified variables
preserves the entire optimal solution set and permits subsequent
reduced computation on the retained variables.

\subsection{Organization}
The rest of the paper is organized as follows.
Section~\ref{sec:ot-badmm} introduces entropy projections and Bregman
proximal steps, derives the algorithm, and establishes the iteration
properties used in the analysis.
Section~\ref{sec:ot-convergence} proves convergence and the linear rate.
Section~\ref{sec:ot-implementation} presents implementation techniques.
Numerical experiments are reported in Section~\ref{sec:ot-numerical},
followed by concluding remarks in Section~\ref{sec:ot-conclusion}.
Appendix~\ref{sec:ot-support-identification} gives the support-identification
rule and its finite-identification guarantee, and
Appendix~\ref{app:ot-computational-details} reports detailed computational results.

\subsection{Notation}
We use the following notation throughout the paper.
\begin{itemize}[leftmargin=0pt,itemsep=5pt,topsep=5pt]
\item[] \tb{Matrices and elementwise operations}:
For a matrix $U$, the inequalities $U\geq0$ and $U>0$ are understood
entrywise, and $U_{\cdot j}$ denotes column $j$.
For arrays of the same size, $(U\odot V)_{ij}\coleq U_{ij}V_{ij}$ and
$(U\oslash V)_{ij}\coleq U_{ij}/V_{ij}$, with nonzero denominators.
Entrywise quotients may also be written as $U/V$.
Logarithms, exponentials and powers are entrywise, so that
$(\exp U)_{ij}=e^{U_{ij}}$ and $(U^{\odot p})_{ij}\coleq U_{ij}^{p}$ for $U>0$.
The same conventions apply to vectors.
For a vector $v$, $\Diag(v)$ is the diagonal matrix with diagonal $v$.

\item[] \tb{Inner products and norms}:
On $\R^{m\times n}$, the Frobenius inner product is
$\inprod UV\coleq\operatorname{tr}(U^\top V)$, with induced norm
$\|\cdot\|_F$.
An unsubscripted norm is Euclidean for vectors and Frobenius for matrices.
The norms $\|U\|_1$ and $\|U\|_\infty$ denote the sum and maximum of the
absolute entry values, respectively.
For $x,y\in\R^m$, we write $\langle x,y\rangle_a\coleq x^\top\Diag(a)y$
and $\|x\|_a\coleq\sqrt{\langle x,x\rangle_a}$.
The weighted inner product $\langle\cdot,\cdot\rangle_b$ and norm
$\|\cdot\|_b$ are defined analogously on $\R^n$.

\item[] \tb{Sets and supports}:
The sets $\R_+^n$ and $\R_{++}^n$ consist of vectors in $\R^n$ with
nonnegative and positive entries, respectively.
For a set $S$, $\operatorname{ri}S$ denotes its relative interior.
For $P\geq0$, $\operatorname{supp}P\coleq\{(i,j)\mid P_{ij}>0\}$.
For an index set $S$, $S^c$ is its complement in the full matrix index
set, and $U|_S$ denotes the entries of $U$ indexed by $S$.

\item[] \tb{Marginal normalizations}:
For $W\in\R^{m\times n}_{++}$ with positive row and column sums, define
\begin{eqnarray*}
\begin{array}{lllll}
\mathcal R_a(W)&\coleq&\Diag\big(a\oslash(We_n)\big)W
&=&\mbox{argmin}\Big\{ D_H(Q;W) \mid Q e_n = a, Q\geq 0 \Big\}
\\[5pt]
\mathcal C_b(W) &\coleq& W\Diag\big(b\oslash(W^\top e_m)\big)
&=&\mbox{argmin}\Big\{ D_H(Q;W) \mid Q^\top e_m = b, Q\geq 0  \Big\}.
\end{array}
\end{eqnarray*}
These are the entropy projections onto $\cX$ and $\cY$, respectively.
\end{itemize}

\section{Bregman ADMM}\label{sec:ot-badmm}
We first recall the optimality conditions and entropy projection
properties for OT. We then introduce Bregman ADMM and establish the
identities used in its convergence analysis.

\subsection{Optimality conditions and optimal solutions}\label{sec:ot-optimality}
We first state the KKT conditions for the OT problem~\eqref{eq:gw}.
Let $P$ denote the primal variable, and let $\alpha\in\R^m$ and
$\beta\in\R^n$ denote the dual variables
associated with $Pe_n=a$ and $P^\top e_m=b$, respectively.
\begin{definition}[KKT solution and strict complementarity]
A primal--dual point $(P,\alpha,\beta)\in\Pi\times\R^m\times\R^n$ is a KKT solution
of~\eqref{eq:gw} if its dual slack satisfies
\begin{equation}\label{eq:gw-certificate}
\Gamma\coleq C-\alpha e_n^\top-e_m\beta^\top\geq0,\qquad P\odot\Gamma=0.
\end{equation}
A KKT solution is strictly complementary if $\Gamma_{ij}>0$ outside
$\operatorname{supp}P$.
\end{definition}

Let $\mathcal O$ denote the optimal solution set and $S$ the union of
the supports of all optimal solutions:
\begin{equation*}
\mathcal O\coleq\argmin\{\langle C,Q\rangle\mid Q\in\Pi\},\qquad
S\coleq\bigcup_{Q\in\mathcal O}\operatorname{supp}Q.
\end{equation*}
The following characterization follows from the Goldman--Tucker strict
complementarity theorem~\cite{goldman1956theory} and the standard
characterization of the relative interior of a polyhedron.
\begin{proposition}[Optimal solutions and strict complementarity]\label{prop:ot-optimal-support}
We have
\begin{equation*}
\mathcal O=\{Q\in\Pi:Q|_{S^c}=0\},\qquad
\operatorname{ri}\mathcal O=\{Q\in\mathcal O:Q_{ij}>0\text{ on }S\}.
\end{equation*}
There is a dual optimal solution $(\alpha_*,\beta_*)$ whose slack
$\Gamma$ in~\eqref{eq:gw-certificate} satisfies
\begin{equation}\label{eq:ot-strict-certificate}
\Gamma_{ij}=0\quad\Longleftrightarrow\quad(i,j)\in S.
\end{equation}
\end{proposition}

We use the connected components of the bipartite support graph
to group the row and column indices in the convergence analysis.
\begin{definition}[Bipartite support graph]\label{def:ot-support-graph}
For $P\in\Pi$, its bipartite support graph has one vertex for each row
and one for each column, with an edge between row $i$ and column $j$
exactly when $P_{ij}>0$. Row and column vertices are distinct even if
their indices coincide. A connected component is a maximal set of
vertices in which any two vertices are joined by a path of edges.
Let $I_d$ and $J_d$ denote the row and column index sets of its $d$th
connected component.
\end{definition}
Since $a,b>0$, no vertex is isolated, so each connected component
contains both row and column vertices. Thus $\{I_d\}$ and $\{J_d\}$
partition the respective index sets, and $P_{I_d,J_e}=0$ for $d\ne e$.

\begin{lemma}[Equality of row and column marginal sums]\label{lem:ot-marginal-sums}
Let $P\in\Pi(a,b)$, $I\subseteq\{1,\ldots,m\}$, and
$J\subseteq\{1,\ldots,n\}$, with complements taken in the respective
index sets. If $P_{ij}=0$ whenever
$(i,j)\in(I\times J^c)\cup(I^c\times J)$, then
\begin{equation*}
\sum_{i\in I}a_i=\sum_{j\in J}b_j.
\end{equation*}
\end{lemma}
\begin{proof}
By primal feasibility and the assumed zero entries,
$
\sum_{i\in I}a_i=\sum_{i\in I,j\in J}P_{ij}=\sum_{j\in J}b_j.
$
This completes the proof.
\end{proof}
In particular, $\sum_{i\in I_d}a_i=\sum_{j\in J_d}b_j$ for every
connected component of the bipartite support graph of $P$.

\subsection{Entropy Bregman projections}\label{sec:ot-projections}
For $W>0$, $a\in\R^m_{++}$ and $b\in \R^n_{++}$ satisfying $e_m^\top a=e_n^\top b$, define
\begin{equation}\label{eq:ot-projection-map}
\mathcal P(W;a,b)\coleq\argmin\{D_H(Q;W)\mid
 Q\geq0,\ Qe_n=a,\ Q^\top e_m=b\}.
\end{equation}
\begin{proposition}[Entropy projection]\label{prop:ot-projection-properties}
The mapping in~\eqref{eq:ot-projection-map} has the following properties.
\begin{enumerate}[label=(\roman*),leftmargin=*,beginpenalty=10000]
\item The minimizer is unique and positive. A feasible matrix $Q>0$
equals $\mathcal P(W;a,b)$ if and only if, for some $\alpha\in\R^m$
and $\beta\in\R^n$,
\begin{equation}\label{eq:ot-projection-scaling}
Q=\Diag(e^\alpha)W\Diag(e^\beta).
\end{equation}
\item For positive vectors $u,v$ and any $t\in\R$,
\begin{eqnarray}
\mathcal P(W;a,b) &=&\mathcal P(\Diag(u)W\Diag(v);a,b),
\label{eq:ot-projection-invariance}
\\[5pt]
\mathcal P(W\odot e^{-tC};a,b) &=& 
\argmin\{t\langle C,Q\rangle+D_H(Q;W)\mid Q\in\Pi(a,b)\}.
\label{eq:ot-projection-cost}
\end{eqnarray}
\item For $Q\coleq\mathcal P(W;a,b)$ and every $V\in\Pi(a,b)$,
\begin{equation}\label{eq:ot-projection-pythagoras}
D_H(V;W)=D_H(V;Q)+D_H(Q;W).
\end{equation}
\item The mapping $(W,a,b)\mapsto\mathcal P(W;a,b)$ is $C^\infty$
on $\{(W,a,b):W>0,\ a>0,\ b>0,\ e_m^\top a=e_n^\top b\}$.
For every compact subset $K$ of this set, there exists $L_K>0$ such that
\begin{equation}\label{eq:ot-projection-lipschitz}
\|\mathcal P(\widehat W;\widehat a,\widehat b)-\mathcal P(W;a,b)\|_F
\leq L_K\big(\|\widehat W-W\|_F+\|\widehat a-a\|+\|\widehat b-b\|\big)
\end{equation}
for all $(W,a,b),(\widehat W,\widehat a,\widehat b)\in K$.
\end{enumerate}
\end{proposition}
\begin{proof}
Parts~\textup{(i)--(iii)} follow from the linear-family projection
properties in~\cite[Theorems~3.1--3.2]{csiszar2004information}, with
the cost absorbed into $W$ in~\eqref{eq:ot-projection-cost}.
Part~\textup{(iv)} follows from the implicit-function argument
in~\cite[Theorem~2 and Supplement, Section~C]{luise2018differential} applied to the cost
$-\log W$; see also~\cite[Theorem~3.3]{pauwels2023derivatives}
for joint dependence on the cost and marginals.
\end{proof}

We next define the optimal solution selected by the initial
matrix and a feasible path approaching this solution.
For a fixed $X^0\in\operatorname{ri}\Pi$, define
\begin{equation}\label{eq:ot-selected}
P_*\coleq\argmin\{D_H(Q;X^0)\mid Q\in\mathcal O\}.
\end{equation}
The associated entropy path is
\begin{equation}\label{eq:ot-entropy-path}
P(t)\coleq\mathcal P(X^0\odot e^{-tC};a,b),\qquad t\geq0.
\end{equation}
For $t>0$, $P(t)$ minimizes the transport cost with an entropy
penalty of weight $1/t$ centered at $X^0$.
\begin{proposition}[Entropy path]\label{prop:ot-entropy-path}
The problem in~\eqref{eq:ot-selected} has a unique solution $P_*$
with $\operatorname{supp}P_*=S$.
The path $P(t)$ in~\eqref{eq:ot-entropy-path} is positive and smooth
for $t\geq0$. Moreover,
\begin{equation}\label{eq:ot-path-limit}
P(t)\to P_*,\qquad D_H(P_*;P(t))\to0\qquad(t\to\infty).
\end{equation}
\end{proposition}
\begin{proof}
The support characterization in Proposition~\ref{prop:ot-optimal-support}
and~\cite[Theorems~3.1--3.2]{csiszar2004information} give existence,
uniqueness and $\operatorname{supp}P_*=S$.
Applying Proposition~\ref{prop:ot-projection-properties}(i), (iv)
to the positive matrix $X^0\odot e^{-tC}$ gives positivity and
smoothness of $P(t)$, and hence bounded derivatives on every finite
interval. The limits in~\eqref{eq:ot-path-limit} follow
from~\cite[Proposition~3.2, Theorem~4.2 and Corollary~4.1]{alvarez2004hessian}.
\end{proof}

\subsection{Bregman ADMM}\label{sec:ot-algorithm}
For the original OT problem, the entropy Bregman proximal point
algorithm (PPA)~\cite{teboulle1992entropic,yang2022bregman} adds to the linear objective
a proximal term centered at the current feasible iterate.
Starting from $Q^0\in\operatorname{ri}\Pi$
and fixing $\rho>0$, it generates $\{Q^k\}$ by
\begin{equation}\label{eq:ot-bregman-proximal-step}
Q^{k+1}=\argmin\{\langle C,Q\rangle+\rho D_H(Q;Q^k)\mid Q\in\Pi\}
=\mathcal P(Q^k\odot e^{-C/\rho};a,b),
\end{equation}
for $k\geq0$, where the second equality follows
from~\eqref{eq:ot-projection-cost}. Proposition~\ref{prop:ot-projection-properties}(i)
ensures that each $Q^{k+1}$ is positive and feasible.
The center $Q^k$ changes at every step, so the proximal term is not a
fixed entropy regularizer added to the target OT problem.
Each proximal step requires a complete entropy projection satisfying
both marginal constraints, which generally requires an inner iteration
of row and column normalizations. To avoid solving this joint
subproblem at every iteration, we separate the two sets of marginal
constraints, using $X$ for the row constraints and $Y$ for the column
constraints, and introduce a multiplier $Z$ for $X=Y$.
The resulting equivalent formulation of the original OT problem is
\begin{equation*}
\min\{\langle C,X\rangle\mid X\in\cX,\ Y\in\cY,\ X-Y=0\}.
\end{equation*}
Here $X,Y$ are the primal variables, and $Z\in\R^{m\times n}$ is the dual
variable associated with $X-Y=0$.
We use entropy Bregman penalties in the primal subproblems
\cite{wang2014bregman} and the nonlinear dual update
of~\cite[Algorithm~3]{khoo2025bregman}, with a relaxation parameter $s>0$.
To simplify the indices in the analysis, Algorithm~\ref{alg:ot}
uses the update order $Y,Z,X$.
\begin{algorithm}[H]
\caption{Bregman ADMM for optimal transport}\label{alg:ot}
\textbf{Input:} Choose $\rho>0$, $s>0$, $X^0>0$ with
$X^0e_n=a$, and $Z^0\in\R^{m\times n}$. Set $k=0$.
\par\noindent\textbf{While} a termination criterion is not met, compute
\begin{subequations}\label{eq:ot-iteration}
\begin{align}
 Y^{k+1}&=\mathcal C_b\!\left(
       X^k\odot\exp\!\left(\frac{Z^k}{\rho}\right)\right),
       \\
 Z^{k+1}&=Z^k+s\rho\big(\log X^k-\log Y^{k+1}\big),
       \\
 X^{k+1}&=\mathcal R_a\!\left(
       Y^{k+1}\odot\exp\!\left(-\frac{C+Z^{k+1}}{\rho}\right)\right).
\end{align}
\end{subequations}
Set $k\leftarrow k+1$ and continue.
\end{algorithm}

After the first two updates, Algorithm~\ref{alg:ot} starts from
$(Y^1,Z^1)$ and computes $X^k,Y^{k+1},Z^{k+1}$ for $k\geq1$.
This gives the update order in~\eqref{eq:ot-bregman-updates}.
If $X^0\in\Pi$ and $Z^0=0$, then $Y^1=X^0$ and $Z^1=0$.

\begin{remark}[Relation to ADEMM and BDRS]
For $s=1$, $Z^0=0$ and matching starting points and iteration indices,
Algorithm~\ref{alg:ot} has the same scaling updates as the alternating
direction exponential multiplier method (ADEMM) and
an equivalent entropy Bregman Douglas--Rachford splitting (BDRS) formulation
\cite[Theorem~3.3 and Eq.~(5.12)]{ma2025bdrs}.
\end{remark}

\subsection{Properties of the iterates}\label{sec:ot-properties}
We establish the identities and projection properties of
Algorithm~\ref{alg:ot} needed for the convergence analysis.
\begin{proposition}[Basic identities and equivalent updates]\label{prop:ot-equivalent-updates}
Let $\{(X^k,Y^k,Z^k)\}$ be generated by Algorithm~\ref{alg:ot}.
\begin{enumerate}[label=(\roman*),leftmargin=*]
\item The primal updates are uniquely defined and, for every $k\geq1$,
\begin{equation*}
X^k>0,\qquad Y^k>0,\qquad X^ke_n=a,\qquad (Y^k)^\top e_m=b.
\end{equation*}
There exist dual variables $r^k\in\R^m$ for the row constraints in
the $X^k$ subproblems ($k\geq1$) and $c^k\in\R^n$ for the column
constraints in the $Y^{k+1}$ subproblems ($k\geq0$) such that
\begin{subequations}\label{eq:sub-kkt}
\begin{align}
C+Z^k+\rho(\log X^k-\log Y^k)+r^ke_n^\top&=0,\quad k\geq1,\\
-Z^k+\rho(\log Y^{k+1}-\log X^k)+e_m(c^k)^\top&=0,\quad k\geq0.
\end{align}
\end{subequations}
\item If $Z^0=0$, all entries in each column of $Z^k$ are equal. Thus there is a vector
$z^k\in\R^n$, with $z^0=0$, such that
$Z^k=-\rho e_m(z^k)^\top$.
The column update then reduces to $Y^{k+1}=\mathcal C_b(X^k)$, and
\begin{equation}\label{eq:ot-vector-scaling}
z^{k+1}=z^k-s\log\big(((X^k)^\top e_m)\oslash b\big).
\end{equation}
The row update $X^{k+1}$ is obtained by substituting
$Z^{k+1}=-\rho e_m(z^{k+1})^\top$ in Algorithm~\ref{alg:ot}.
\item If $Z^0=0$, then for $k\geq1$, eliminating the multiplier gives the recurrence
\begin{equation}\label{eq:ot-two-primal-recurrence}
Y^{k+1}=\mathcal C_b(X^k),\qquad
X^{k+1}=\mathcal R_a\!\left[
\big((X^k)^{\odot2}\oslash Y^k\big)
\Diag\!\left(\big(b\oslash((X^k)^\top e_m)\big)^{1+s}\right)
\right].
\end{equation}
\end{enumerate}
\end{proposition}
\begin{proof}
For~\textup{(i)}, the normalizations in~\eqref{eq:ot-iteration}
give positive matrices with the prescribed marginals.
Strict convexity gives uniqueness. Taking logarithms of the two
primal normalization formulas gives~\eqref{eq:sub-kkt}, with the
row and column normalization factors determining $r^k$ and $c^k$.

For~\textup{(ii)}, suppose that $Z^k=-\rho e_m(z^k)^\top$.
The column factors $e^{-z^k}$ cancel in the normalization defining
$Y^{k+1}$. Hence $Y^{k+1}=\mathcal C_b(X^k)$ and
\begin{equation*}
\log X^k-\log Y^{k+1}
=e_m\big[\log(((X^k)^\top e_m)\oslash b)\big]^\top.
\end{equation*}
The multiplier update gives the formula for $z^{k+1}$ in~\eqref{eq:ot-vector-scaling}.
Induction from $Z^0=0$ proves the stated column structure for $Z^{k+1}$ and the
vector recurrence for $z^{k+1}$. Substituting this representation of $Z^{k+1}$
in~\eqref{eq:ot-iteration} gives the row update for $X^{k+1}$.

For~\textup{(iii)}, the row optimality condition in~\eqref{eq:sub-kkt}
gives
$e^{-C_{ij}/\rho+z_j^k}=e^{r_i^k/\rho}X^k_{ij}/Y^k_{ij}$.
Part~\textup{(ii)} gives
$Y^{k+1}_{ij}=X^k_{ij}b_j/((X^k)^\top e_m)_j$ and
$e^{z_j^{k+1}-z_j^k}=(b_j/((X^k)^\top e_m)_j)^s$.
Substituting both expressions into the input of the next row
normalization yields
\begin{equation*}
Y^{k+1}_{ij}e^{-C_{ij}/\rho+z_j^{k+1}}
=e^{r_i^k/\rho}\frac{(X^k_{ij})^2}{Y^k_{ij}}
 \left(\frac{b_j}{((X^k)^\top e_m)_j}\right)^{1+s}.
\end{equation*}
The factor $e^{r_i^k/\rho}$ is constant within row $i$ and cancels
under $\mathcal R_a$. This proves~\eqref{eq:ot-two-primal-recurrence}.
\end{proof}

The recurrence~\eqref{eq:ot-two-primal-recurrence} maps $(Y^k,X^k)$
to $(Y^{k+1},X^{k+1})$ without the multiplier. We use this relation
to estimate the marginal residuals in Section~\ref{sec:ot-convergence}.
For $k\geq1$, define the relative marginal residuals
\begin{equation}\label{eq:ot-marginal-residuals}
r_Y^k\coleq(Y^ke_n)\oslash a-e_m,\qquad
c_X^k\coleq((X^k)^\top e_m)\oslash b-e_n.
\end{equation}
They measure the violation of the row constraints by $Y^k$
and of the column constraints by $X^k$. Since the entries of $X^k$ and $Y^k$
each sum to one, we have
$a^\top r_Y^k=b^\top c_X^k=0$ for all $k$. 

Although $Y^k$ need not satisfy the row constraints, its
feasible projections satisfy the exact entropy PPA iteration, starting
from $\mathcal P(Y^1;a,b)$, as shown in the next proposition. These projections are used only in the
analysis and need not be computed by Algorithm~\ref{alg:ot}.
\begin{proposition}[Entropy projections of the iterates]\label{prop:ot-iterate-projections}
Let Algorithm~\ref{alg:ot} use a fixed $\rho>0$. For every $k\geq1$,
the following properties hold.
\begin{enumerate}[label=(\roman*),leftmargin=*]
\item The feasible projections satisfy the exact proximal iteration
\begin{equation}\label{eq:ot-exact-ppa}
\mathcal P(Y^{k+1};a,b)
=\argmin\{\langle C,Q\rangle
+\rho D_H(Q;\mathcal P(Y^k;a,b))\mid Q\in\Pi\}
\end{equation}
and the path representation
\begin{equation}\label{eq:ot-path-projection}
\mathcal P(Y^{k+1};a,b)
=\mathcal P\big(X^0\odot\exp((Z^0-kC)/\rho);a,b\big).
\end{equation}
\item If $Z^0=0$, then
\begin{equation}\label{eq:ot-shifted-projections}
Y^{k+1}=\mathcal P(X^0\odot e^{-kC/\rho};Y^{k+1}e_n,b),\qquad
X^k=\mathcal P(X^0\odot e^{-kC/\rho};a,(X^k)^\top e_m).
\end{equation}
\end{enumerate}
\end{proposition}
\begin{proof}
(i) We first compare two successive feasible projections.
Adding the row optimality condition at $X^k$ and the column
optimality condition at $Y^{k+1}$ in~\eqref{eq:sub-kkt} cancels $Z^k$
and $\log X^k$, giving
\begin{equation}\label{eq:ot-log-telescope}
\rho(\log Y^{k+1}-\log Y^k)+C+r^ke_n^\top+e_m(c^k)^\top=0.
\end{equation}
Hence $Y^{k+1}$ is a positive row and column diagonal scalings of $Y^k\odot e^{-C/\rho}$.
By Proposition~\ref{prop:ot-projection-properties}(i), $\mathcal P(Y^k;a,b)$
is also a positive diagonal scaling of $Y^k$.
Applying~\eqref{eq:ot-projection-invariance} to these two scalings gives
\begin{equation*}
\mathcal P(Y^{k+1};a,b)
=\mathcal P(Y^k\odot e^{-C/\rho};a,b)
=\mathcal P(\mathcal P(Y^k;a,b)\odot e^{-C/\rho};a,b).
\end{equation*}
Applying~\eqref{eq:ot-projection-cost} with
$t=1/\rho$ and $W=\mathcal P(Y^k;a,b)$ proves~\eqref{eq:ot-exact-ppa}.

To identify the projection from the initial point, sum
\eqref{eq:ot-log-telescope} over $t=1,\ldots,k$:
\begin{equation}\label{eq:ot-scaling-invariant}
\log Y^{k+1}-\log Y^1
=-\frac{kC}{\rho}-\frac1\rho\sum_{t=1}^{k}
\big(r^te_n^\top+e_m(c^t)^\top\big).
\end{equation}
The initialization gives $Y^1=\mathcal C_b(X^0\odot e^{Z^0/\rho})$.
Thus $Y^{k+1}$ is a positive diagonal scaling of
$X^0\odot e^{(Z^0-kC)/\rho}$. Scaling invariance then
proves~\eqref{eq:ot-path-projection}.

(ii) Suppose now that $Z^0=0$.
The same representation at index $k+1$ and the column marginal $b$
show that $Y^{k+1}$
is a feasible positive diagonal scaling for the first projection
in~\eqref{eq:ot-shifted-projections}.
Furthermore, Proposition~\ref{prop:ot-equivalent-updates}(ii) and
the row optimality condition give
\begin{equation*}
X^k=\Diag(e^{-r^k/\rho})(Y^k\odot e^{-C/\rho})\Diag(e^{z^k}).
\end{equation*}
Consequently $X^k$ is a positive diagonal scaling of
$X^0\odot e^{-kC/\rho}$ with row sums $a$ and column sums
$(X^k)^\top e_m$.
The characterization in Proposition~\ref{prop:ot-projection-properties}(i),
applied with these respective marginal pairs, proves both identities
in~\eqref{eq:ot-shifted-projections}.
\end{proof}
Convergence of the projected sequence alone does not imply
convergence of the actual iterates. In Section~\ref{sec:ot-convergence}, we estimate
bounds for their marginal residuals and prove convergence.

\section{Convergence analysis}\label{sec:ot-convergence}
We establish convergence of the primal--dual sequence
generated by Algorithm~\ref{alg:ot}.

\begin{theorem}[Global convergence]\label{thm:ot-global}
Fix $s\in(0,2)$ and $X^0\in\operatorname{ri}\Pi$.
There exists $\rho_0>0$, depending on $(C,a,b,X^0,s)$, such that
Algorithm~\ref{alg:ot} with $Z^0=0$ and any fixed $\rho\geq\rho_0$
generates a sequence $\{(X^k,Y^k,Z^k)\}$ with the following properties.
\begin{enumerate}[label=(\roman*),ref=(\roman*),leftmargin=*]
\item (Finite length.) The sequence has finite length:
\begin{equation*}
\sum_{k=1}^{\infty}\big(\|X^{k+1}-X^k\|_F+\|Y^{k+1}-Y^k\|_F
 +\|Z^{k+1}-Z^k\|_F\big)<\infty.
\end{equation*}
\item (KKT convergence.) The sequence converges to a KKT solution
$(P_*,P_*,Z_\infty)$ of~\eqref{eq:ot-split},
whose associated OT solution is strictly complementary,
and $P_*$ is the Bregman projection of $X^0$ onto $\mathcal O$
in~\eqref{eq:ot-selected}.
\item (Linear convergence.) The convergence is $R$-linear: for some
$M>0$ and $q\in(0,1)$, and sufficiently large $k$, we have
\begin{equation*}
\|X^k-P_*\|_F+\|Y^k-P_*\|_F+\|Z^k-Z_\infty\|_F\leq Mq^k.
\end{equation*}
\end{enumerate}
\end{theorem}
When $X^0=ab^\top$, the selected solution has maximum
entropy over $\mathcal O$.

\begin{remark}[Comparison with BDRS counterexamples]\label{rem:ot-bdrs-counterexamples}
For $s=1$ and a prescribed infeasible initialization,
Ma~\cite[Theorem~5.1]{ma2026failure} constructs
an OT counterexample whose BDRS iterates fail to converge for every
fixed $\rho>0$.
This does not contradict our convergence result, because
Theorem~\ref{thm:ot-global} additionally requires
$X^0\in\operatorname{ri}\Pi$, $Z^0=0$, and sufficiently large fixed $\rho$.
\end{remark}

For an illustration of the linear convergence result
in Theorem~\ref{thm:ot-global}, we consider $a$ and $b$ obtained by aggregating the $32\times32$ DOTmark CauchyDensity pair
to an $8\times8$ pair by block summation, with a squared Euclidean cost matrix
(see Section~\ref{sec:ot-numerical} for the data and cost construction).
With $X^0=ab^\top$, $Z^0=0$, $s=1$, and fixed $\rho\in\{1,2,5,10\}$,
Figure~\ref{fig:ot-linear-convergence} plots the errors against the iteration
number. Their approximately geometric decay is consistent with the
$R$-linear convergence established in Theorem~\ref{thm:ot-global}.

\begin{figure}[!htbp]
\centering
\includegraphics[width=0.52\textwidth]{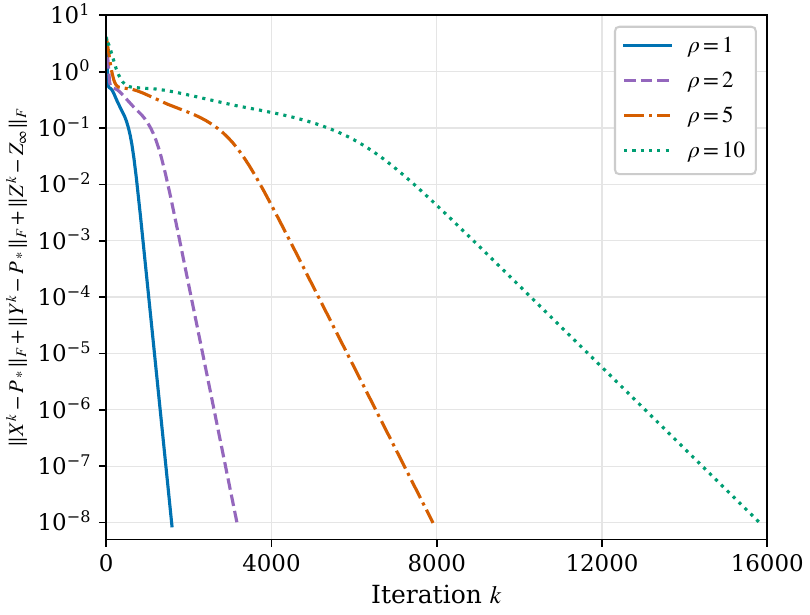}
\caption{Convergence of the primal--dual iterates for fixed penalty parameters.}
\label{fig:ot-linear-convergence}
\end{figure}

Throughout the proof, fix $s\in(0,2)$ and $X^0\in\operatorname{ri}\Pi$,
and consider Algorithm~\ref{alg:ot} with $Z^0=0$.
The penalty $\rho>0$ is fixed throughout each run of Algorithm~\ref{alg:ot}.
Figure~\ref{fig:ot-proof-flowchart} shows the three stages of the proof
and their dependencies.

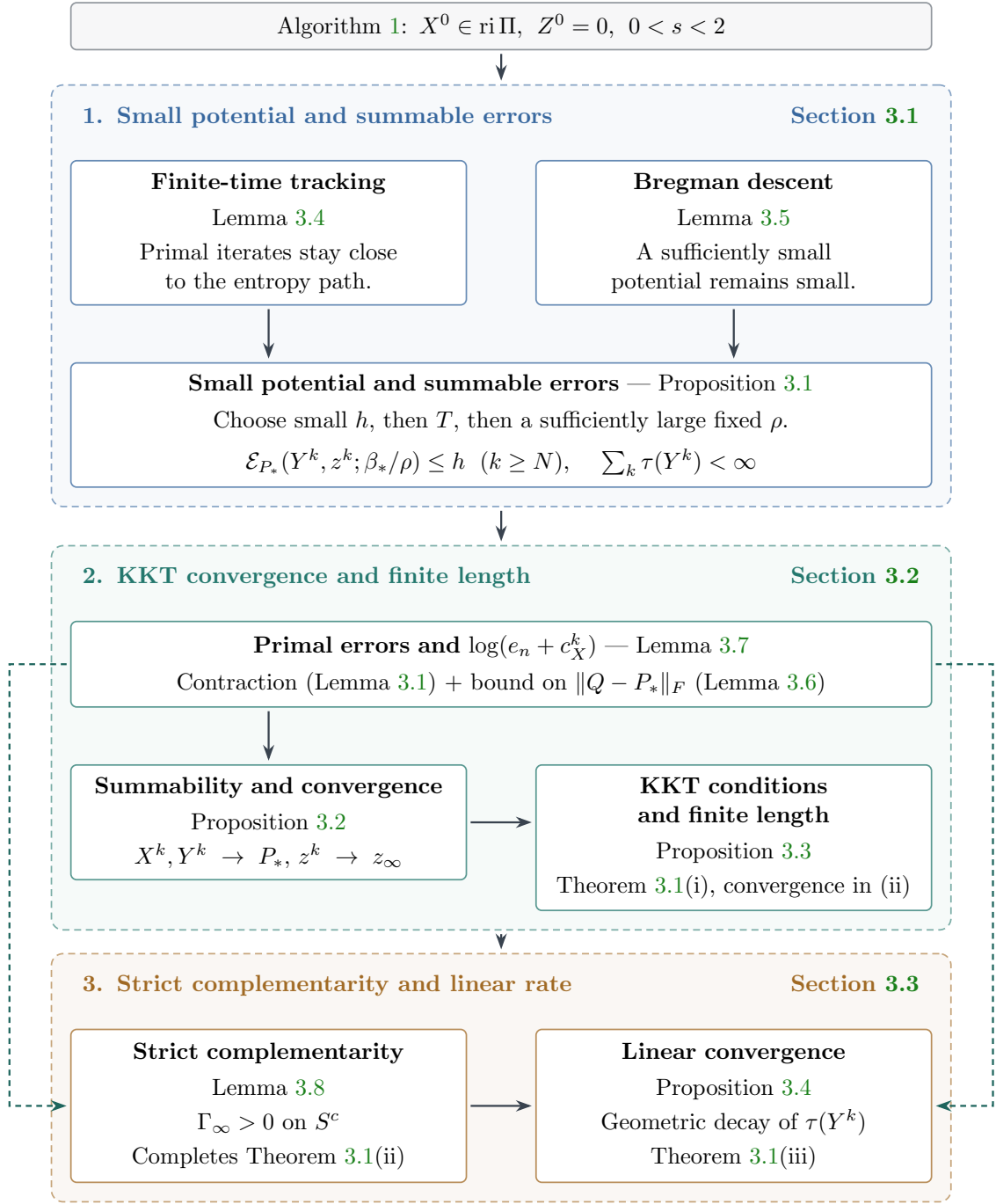
\begin{figure}[!tp]
\centering
\begingroup
\definecolor{proofblue}{RGB}{55,103,158}
\definecolor{proofteal}{RGB}{35,119,109}
\definecolor{proofamber}{RGB}{160,105,35}
\definecolor{proofink}{RGB}{57,67,79}
\begin{tikzpicture}[
  font=\small,
  proofbox/.style={draw,line width=.65pt,rounded corners=2.5pt,
    fill=white,text=black,align=center,text width=12.6cm,
    inner xsep=5pt,inner ysep=5pt},
  proofhalf/.style={proofbox,text width=5.6cm,minimum height=1.65cm},
  grouptitle/.style={font=\small\bfseries,align=left,
    text width=12.6cm,inner xsep=0pt,inner ysep=1pt},
  proofgroup/.style={rounded corners=5pt,line width=.65pt,
    dash pattern=on 3pt off 2pt,inner xsep=8pt,inner ysep=8pt},
  proofarrow/.style={-{Stealth[length=2.4mm,width=1.7mm]},
    draw=proofink,line width=.85pt,shorten <=1.5pt,shorten >=1.5pt},
  reuse/.style={proofarrow,draw=proofteal!85!black,
    dash pattern=on 2.5pt off 2pt}
]
\node[proofbox,draw=proofink!55,fill=proofink!4,minimum height=.7cm]
(start) {Algorithm~\ref{alg:ot}: \mbox{$X^0\in\operatorname{ri}\Pi$},\; \mbox{$Z^0=0$},\; \mbox{$0<s<2$}};

\node[grouptitle,text=proofblue] (title1)
  at ($(start.south)+(0,-1cm)$)
  {1. Small potential and summable errors\hfill Section~\ref{sec:ot-feasibility}};
\node[proofhalf,draw=proofblue!75,anchor=north] (tracking)
  at ($(title1.south)+(-3.5cm,-.45cm)$) {
\textbf{Finite-time tracking}\\[3pt]
Lemma~\ref{lem:ot-path-one-step}\\[2pt]
Primal iterates stay close to the entropy path.
};
\node[proofhalf,draw=proofblue!75,anchor=north] (descent)
  at ($(title1.south)+(3.5cm,-.45cm)$) {
\textbf{Bregman descent}\\[3pt]
Lemma~\ref{prop:ot-entropy-identity}\\[2pt]
A sufficiently small\\potential remains small.
};
\node[proofbox,draw=proofblue!75,minimum height=1.5cm,anchor=north] (capture)
  at ($(tracking.south)+(3.5cm,-.85cm)$) {
\textbf{Small potential and summable errors} --- Proposition~\ref{prop:ot-feasibility-sum}\\[3pt]
Choose small $h$, then $T$, then a sufficiently large fixed $\rho$.\\[5pt]
\mbox{$\mathcal E_{P_*}(Y^k,z^k;\beta_*/\rho)\leq h\;\ (k\geq N)$},
\quad \mbox{$\sum_k\tau(Y^k)<\infty$}
};

\node[grouptitle,text=proofteal] (title2)
  at ($(capture.south)+(0,-1.35cm)$)
  {2. KKT convergence and finite length\hfill Section~\ref{sec:ot-kkt-convergence}};
\node[proofbox,draw=proofteal!75,minimum height=1.25cm,anchor=north] (support)
  at ($(title2.south)+(0,-.45cm)$) {
\textbf{Primal errors and \mbox{$\log(e_n+c_X^k)$}} --- Lemma~\ref{lem:ot-supported-estimates}\\[3pt]
Contraction (Lemma~\ref{lem:ot-linear-contraction})
$+$ bound on \mbox{$\|Q-P_*\|_F$} (Lemma~\ref{lem:ot-support-recovery})
};
\node[proofhalf,draw=proofteal!75,anchor=north] (summable)
  at ($(support.south)+(-3.5cm,-.85cm)$) {
\textbf{Summability and convergence}\\[3pt]
Proposition~\ref{lem:ot-primal-convergence}\\[2pt]
$X^k,Y^k\to P_*$, $z^k\to z_\infty$
};
\node[proofhalf,draw=proofteal!75,anchor=north] (kkt)
  at ($(support.south)+(3.5cm,-.85cm)$) {
\textbf{KKT conditions and finite length}\\[3pt]
Proposition~\ref{lemma:ot-feasibility-kkt}\\[2pt]
Theorem~\ref{thm:ot-global}(i), convergence in (ii)
};

\coordinate (middle2) at ($(summable.south)!.5!(kkt.south)$);
\node[grouptitle,text=proofamber] (title3)
  at ($(middle2)+(0,-1.35cm)$)
  {3. Strict complementarity and linear rate\hfill Section~\ref{sec:ot-strict-linear}};
\node[proofhalf,draw=proofamber!75,minimum height=2.3cm,anchor=north] (strict)
  at ($(title3.south)+(-3.5cm,-.45cm)$) {
\textbf{Strict complementarity}\\[3pt]
Lemma~\ref{lem:ot-selected-strict}\\[2pt]
\mbox{$\Gamma_\infty>0$ on $S^c$}\\[2pt]
Completes Theorem~\ref{thm:ot-global}(ii)
};
\node[proofhalf,draw=proofamber!75,minimum height=2.3cm,anchor=north] (rate)
  at ($(title3.south)+(3.5cm,-.45cm)$) {
\textbf{Linear convergence}\\[3pt]
Proposition~\ref{lem:ot-linear-length}\\[2pt]
Geometric decay of $\tau(Y^k)$\\[2pt]
Theorem~\ref{thm:ot-global}(iii)
};
\begin{scope}[on background layer]
\node[proofgroup,draw=proofblue!60,fill=proofblue!4,
  fit=(title1)(tracking)(descent)(capture)] (group1) {};
\node[proofgroup,draw=proofteal!60,fill=proofteal!4,
  fit=(title2)(support)(summable)(kkt)] (group2) {};
\node[proofgroup,draw=proofamber!60,fill=proofamber!5,
  fit=(title3)(strict)(rate)] (group3) {};
\end{scope}

\draw[proofarrow] (start.south) -- (group1.north);
\draw[proofarrow] (group1.south) -- (group2.north);
\draw[proofarrow] (group2.south) -- (group3.north);
\draw[proofarrow] (tracking.south) -- (tracking.south |- capture.north);
\draw[proofarrow] (descent.south) -- (descent.south |- capture.north);
\draw[proofarrow] (summable.north |- support.south) -- (summable.north);
\draw[proofarrow] (summable.east) -- (kkt.west |- summable.east);
\draw[proofarrow] (strict.east) -- (rate.west);
\draw[reuse] (support.west) -- (-7.4cm,0 |- support.west)
  |- (strict.west);
\draw[reuse] (support.east) -- (7.4cm,0 |- support.east)
  |- (rate.east);
\end{tikzpicture}
\endgroup\space
\caption{Proof of Theorem~\ref{thm:ot-global}, organized into three parts.
Dashed arrows show the use of Lemma~\ref{lem:ot-supported-estimates}
in the proofs of strict complementarity and linear convergence.}
\label{fig:ot-proof-flowchart}
\end{figure}

\subsection{Small potential and summable errors}\label{sec:ot-feasibility}
We first control the marginal residuals $(r_Y^k,c_X^k)$ defined
in \eqref{eq:ot-marginal-residuals}
to show that the primal iterates
track the entropy PPA path over each finite time interval.
We then use Bregman descent to prove summability of the two Bregman
distances and the sums of entries outside $S$
in~\eqref{eq:ot-feasibility-summable}.

For $P\in\Pi$, define the row and column weighted-average operators
$M_P\coleq\Diag(a)^{-1}P$ and $N_P\coleq\Diag(b)^{-1}P^\top$.
For the fixed $s\in(0,2)$, define
\begin{equation}\label{eq:ot-residual-linear}
B_P(r,c)\coleq\big(-M_Pc,\ N_Pr+[(1-s)I+(1+s)N_PM_P]c\big).
\end{equation}
Lemma~\ref{lem:ot-residual-one-step} shows that $B_P$ gives the linear
part of the marginal residual update. We first establish its contraction property.
Let $I_d,J_d$ be the row and column index sets of the connected
components of $P$ from Definition~\ref{def:ot-support-graph}.
We consider the subspace defined by
\begin{equation}\label{eq:ot-component-zero-sums}
\mathcal{L}_{P} \;:=\; \Big\{ (r,c)\in\R^m\times\R^n \;\Big|\; \sum_{i\in I_d}a_ir_i=0,\quad
\sum_{j\in J_d}b_jc_j=0\quad\text{for every }d\Big\}.
\end{equation}
On this subspace, define the following quantity, which
Lemma~\ref{lem:ot-linear-contraction} shows to be a norm:
\begin{equation}\label{eq:ot-residual-norm}
\|(r,c)\|_P
\coleq\Big(\|r\|_a^2-\|N_Pr\|_b^2+s^{-1}\|c+N_Pr\|_b^2\Big)^{1/2}.
\end{equation}
\begin{lemma}[Contraction of the linear residual update]\label{lem:ot-linear-contraction}
The following statements hold for $s\in(0,2)$.
\begin{enumerate}[label=(\roman*),leftmargin=*]
\item For every $P\in\Pi$, the expression in~\eqref{eq:ot-residual-norm}
defines a norm on the subspace $\mathcal{L}_P$ in~\eqref{eq:ot-component-zero-sums},
and there exists $q_P\in(0,1)$ such that
$\|B_P(r,c)\|_P\leq q_P\|(r,c)\|_P$ on this subspace.
\item For every nonempty compact set $K\subseteq\Pi$ whose matrices
have the same support, there exist $m_K,M_K>0$ and $q_K\in(0,1)$ such that
\begin{equation*}
m_K\|(r,c)\|\leq\|(r,c)\|_P\leq M_K\|(r,c)\|,
\qquad \|B_P(r,c)\|_P\leq q_K\|(r,c)\|_P
\end{equation*}
for every $P\in K$ and every $(r,c)$ satisfying~\eqref{eq:ot-component-zero-sums}.
\end{enumerate}
\end{lemma}
\begin{proof}
(i) We first show that the quantity under the square root
in~\eqref{eq:ot-residual-norm} is nonnegative. By definition,
\begin{equation*}
\|(r,c)\|_P^2
=\big(\|r\|_a^2-\|N_Pr\|_b^2\big)+s^{-1}\|c+N_Pr\|_b^2.
\end{equation*}
The second term is nonnegative because $s>0$.
For the first term, $Pe_n=a$, $P^\top e_m=b$ and
$b_j(N_Pr)_j=\sum_iP_{ij}r_i$ give
\begin{equation}\label{eq:ot-weighted-row-difference}
\begin{aligned}
0 \;\leq\; \sum_{i,j}P_{ij}\big(r_i-(N_Pr)_j\big)^2
&=\sum_i a_ir_i^2-2\sum_j b_j(N_Pr)_j^2+\sum_j b_j(N_Pr)_j^2\\
&=\|r\|_a^2-\|N_Pr\|_b^2.
\end{aligned}
\end{equation}
Thus both terms are nonnegative for every $(r,c)$.

We next show that $\|(r,c)\|_P=0$, together with
\eqref{eq:ot-component-zero-sums}, implies $r=0$ and $c=0$.
Suppose these conditions hold. Since the two nonnegative terms in
$\|(r,c)\|_P^2$ sum to zero, each term must vanish.
It follows from~\eqref{eq:ot-weighted-row-difference} that
$r_i=(N_Pr)_j$ whenever $P_{ij}>0$.
Consequently, if $P_{ij}>0$ and $P_{i'j}>0$, then
$r_i=(N_Pr)_j=r_{i'}$.
For any $i,i'\in I_d$, connectedness gives row indices
$i=i_0,i_1,\ldots,i_\ell=i'$ and column indices $j_1,\ldots,j_\ell$
such that
\begin{equation*}
P_{i_{t-1}j_t}>0,\qquad P_{i_tj_t}>0,\qquad t=1,\ldots,\ell.
\end{equation*}
Applying the preceding equality successively gives
$r_i=r_{i_1}=\cdots=r_{i_\ell}=r_{i'}$.
Thus $r_i=\gamma_d$ for all $i\in I_d$, for some constant $\gamma_d$.
Condition~\eqref{eq:ot-component-zero-sums} gives
$0=\sum_{i\in I_d}a_ir_i=\gamma_d\sum_{i\in I_d}a_i$.
Since $\sum_{i\in I_d}a_i>0$, this forces $\gamma_d=0$ for every $d$,
hence $r=0$.
The vanishing second term now gives $c=0$.
Thus $\|(r,c)\|_P^2$ is a positive definite quadratic form on the stated
subspace, and $\|\cdot\|_P$ is a norm there.

Before comparing the norms of $(r,c)$ and $B_P(r,c)$, we first verify that $(r,c)\in \mathcal{L}_P$ implies that 
$B_P(r,c)\in \mathcal{L}_P$.
Since $P$ has no entries between distinct connected components,
\begin{equation*}
\begin{aligned}
\sum_{i\in I_d}a_i(M_Pc)_i
&=\sum_{i\in I_d,j\in J_d}P_{ij}c_j
 =\sum_{j\in J_d}b_jc_j,\\
\sum_{j\in J_d}b_j(N_Pr)_j
&=\sum_{i\in I_d,j\in J_d}P_{ij}r_i
 =\sum_{i\in I_d}a_ir_i.
\end{aligned}
\end{equation*}
Thus $M_Pc$, $N_Pr$ and $N_PM_Pc$ have zero weighted sums on their
respective components whenever $(r,c)$ satisfies
\eqref{eq:ot-component-zero-sums}. Both components of $B_P(r,c)$
in~\eqref{eq:ot-residual-linear} consequently satisfy the same conditions.

We next show that $\|B_P(r,c)\|_P^2 \leq \|(r,c)\|_P^2$.
From the definitions of $M_P,N_P$, we have
\begin{equation}\label{eq:ot-weighted-adjoint}
\langle M_Pc,r\rangle_a=r^\top Pc=\langle c,N_Pr\rangle_b.
\end{equation}
Substituting the two components in~\eqref{eq:ot-residual-linear}
into the square of~\eqref{eq:ot-residual-norm} and expanding the last term gives
\begin{equation}\label{eq:ot-linear-norm-expansion}
\begin{aligned}
\|B_P(r,c)\|_P^2
&=\|M_Pc\|_a^2-\|N_PM_Pc\|_b^2
 +s^{-1}\|c+N_Pr-s(I-N_PM_P)c\|_b^2\\
&=\|M_Pc\|_a^2-\|N_PM_Pc\|_b^2+s^{-1}\|c+N_Pr\|_b^2\\
&\quad-2\langle c+N_Pr,(I-N_PM_P)c\rangle_b
 +s\|(I-N_PM_P)c\|_b^2.
\end{aligned}
\end{equation}
Applying~\eqref{eq:ot-weighted-adjoint} to $(M_Pc,c)$ and $(r,c)$,
and expanding $\|(I-N_PM_P)c\|_b^2$, we obtain
\begin{equation}\label{eq:ot-energy-inner-products}
\begin{aligned}
\langle c,(I-N_PM_P)c\rangle_b
&=\|c\|_b^2-\|M_Pc\|_a^2
 =\|M_Pc\|_a^2-\|N_PM_Pc\|_b^2+\|(I-N_PM_P)c\|_b^2,\\
\langle N_Pr,(I-N_PM_P)c\rangle_b
&=\langle r,M_Pc\rangle_a-\langle N_Pr,N_PM_Pc\rangle_b.
\end{aligned}
\end{equation}
Subtracting~\eqref{eq:ot-linear-norm-expansion} from the square
of~\eqref{eq:ot-residual-norm} cancels $s^{-1}\|c+N_Pr\|_b^2$.
Substituting~\eqref{eq:ot-energy-inner-products} into the remaining terms gives
\begin{equation}\label{eq:ot-residual-energy}
\|(r,c)\|_P^2-\|B_P(r,c)\|_P^2
=\|r+M_Pc\|_a^2-\|N_P(r+M_Pc)\|_b^2
 +(2-s)\|(I-N_PM_P)c\|_b^2.
\end{equation}

Applying~\eqref{eq:ot-weighted-row-difference} with $r$ replaced by
$r+M_Pc$ shows that the difference of the first two terms on the right-hand side
of~\eqref{eq:ot-residual-energy} is nonnegative. Since $2-s>0$,
the second term is also nonnegative. Hence
$\|B_P(r,c)\|_P\leq\|(r,c)\|_P$.
Now we show that equality in this inequality holds only when
$(r,c)=(0,0)$. Suppose equality holds. Both nonnegative terms
in~\eqref{eq:ot-residual-energy} must then vanish, giving $N_PM_Pc=c$.
By~\eqref{eq:ot-weighted-adjoint},
$\langle c,N_PM_Pc\rangle_b=\|M_Pc\|_a^2$, so
\begin{equation*}
\sum_{i,j}P_{ij}\big(c_j-(M_Pc)_i\big)^2
=\|c\|_b^2-\|M_Pc\|_a^2
=\langle c,(I-N_PM_P)c\rangle_b=0.
\end{equation*}
Thus $c_j=(M_Pc)_i$ whenever $P_{ij}>0$.
The connected-component argument used to prove positive definiteness,
with rows and columns interchanged, shows that $c$ is constant on
each $J_d$. By~\eqref{eq:ot-component-zero-sums} and
$\sum_{j\in J_d}b_j>0$, each constant is zero, so $c=0$.
The first difference in~\eqref{eq:ot-residual-energy} now reduces to
$\|r\|_a^2-\|N_Pr\|_b^2=0$.
As established after~\eqref{eq:ot-weighted-row-difference}, this equality
makes $r$ constant on each $I_d$.
Condition~\eqref{eq:ot-component-zero-sums} then gives $r=0$.
It follows that $\|B_P(r,c)\|_P<\|(r,c)\|_P$ for every nonzero pair $(r,c)$
in the subspace $\mathcal{L}_P$.

We next choose a contraction factor $q_P<1$ that works for all pairs
in the subspace, with $P$ fixed.
If the subspace is $\{(0,0)\}$, any $q_P\in(0,1)$ suffices.
Otherwise its unit sphere in $\|\cdot\|_P$ is nonempty and compact.
The continuous function $(r,c)\mapsto\|B_P(r,c)\|_P$ attains its
maximum on this sphere. The strict inequality just proved makes
that maximum less than one. Choose $q_P\in(0,1)$ not smaller than
this maximum. Normalizing each nonzero pair by $\|(r,c)\|_P$ and
using homogeneity gives~\textup{(i)}; the zero pair satisfies the
bound as well.

(ii) Finally, we choose $m_K,M_K>0$ and $q_K\in(0,1)$, independent of
$P\in K$, so that both inequalities in~\textup{(ii)} hold.
The common support gives the same connected components for every
$P\in K$, so~\eqref{eq:ot-component-zero-sums} defines one subspace.
If this subspace is $\{(0,0)\}$, any $m_K,M_K>0$ and $q_K\in(0,1)$
satisfy~\textup{(ii)}. Otherwise the product of $K$ and the Euclidean
unit sphere of this subspace is nonempty and compact.
On this product, the functions
\begin{equation*}
(P,r,c)\mapsto\|(r,c)\|_P,
\qquad
(P,r,c)\mapsto\frac{\|B_P(r,c)\|_P}{\|(r,c)\|_P}
\end{equation*}
are continuous by~\eqref{eq:ot-residual-linear}--\eqref{eq:ot-residual-norm}.
Part~\textup{(i)} shows that the first function is positive and the
second is strictly less than one at every point.
They therefore attain a positive minimum $m_K$ and a finite maximum
$M_K$ for the first function, and a maximum less than one for the
second. Choose $q_K\in(0,1)$ not smaller than the latter maximum.
Normalizing each nonzero pair by its Euclidean norm and using
homogeneity gives the bounds in~\textup{(ii)} for all pairs satisfying
\eqref{eq:ot-component-zero-sums}; the zero pair also satisfies them.
\end{proof}

Write $P^k\coleq P((k-1)/\rho)$, where $P(\cdot)$ is defined in
\eqref{eq:ot-entropy-path}. The next lemma bounds
$\|Y^k-P^k\|_F$ and $\|X^k-P^{k+1}\|_F$ in terms of
$r_Y^k$ and $c_X^k$, respectively.

\begin{lemma}[Primal error bounds from marginal residuals]\label{lem:ot-path-reconstruction}
Fix $0\leq T<\infty$. There exist $\delta_0\in(0,1/2]$ and
$L_1>0$ such that, for every $\rho\geq1$ and
$1\leq k\leq\lfloor T\rho\rfloor+1$, the iterates satisfy
\begin{equation*}
Y^k=\mathcal P(P^k;a\odot(e_m+r_Y^k),b),\qquad
X^k=\mathcal P(P^{k+1};a,b\odot(e_n+c_X^k)).
\end{equation*}
Moreover, $\|(r_Y^k,c_X^k)\|\leq\delta_0$ implies
\begin{equation}\label{eq:ot-tracking-reconstruction}
\|Y^k-P^k\|_F\leq L_1\|r_Y^k\|,\qquad
\|X^k-P^{k+1}\|_F\leq L_1\|c_X^k\|.
\end{equation}
\end{lemma}
\begin{proof}
The path $\{P(t):0\leq t\leq T+1\}$ is compact and positive.
Its entries are therefore bounded away from zero, and its derivatives
are bounded on this interval. For $\rho\geq1$ and
$1\leq k\leq\lfloor T\rho\rfloor+1$, both $P^k$ and $P^{k+1}$
belong to this path.
By the path definition~\eqref{eq:ot-entropy-path} and
Proposition~\ref{prop:ot-projection-properties}(i), $P^k$ is a positive
diagonal scaling with
$P^k = {\rm Diag}(u^k)(X^0\odot e^{-(k-1)C/\rho}) {\rm Diag}(v^k)$ for some positive vectors $u^k$ and $v^k$. The same
statement holds at $P^{k+1}$.
For $k=1$, $Y^1=X^0=P^1$. For $k\geq2$, apply the first
identity in~\eqref{eq:ot-shifted-projections} with index $k-1$.
From the representations
of $Y^k$
and $X^k$ in~\eqref{eq:ot-shifted-projections}, and substituting
the definitions of $r_Y^k,c_X^k$ in~\eqref{eq:ot-marginal-residuals},
we get
$Y^k= \mathcal P (X^0\odot e^{-(k-1)C/\rho}; a \odot (e_m+r^k_Y),b)$,
and $X^k = \mathcal P (X^0\odot e^{-kC/\rho}; a, b\odot (e_n+c^k_X)).$
By applying \eqref{eq:ot-projection-invariance} to $Y^k$ and $X^k$, we
get
 the claimed projection representations.
Since $P^k,P^{k+1}\in\Pi$, their projections onto $\Pi$ are unchanged,
i.e., $P^k = \mathcal P (P^k; a,b)$ and $P^{k+1}=\mathcal P(P^{k+1}; a,b)$.

Choose $0<\delta_0\leq1/2$. For a pair $(r,c)$ satisfying
$a^\top r=b^\top c=0$ and $\|(r,c)\|\leq\delta_0$, the perturbed
marginals $a\odot(e_m+r)$ and $b\odot(e_n+c)$ are positive and
have the same respective total sums as $a$ and $b$.
The triples $(P(t),a\odot(e_m+r),b)$ and
$(P(t),a,b\odot(e_n+c))$ thus range over a compact subset of the domain in
Proposition~\ref{prop:ot-projection-properties}(iv).
Let $L_{\mathrm{proj}}>0$ be its Lipschitz constant on this compact
set, and set $L_1\coleq L_{\mathrm{proj}}\max\{\|a\|_\infty,\|b\|_\infty\}$.
Therefore, whenever $\|(r_Y^k,c_X^k)\|\leq\delta_0$,
\begin{equation*}
\begin{gathered}
\begin{aligned}
\|Y^k-P^k\|_F
&=\|\mathcal P(P^k;a\odot(e_m+r_Y^k),b)-\mathcal P(P^k;a,b)\|_F\\
&\leq L_{\mathrm{proj}}\|a\odot r_Y^k\|
 \leq L_{\mathrm{proj}}\|a\|_\infty\|r_Y^k\|
 \leq L_1\|r_Y^k\|,
\end{aligned}\\[3pt]
\begin{aligned}
\|X^k-P^{k+1}\|_F
&=\|\mathcal P(P^{k+1};a,b\odot(e_n+c_X^k))-\mathcal P(P^{k+1};a,b)\|_F\\
&\leq L_{\mathrm{proj}}\|b\odot c_X^k\|
 \leq L_{\mathrm{proj}}\|b\|_\infty\|c_X^k\|
 \leq L_1\|c_X^k\|.
\end{aligned}
\end{gathered}
\end{equation*}
This completes the proof.
\end{proof}

The next lemma writes $(r_Y^{k+1},c_X^{k+1})$ as
$B_{P^k}(r_Y^k,c_X^k)+e^k$ and bounds $e^k$ in terms of the current
residuals and $\rho^{-1}$.
\begin{lemma}[One-step estimate for marginal residuals]\label{lem:ot-residual-one-step}
Fix $0\leq T<\infty$. There exist $\delta_1\in(0,1/2]$,
$\rho_1\geq1$ and $L_4>0$ such that, for every fixed
$\rho\geq\rho_1$ and $1\leq k\leq\lfloor T\rho\rfloor$,
$\|(r_Y^k,c_X^k)\|\leq\delta_1$ implies
\begin{equation}\label{eq:ot-residual-remainder}
\begin{aligned}
(r_Y^{k+1},c_X^{k+1})&=B_{P^k}(r_Y^k,c_X^k)+e^k,\\
\|e^k\|&\leq L_4\big(\|(r_Y^k,c_X^k)\|^2
 +\rho^{-1}\|(r_Y^k,c_X^k)\|+\rho^{-2}\big)
\end{aligned}
\end{equation}
Here $B_{P^k}$ is defined in~\eqref{eq:ot-residual-linear}.
\end{lemma}
\begin{proof}
\indent\textup{(i)} We first define $F$ so that
$F(P^k,P^{k+1};(r_Y^k,c_X^k))=(r_Y^{k+1},c_X^{k+1})$.
Choose $0<\delta_1\leq\delta_0$, where $\delta_0$ is given by
Lemma~\ref{lem:ot-path-reconstruction}.
In this proof, $A,B$ belong to the convex hull $\mathcal K_T$ of
$\{P(t):0\leq t\leq T+1\}$, and $d\coleq(r,c)$ satisfies
$a^\top r=b^\top c=0$ and $\|d\|\leq\delta_1$.
For each $(A,B,d)$, we introduce auxiliary matrices
$\widehat Y,\widehat X$ and their outputs
$\widehat Y^+,\widehat X^+$ under one application
of~\eqref{eq:ot-two-primal-recurrence}.
These matrices depend on $(A,B,d)$ and do not form an additional
iteration sequence. Specifically, set
\begin{equation*}
\widehat Y\coleq\mathcal P(A;a\odot(e_m+r),b),\qquad
\widehat X\coleq\mathcal P(B;a,b\odot(e_n+c)).
\end{equation*}
Since $\widehat X^\top e_m=b\odot(e_n+c)$, the factor
$b\oslash(\widehat X^\top e_m)$ in~\eqref{eq:ot-two-primal-recurrence}
is $(e_n+c)^{-1}$. Accordingly, define
\begin{equation*}
\widehat Y^+\coleq\mathcal C_b(\widehat X),\qquad
\widehat X^+\coleq\mathcal R_a\!\left[
 (\widehat X^{\odot2}\oslash\widehat Y)
 \Diag((e_n+c)^{-(1+s)})\right]
\end{equation*}
and
\begin{equation*}
F(A,B;d)\coleq
\big((\widehat Y^+e_n)\oslash a-e_m,
 ((\widehat X^+)^\top e_m)\oslash b-e_n\big)
 \in\R^m\times\R^n.
\end{equation*}
For the actual iterates, $a^\top r_Y^k=b^\top c_X^k=0$.
Lemma~\ref{lem:ot-path-reconstruction} identifies $Y^k,X^k$ with
$\widehat Y,\widehat X$ when $A=P^k$, $B=P^{k+1}$ and
$d=(r_Y^k,c_X^k)$. Since
$b\oslash((X^k)^\top e_m)=(e_n+c_X^k)^{-1}$,
equation~\eqref{eq:ot-two-primal-recurrence} then identifies
$\widehat Y^+,\widehat X^+$ with $Y^{k+1},X^{k+1}$, respectively.
By the residual definitions~\eqref{eq:ot-marginal-residuals}, we obtain
\begin{equation*}
(r_Y^{k+1},c_X^{k+1})=F(P^k,P^{k+1};(r_Y^k,c_X^k)).
\end{equation*}

To differentiate $F$ and obtain Taylor bounds uniform in $\rho$ and $k$,
we first verify the domain conditions. The matrices $A,B$ are positive and
feasible, and the perturbed marginals $a\odot (e_m+r)$ and $b\odot(e_n+c)$ are positive because
$\delta_1\leq1/2$. The conditions $a^\top r=b^\top c=0$ ensure
compatible total sums in each projection.
The projection arguments range over a compact subset of the domain
in Proposition~\ref{prop:ot-projection-properties}(iv).
The resulting $\widehat X,\widehat Y$ are positive and depend
smoothly on $(A,B,d)$. Compactness bounds their entries away from
zero, as well as every denominator in the two normalizations.
Thus $F$ is smooth on a relative neighborhood of this domain,
and its derivatives through order two are bounded independently
of $\rho$ and $k$. Derivatives in $d$ are taken on the subspace
$a^\top r=b^\top c=0$.

\textup{(ii)} We next prove $F(P,P;0)=0$ and $D_dF(P,P;0)=B_P$
for every $P$ in the convex hull $\mathcal K_T$. At $d=0$, the two projections
$\widehat Y,\widehat X$ equal $P$. Since $P\in\Pi$, both
normalizations leave $P$ unchanged, so $F(P,P;0)=0$.

For a direction $d=(r,c)$ in the stated subspace, define
\begin{equation*}
G\coleq\left.\frac{d}{d\tau}\mathcal P(P;a\odot(e_m+\tau r),b)\right|_{\tau=0},\qquad
H\coleq\left.\frac{d}{d\tau}\mathcal P(P;a,b\odot(e_n+\tau c))\right|_{\tau=0}.
\end{equation*}
The marginal constraints of these projections give
\begin{equation*}
Ge_n=a\odot r,\quad G^\top e_m=0,\qquad
He_n=0,\quad H^\top e_m=b\odot c.
\end{equation*}
Using $M_P,N_P$ defined before Lemma~\ref{lem:ot-linear-contraction}, we
calculate the two marginal derivatives defining $D_dF(P,P;0)d$.
Since $(P+\tau H)^\top e_m=b\odot(e_n+\tau c)$,
\begin{equation*}
\left.\frac{d}{d\tau}\mathcal C_b(P+\tau H)\right|_{\tau=0}
=\left.\frac{d}{d\tau}\big[(P+\tau H)\Diag((e_n+\tau c)^{-1})\big]\right|_{\tau=0}
=H-P\Diag(c).
\end{equation*}
Taking its row sums and dividing by $a$ yields
\begin{equation*}
\Diag(a)^{-1}(H-P\Diag(c))e_n
=\Diag(a)^{-1}(0-Pc)=-M_Pc.
\end{equation*}

For the row update, entrywise differentiation gives
\begin{equation*}
\left.\frac{d}{d\tau}\big[(P+\tau H)^{\odot2}\oslash(P+\tau G)\big]\right|_{\tau=0}=2H-G,\qquad\left.\frac{d}{d\tau}(e_n+\tau c)^{-(1+s)}\right|_{\tau=0}=-(1+s)c.
\end{equation*}
By the product rule, the variation before normalization is
$$
\left.\frac{d}{d\tau}\big[\big((P+\tau H)^{\odot2}\oslash(P+\tau G)\big)
{\rm Diag}\big((e_n+\tau c)^{-(1+s)}\big)\big]\right|_{\tau=0}=2H-G -(1+s)P\Diag(c).
$$
 Its row sums are
\begin{equation*}
\big(2H-G-(1+s)P\Diag(c)\big)e_n
=-a\odot r-(1+s)Pc.
\end{equation*}
For any matrix direction $D$, the identity $Pe_n=a$ implies
\begin{equation*}
\left.\frac{d}{d\tau}\mathcal R_a(P+\tau D)\right|_{\tau=0}
=\left.\frac{d}{d\tau}\big[\Diag(a\oslash(a+\tau De_n))(P+\tau D)\big]\right|_{\tau=0}
=D-\Diag((De_n)\oslash a)P.
\end{equation*}
Substituting the variation before normalization for $D=2H-G-(1+s)P\Diag(c)$ gives the
variation of $\widehat X^+$:
\begin{equation*}
2H-G-(1+s)P\Diag(c)+\Diag(r+(1+s)M_Pc)P.
\end{equation*}
Its column sums, divided by $b$, are
\begin{equation*}
\begin{aligned}
&\Diag(b)^{-1}\big[2H^\top e_m-G^\top e_m
 -(1+s)\Diag(c)P^\top e_m+P^\top(r+(1+s)M_Pc)\big]\\[3pt]
&=2c-(1+s)c+N_P(r+(1+s)M_Pc)
\;=\;N_Pr+[(1-s)I+(1+s)N_PM_P]c.
\end{aligned}
\end{equation*}
Combining the two marginal derivatives with~\eqref{eq:ot-residual-linear}
gives
\begin{equation*}
D_dF(P,P;0)d=B_P(r,c).
\end{equation*}

\textup{(iii)} We next bound $\|F(P^k,P^{k+1};0)\|$ by a constant times
$\|P^{k+1}-P^k\|_F^2$ to obtain the $\rho^{-2}$ term
in~\eqref{eq:ot-residual-remainder}.
Since $A,B\in\Pi$, at $d=0$ we have
$\widehat Y=A$, $\widehat X=B$ and $\widehat Y^+=B$. Consequently,
\begin{equation*}
F(A,B;0)=\big(0,
 [\mathcal R_a(B^{\odot2}\oslash A)^\top e_m-b]\oslash b\big).
\end{equation*}
The identity
\begin{equation*}
B^{\odot2}\oslash A=2B-A+(B-A)^{\odot2}\oslash A
\end{equation*}
separates the quadratic term from a matrix satisfying both marginal
constraints:
\begin{equation*}
(2B-A)e_n=2a-a=a,\qquad (2B-A)^\top e_m=2b-b=b.
\end{equation*}
For $A=P^k$ and $B=P^{k+1}$, let
$\gamma\coleq\min_{t\in[0,T+1],i,j}P_{ij}(t)>0$.
The bounded path derivative gives $L_2>0$ such that
$\|P^{k+1}-P^k\|_F\leq L_2/\rho$.
If $\rho\geq2L_2/\gamma$, every entry of $2P^{k+1}-P^k$ is at
least $\gamma/2$. Its row constraints therefore give
$\mathcal R_a(2P^{k+1}-P^k)=2P^{k+1}-P^k$.
The matrices $(P^{k+1})^{\odot2}\oslash P^k$ and $2P^{k+1}-P^k$,
and the segments joining them, lie in a fixed compact positive set.
The derivative of $\mathcal R_a$ is bounded on that set.
Using also $(2P^{k+1}-P^k)^\top e_m=b$, we obtain $L_3>0$ such that
\begin{eqnarray*}
&& \|F(P^k,P^{k+1};0)\|
=\big\|[\mathcal R_a((P^{k+1})^{\odot2}\oslash P^k)^\top e_m-b]\oslash b\big\|\\
&&\leq\frac{\sqrt m}{\min_jb_j}
 \big\|\mathcal R_a((P^{k+1})^{\odot2}\oslash P^k)-(2P^{k+1}-P^k)\big\|_F\;\leq\; L_3\|(P^{k+1}-P^k)^{\odot2}\oslash P^k\|_F\\
&&\leq\frac{L_3}{\gamma}\|P^{k+1}-P^k\|_F^2
 \leq\frac{L_3L_2^2}{\gamma\rho^2}.
\end{eqnarray*}
Thus feasibility of $2P^{k+1}-P^k$ removes the first-order change
from both residuals.

\textup{(iv)} We finally estimate
$F(P^k,P^{k+1};d)-B_{P^k}d$ by the three terms
$\rho^{-2}$, $\rho^{-1}\|d\|$ and $\|d\|^2$
in~\eqref{eq:ot-residual-remainder}.
Fix $A,B$ and expand $F(A,B;d)$ at $d=0$.
Since the segment from $0$ to $d$ stays in the chosen residual
neighborhood, the bounded second derivatives give a constant
$L_F>0$ and a remainder $R(A,B;d)$ satisfying
\begin{equation*}
F(A,B;d)=F(A,B;0)+D_dF(A,B;0)d+R(A,B;d),\qquad\|R(A,B;d)\|\leq L_F\|d\|^2.
\end{equation*}
By part~\textup{(ii)}, $D_dF(A,A;0)d=B_A d$.
Subtracting $B_A d$ from the Taylor expansion therefore gives
\begin{equation} \label{eq-Taylor}
F(A,B;d)-B_A d
=F(A,B;0)+[D_dF(A,B;0)-D_dF(A,A;0)]d+R(A,B;d).
\end{equation}
The derivative bounds established in part~\textup{(i)} make $L_F$
independent of $\rho$ and $k$.
The segment from $A$ to $B$ remains in the positive convex hull $\mathcal{K}_T$.
Integrating the mixed derivative with respect to the second matrix
and $d$ along this segment, and increasing the same constant $L_F$
if necessary, gives
\begin{equation*}
\|D_dF(A,B;0)-D_dF(A,A;0)\|\leq L_F\|B-A\|_F,
\end{equation*}
where the left-hand-side norm is the Euclidean operator norm on the
residual subspace. Substitute $A=P^k$, $B=P^{k+1}$ and
$d=(r_Y^k,c_X^k)$ in the Taylor expansion \eqref{eq-Taylor}, and 
using Part~\textup{(i)},  we get 
\begin{equation*}
\begin{aligned}
\|(r_Y^{k+1},c_X^{k+1})-B_{P^k}(r_Y^k,c_X^k)\|
&\leq\|F(P^k,P^{k+1};0)\| + L_F\|P^{k+1}-P^k\|_F\|d\|
+ L_F\|d\|^2
\\
&\leq
L_F\|(r_Y^k,c_X^k)\|^2
 +\frac{L_FL_2}{\rho}\|(r_Y^k,c_X^k)\|
 +\frac{L_3L_2^2}{\gamma\rho^2}.
\end{aligned}
\end{equation*}
In the last inequality, we 
 used the bound from part~\textup{(iii)} for
$F(P^k,P^{k+1};0)$ and $\|P^{k+1}-P^k\|_F\leq L_2/\rho$.

Define $e^k\coleq(r_Y^{k+1},c_X^{k+1})-B_{P^k}(r_Y^k,c_X^k)$.
Choose
\begin{equation*}
\rho_1\geq\max\{1,2L_2/\gamma\},\qquad
L_4\geq\max\{L_F,L_FL_2,L_3L_2^2/\gamma\}.
\end{equation*}
All constants depend only on $T,C,a,b,X^0,s$, so the estimates hold
uniformly for $\rho\geq\rho_1$ and
$1\leq k\leq\lfloor T\rho\rfloor$ whenever the current residual
norm is at most $\delta_1$. This proves~\eqref{eq:ot-residual-remainder}.
\end{proof}

We now combine Lemmas~\ref{lem:ot-path-reconstruction}
and~\ref{lem:ot-residual-one-step} with the uniform contraction in
Lemma~\ref{lem:ot-linear-contraction}(ii) to obtain finite-time
tracking of the actual iterates $X^k$ and $Y^k$ along the entropy path
$P(\cdot)$. 

\begin{lemma}[Marginal residuals and finite-time tracking]
\label{lem:ot-path-one-step}
Fix $0\leq T<\infty$. There exist $\rho_0\geq1$, $A_T,B_T,C_T>0$ and
$q\in(0,1)$ such that, for every fixed $\rho\geq\rho_0$,
\begin{equation}\label{eq:ot-tracking-residual-bound}
\|r_Y^k\|+\|c_X^k\|
\leq A_Tq^{k-1}/\rho+B_T/\rho^2,
\qquad 1\leq k\leq\lfloor T\rho\rfloor+1.
\end{equation}
Moreover,
\begin{equation}\label{eq:ot-finite-time-tracking}
\begin{gathered}
\max_{1\leq k\leq\lfloor T\rho\rfloor+1}
\big(\|Y^k-P((k-1)/\rho)\|_F+\|X^k-P(k/\rho)\|_F\big)
\leq C_T/\rho,\\
\sum_{k=1}^{\lfloor T\rho\rfloor}\|z^{k+1}-z^k\|
\leq C_T/\rho.
\end{gathered}
\end{equation}
\end{lemma}
\begin{proof}
Take $\delta_0,L_1$ from Lemma~\ref{lem:ot-path-reconstruction}
and $\delta_1,\rho_1,L_4$ from Lemma~\ref{lem:ot-residual-one-step}.
In this proof, write $d^k\coleq(r_Y^k,c_X^k)$ and
$u_k\coleq\|d^k\|_{P^k}$.

\textup{(i)} We first compare the norms at consecutive points on the entropy path.
The path $\{P(t):0\leq t\leq T+1\}$ is compact and positive.
Its support graph has one connected component, so the equations in 
\eqref{eq:ot-component-zero-sums} reduce to $a^\top r=b^\top c=0$.
By Lemma~\ref{lem:ot-linear-contraction}(ii), there exist
$m_T,M_T>0$ and $q_0\in(0,1)$ such that, for every $P$ on the path
and every $d=(r,c)$ with $a^\top r=b^\top c=0$,
\begin{equation*}
m_T\|d\|\leq\|d\|_P\leq M_T\|d\|,\qquad
\|B_Pd\|_P\leq q_0\|d\|_P.
\end{equation*}
By~\eqref{eq:ot-residual-norm}, the squared norm is smooth in $P$
and quadratic in $d$. Its derivative along the compact smooth path
is bounded by a constant times $\|d\|^2$.
For $\rho\geq1$ and $1\leq k\leq\lfloor T\rho\rfloor$, integrate
over the path interval from $(k-1)/\rho$ to $k/\rho$ and use
$m_T\|d\|\leq\|d\|_{P^k}$. This gives $L_5>0$, independent of
$\rho,k,d$, such that
\begin{equation*}
\big|\|d\|_{P^{k+1}}^2-\|d\|_{P^k}^2\big|\leq\frac{2L_5}{\rho}\|d\|_{P^k}^2,\qquad\|d\|_{P^{k+1}}\leq\sqrt{1+2L_5/\rho}\,\|d\|_{P^k}\leq(1+L_5/\rho)\|d\|_{P^k}.
\end{equation*}

\textup{(ii)} We next derive a recurrence for the residuals, assuming that
they are sufficiently small. Fix $q\in(q_0,1)$.
The actual residuals satisfy $a^\top r_Y^k=b^\top c_X^k=0$.
Since $B_{P^k}$ preserves this subspace, and also 
$a^\top r_Y^{k+1}=b^\top c_X^{k+1}=0$, the remainder $e^k$
in~\eqref{eq:ot-residual-remainder} belongs to it as well.

Choose $0<\delta\leq\min\{\delta_0,\delta_1,\frac{(q-q_0)m_T}{2L_4M_T}\}$.
For now, suppose $\rho\geq\rho_1$,
$1\leq k\leq\lfloor T\rho\rfloor$ and
$\|d^k\|\leq\delta$.
Using $\|d^k\|^2\leq\delta\|d^k\|$ and $\|d^k\|\leq u_k/m_T$
in~\eqref{eq:ot-residual-remainder}, we obtain
\begin{equation*}
\|e^k\|\leq L_4\big[(\delta+\rho^{-1})\|d^k\|+\rho^{-2}\big]
\leq\frac{L_4(\delta+\rho^{-1})}{m_T}u_k+\frac{L_4}{\rho^2}.
\end{equation*}
Applying the bounds in part~\textup{(i)} to $B_{P^k}d^k+e^k$ gives
\begin{equation*}
u_{k+1}\leq(1+L_5/\rho)\big(q_0u_k+M_T\|e^k\|\big)\leq(1+L_5/\rho)\left[\left(q_0+\frac{M_TL_4(\delta+\rho^{-1})}{m_T}\right)u_k+\frac{M_TL_4}{\rho^2}\right].
\end{equation*}
Now choose $\rho_0\geq\max\{1,\rho_1\}$ sufficiently large so that
for all $\rho\geq\rho_0$,
\begin{equation*}
(1+L_5/\rho)\left(\frac{q_0+q}{2}
 +\frac{M_TL_4}{m_T\rho}\right)\leq q,\qquad L_5/\rho\leq1.
\end{equation*}
Such a choice is possible because the first left-hand side converges
to $(q_0+q)/2<q$. Setting $L_6\coleq2M_TL_4$ gives
\begin{equation}\label{eq:ot-tracking-contraction}
u_{k+1}\leq q u_k+L_6/\rho^2
\end{equation}
whenever $\rho\geq\rho_0$, $1\leq k\leq\lfloor T\rho\rfloor$
and $\|d^k\|\leq\delta$.

\textup{(iii)} We show by induction that the residuals remain in this
neighborhood throughout the stated iteration range.
The initialization gives $Y^1=X^0$, $Z^1=0$ and
$X^1=\mathcal R_a(X^0\odot e^{-C/\rho})$, with $X^0\in\Pi$.
The map $t\mapsto\mathcal R_a(X^0\odot e^{-tC})$ equals $X^0$ at
$t=0$ and has bounded derivative for $0\leq t\leq1$.
By~\eqref{eq:ot-marginal-residuals}, $r_Y^1=0$ and there is $L_7>0$
such that
\begin{equation*}
\|c_X^1\|
=\big\|\big[(\mathcal R_a(X^0\odot e^{-C/\rho})-X^0)^\top e_m\big]\oslash b\big\|
\leq L_7/\rho.
\end{equation*}
Thus $u_1\leq M_TL_7/\rho$.
With $q$ and $\delta$ already fixed, enlarge $\rho_0$ so that,
for every $\rho\geq\rho_0$,
\begin{equation*}
\frac1{m_T}\left(\frac{M_TL_7}{\rho}
 +\frac{L_6}{(1-q)\rho^2}\right)\leq\frac\delta2.
\end{equation*}
This gives $\|d^1\|\leq u_1/m_T\leq\delta/2$.
Suppose $\|d^j\|\leq\delta/2$ for $1\leq j\leq k$, where
$k\leq\lfloor T\rho\rfloor$. Then~\eqref{eq:ot-tracking-contraction}
is valid for every $j=1,\ldots,k$, so
\begin{equation*}
u_{k+1}\leq q^ku_1+\frac{L_6}{\rho^2}\sum_{j=0}^{k-1}q^j
\leq\frac{M_TL_7q^k}{\rho}+\frac{L_6}{(1-q)\rho^2}.
\end{equation*}
It follows that $\|d^{k+1}\|\leq u_{k+1}/m_T\leq\delta/2$,
which closes the induction. The same bounds, including the initial
index $k=1$, give
\begin{equation} \label{eq-residual-bound}
\|r_Y^k\|+\|c_X^k\|
\leq\sqrt2\|d^k\|\leq\frac{\sqrt2}{m_T}u_k
\leq\frac{\sqrt2M_TL_7q^{k-1}}{m_T\rho}
 +\frac{\sqrt2L_6}{m_T(1-q)\rho^2}.
\end{equation}
Hence~\eqref{eq:ot-tracking-residual-bound} holds with
$A_T\coleq\sqrt2M_TL_7/m_T$ and
$B_T\coleq\sqrt2L_6/(m_T(1-q))$.

\textup{(iv)} We use the residual bound \eqref{eq-residual-bound} to prove the two remaining estimates
in the lemma.
Since part~\textup{(iii)} gives $\|d^k\|\leq\delta/2\leq\delta_0$,
Lemma~\ref{lem:ot-path-reconstruction} applies throughout
$1\leq k\leq\lfloor T\rho\rfloor+1$.
Using~\eqref{eq:ot-tracking-residual-bound} and $\rho\geq1$
in~\eqref{eq:ot-tracking-reconstruction}, we obtain
\begin{equation*}
\|Y^k-P^k\|_F+\|X^k-P^{k+1}\|_F
\leq L_1(\|r_Y^k\|+\|c_X^k\|)
\leq\frac{L_1(A_T+B_T)}{\rho}.
\end{equation*}
For the multiplier increments, Proposition~\ref{prop:ot-equivalent-updates}(ii)
and~\eqref{eq:ot-marginal-residuals} give
\begin{equation}\label{eq:ot-residual-update}
z^{k+1}-z^k=-s\log(e_n+c_X^k).
\end{equation}
Part~\textup{(iii)} also gives $\|c_X^k\|_\infty\leq\delta\leq1/2$.
For $|t|\leq1/2$, the derivative of $\log(1+t)$ is at most $2$,
so $\|\log(e_n+c_X^k)\|\leq2\|c_X^k\|$.
Combining~\eqref{eq:ot-residual-update} with
\eqref{eq:ot-tracking-residual-bound}, we obtain
\begin{equation*}
\sum_{k=1}^{\lfloor T\rho\rfloor}\|z^{k+1}-z^k\|\leq2s\sum_{k=1}^{\lfloor T\rho\rfloor}\|c_X^k\|\leq\frac{2sA_T}{\rho}\sum_{k=1}^{\lfloor T\rho\rfloor}q^{k-1}+\frac{2sB_T\lfloor T\rho\rfloor}{\rho^2}\leq\frac{2s}{\rho}\left(\frac{A_T}{1-q}+TB_T\right).
\end{equation*}
Thus both inequalities in~\eqref{eq:ot-finite-time-tracking} hold with
$C_T\coleq\max\{L_1(A_T+B_T),2s(A_T/(1-q)+TB_T)\}$.
All chosen constants are independent of $\rho$ and $k$.
If $T=0$, part~\textup{(iii)} requires only the initial index $k=1$,
and the sum of multiplier increments is empty, so both conclusions
remain valid.
\end{proof}

We estimate feasibility by the Bregman distance 
$D_H(Y^{k+1};X^k).$
Proposition~\ref{prop:ot-equivalent-updates}(ii)
gives
\begin{equation}\label{eq:ot-feasibility-error}
D_H(Y^{k+1};X^k)=D_H\big(b;(X^k)^\top e_m\big) 
= \sum_jb_j(c_{X,j}^k-\log(1+c_{X,j}^k)),
\end{equation}
where the last equality used the fact that 
$(X^k)^\top e_m = b\odot(e_n + c^k_X)$.

For an {\em optimal} solution $P$ and a dual optimal solution $(\alpha_*,\beta_*)$,
let $\Gamma$ be its dual slack and set $z_*\coleq\beta_*/\rho$.
Define the potential function: 
\begin{equation}\label{eq:ot-radial-energy}
\mathcal E_P(Y,z;z_*)\coleq D_H(P;Y)+\frac{\|z-z_*\|_b^2}{2s},
\qquad Y>0,\quad z\in\R^n.
\end{equation}
Here $Z_*\coleq-\rho e_mz_*^\top$ is a multiplier for the constraint $X-Y=0$ at the
optimal solution. We estimate $D_H(P;Y^k)$ and $\|z^k-z_*\|_b^2/(2s)$
together. Their changes contribute the inner products
$\langle c_X^k,z^k-z_*\rangle_b$ and
$-\langle\log(e_n+c_X^k),z^k-z_*\rangle_b$. Their sum contains
$c_X^k-\log(e_n+c_X^k)$, which is of quadratic order in $c_X^k$ near zero.

\begin{lemma}[Bregman descent]\label{prop:ot-entropy-identity}
Let Algorithm~\ref{alg:ot} generate $\{(X^k,Y^k,Z^k)\}$ with
$Z^0=0$ and $0<s<2$, and write $Z^k=-\rho e_m(z^k)^\top$ as in
Proposition~\ref{prop:ot-equivalent-updates}(ii).
Fix a primal--dual optimal solution $(P,\alpha_*,\beta_*)$.
There exists $h>0$ such that, for every $k\geq1$,
$\mathcal E_P(Y^k,z^k;z_*)\leq h$ implies
\begin{equation}\label{eq:ot-entropy-descent}
\mathcal E_P(Y^{k+1},z^{k+1};z_*)
\leq\mathcal E_P(Y^k,z^k;z_*)-\rho^{-1}\langle\Gamma,X^k\rangle
-\tfrac12D_H(Y^{k+1};Y^k)-\tfrac{2-s}{2}D_H(Y^{k+1};X^k).
\end{equation}
The threshold $h$ can be chosen independently of $\rho>0$ and does not
require a positive lower bound on the nonnegative dual slack matrix $\Gamma$.
\end{lemma}
\begin{proof}
Fix $k\geq1$.

\textup{(i)} We first derive an exact expression for the change in the potential.
In the row optimality condition~\eqref{eq:sub-kkt}, substitute
$C=\Gamma+\alpha_*e_n^\top+e_m\beta_*^\top$,
$Z^k=-\rho e_m(z^k)^\top$ and $z_*=\beta_*/\rho$ to obtain
\begin{equation} \label{eq-tmp1}
\log X^k-\log Y^k
=-\Gamma/\rho+e_m(z^k-z_*)^\top
 -\rho^{-1}(\alpha_*+r^k)e_n^\top.
\end{equation}
Exponentiating and
imposing $X^ke_n=a$, together with the column update in
Proposition~\ref{prop:ot-equivalent-updates}(ii), gives
\begin{equation}\label{eq:ot-potential-update}
X^k=\mathcal R_a\!\left(Y^k\odot
 e^{-\Gamma/\rho+e_m(z^k-z_*)^\top}\right),\qquad
Y^{k+1}=\mathcal C_b(X^k).
\end{equation}

The Bregman three-point identity~\cite[Lemma~2.1]{yang2022bregman},
applied to $(P,X^k,Y^k)$, yields
\begin{equation*}
D_H(P;X^k)-D_H(P;Y^k)
=-D_H(X^k;Y^k)+\langle\log X^k-\log Y^k,X^k-P\rangle.
\end{equation*}
Since $(X^k-P)e_n=0$, the last term in the logarithmic
relation in \eqref{eq-tmp1} has zero inner product with $X^k-P$.
Using also $\langle\Gamma,P\rangle=0$ and $P^\top e_m=b$, we obtain
\begin{eqnarray*}
D_H(P;X^k)-D_H(P;Y^k)
&=&-D_H(X^k;Y^k)-\rho^{-1}\langle\Gamma,X^k\rangle
+\langle (X^k)^\top e_m-b,z^k-z_*\rangle.
\end{eqnarray*}
By~\eqref{eq:ot-potential-update}, $Y^{k+1}$ is the entropy projection
of $X^k$ onto $\cY$. Since $Y^{k+1},P\in\cY$, 
the Bregman three-point identity
applied to $(P,Y^{k+1},X^k)$ and using
\eqref{eq:sub-kkt}, yields the Pythagorean identity:
\begin{equation*}
D_H(P;X^k)=D_H(P;Y^{k+1})+D_H(Y^{k+1};X^k).
\end{equation*}
For the multiplier term, substitute~\eqref{eq:ot-residual-update}
into the difference of the two squared norms to obtain
\begin{equation*}
\frac{\|z^{k+1}-z_*\|_b^2-\|z^k-z_*\|_b^2}{2s}
=-\langle\log(e_n+c_X^k),z^k-z_*\rangle_b
 +\frac{s}{2}\|\log(e_n+c_X^k)\|_b^2.
\end{equation*}
Using the fact that 
$(X^k)^\top e_m-b=b\odot c_X^k$ by the definition~\eqref{eq:ot-marginal-residuals}, we get
\begin{eqnarray} 
&&\hspace{-10mm}
\mathcal E_P(Y^{k+1},z^{k+1};z_*)-\mathcal E_P(Y^k,z^k;z_*)
= D_H(P;Y^{k+1}) - D_H(P;Y^k) + 
\frac{\|z^{k+1}-z_*\|_b^2-\|z^k-z_*\|_b^2}{2s}\nonumber
\\
&=&-\rho^{-1}\langle\Gamma,X^k\rangle-D_H(X^k;Y^k)
 -D_H(Y^{k+1};X^k)\nonumber \\
&&+\langle c_X^k-\log(e_n+c_X^k),z^k-z_*\rangle_b
 +\frac{s}{2}\|\log(e_n+c_X^k)\|_b^2.
 \label{eq-potential-difference}
\end{eqnarray}

Now we estimate $D_H(X^k;Y^k)$ and $D_H(Y^{k+1};X^k)$. 
Note that
$X^k_{\cdot j}=(1+c_{X,j}^k)Y^{k+1}_{\cdot j}$ from
\eqref{eq:ot-potential-update}. Both $Y^k_{\cdot j}$ and
$Y^{k+1}_{\cdot j}$ have entry sum $b_j$, so the definition of $D_H$ gives
\begin{equation*}
D_H(X^k_{\cdot j};Y^k_{\cdot j})=(1+c_{X,j}^k)D_H(Y^{k+1}_{\cdot j};Y^k_{\cdot j})+b_j\big[(1+c_{X,j}^k)\log(1+c_{X,j}^k)-c_{X,j}^k\big].
\end{equation*}
The sum of the last terms over $j$ is $D_H((X^k)^\top e_m;b)$.
Consequently,
\begin{equation}\label{eq:ot-entropy-remainder}
D_H(X^k;Y^k)-D_H((X^k)^\top e_m;b)
=\sum_j(1+c_{X,j}^k)D_H(Y^{k+1}_{\cdot j};Y^k_{\cdot j})\geq0.
\end{equation}
The same column relation gives the two formulas
\begin{equation*}
\begin{aligned}
D_H((X^k)^\top e_m;b)
&=\sum_jb_j\big[(1+c_{X,j}^k)\log(1+c_{X,j}^k)-c_{X,j}^k\big],\\
D_H(Y^{k+1};X^k)
&=\sum_jb_j\big[c_{X,j}^k-\log(1+c_{X,j}^k)\big].
\end{aligned}
\end{equation*}
Adding them yields
\begin{equation*}
D_H((X^k)^\top e_m;b)+D_H(Y^{k+1};X^k)
=\langle c_X^k,\log(e_n+c_X^k)\rangle_b.
\end{equation*}
In this inner product, write
$c_X^k=\log(e_n+c_X^k)+(c_X^k-\log(e_n+c_X^k))$.
Substituting into the potential difference \eqref{eq-potential-difference} gives
\begin{equation}\label{eq:ot-radial-identity}
\begin{aligned}
&\mathcal E_P(Y^{k+1},z^{k+1};z_*)-\mathcal E_P(Y^k,z^k;z_*)\\
&\quad=-\rho^{-1}\langle\Gamma,X^k\rangle
-\big[D_H(X^k;Y^k)-D_H((X^k)^\top e_m;b)\big]
-\left(1-\frac{s}{2}\right)\|\log(e_n+c_X^k)\|_b^2\\
&\qquad+\big\langle c_X^k-\log(e_n+c_X^k),
 z^k-z_*-\log(e_n+c_X^k)\big\rangle_b.
\end{aligned}
\end{equation}

\textup{(ii)} We next show that a small potential makes
$c_X^k$ and $z^k-z_*$ small.
The definition~\eqref{eq:ot-radial-energy} directly controls
$D_H(P;Y^k)$ and $\|z^k-z_*\|_b$, but we also need to control $X^k$.
For this purpose, define the input to the row projection by
$W\coleq Y^k\odot e^{-\Gamma/\rho+e_m(z^k-z_*)^\top}$.
Equation~\eqref{eq:ot-potential-update} gives $X^k=\mathcal R_a(W)$.
Since $X^k,P\in\cX$,
the Bregman three-point identity
applied to $(P,X^k,W)$ and using \eqref{eq-tmp1}, yields the Pythagorean identity:
\begin{equation*}
D_H(P;W)=D_H(P;X^k)+D_H(X^k;W)\geq D_H(P;X^k).
\end{equation*}
Using $P^\top e_m=b$ in the definition of $D_H$, we have
\begin{equation*}
D_H(P;W)=D_H(P;Y^k)+\rho^{-1}\langle\Gamma,P\rangle-\langle b,z^k-z_*\rangle+\sum_{i,j}Y^k_{ij}\big(e^{-\Gamma_{ij}/\rho+z_j^k-(z_*)_j}-1\big).
\end{equation*}
Complementarity gives $\langle\Gamma,P\rangle=0$.
Since $\Gamma\geq0$ and $(Y^k)^\top e_m=b$, dropping
$-\Gamma_{ij}/\rho$ in the exponent gives an upper bound.
Then $e^t-1-t\leq e^{|t|}t^2/2$ implies
\begin{equation*}
\begin{aligned}
D_H(P;X^k)
&\leq D_H(P;Y^k)
 +\sum_jb_j\big(e^{z_j^k-(z_*)_j}-1-z_j^k+(z_*)_j\big)\\
&\leq D_H(P;Y^k)
 +\tfrac12e^{\|z^k-z_*\|_\infty}\|z^k-z_*\|_b^2\\
&\leq\max\{1,se^{\|z^k-z_*\|_\infty}\}
       \mathcal E_P(Y^k,z^k;z_*).
\end{aligned}
\end{equation*}

The matrices $P,X^k$ have entry sum one. Pinsker's inequality
and~\eqref{eq:ot-radial-energy} therefore give
\begin{equation} \label{eq-tmp3}
\|P-X^k\|_1^2\leq2D_H(P;X^k),\qquad
\|z^k-z_*\|_\infty^2\leq\frac{2s\mathcal E_P(Y^k,z^k;z_*)}{\min_jb_j}.
\end{equation}
Moreover, $P^\top e_m=b$ and~\eqref{eq:ot-marginal-residuals} imply
\begin{equation*}
\|c_X^k\|_\infty^2
\leq\frac{\|(X^k-P)^\top e_m\|_\infty^2}{(\min_jb_j)^2}
\leq\frac{\|X^k-P\|_1^2}{(\min_jb_j)^2}
\leq \frac{2\max\{1,se^{\|z^k-z_*\|_\infty}\}}{(\min_jb_j)^2}
\mathcal E_P(Y^k,z^k;z_*)
.
\end{equation*}
These bounds show that $\|c_X^k\|_\infty$ and $\|z^k-z_*\|_\infty$
tend to zero as the potential $\mathcal E_P(Y^k,z^k;z_*)$ tends to zero, uniformly in $\rho>0$.
Once $\|c_X^k\|_\infty\leq1/2$, the derivative bound for
$\log(1+t)$ on $|t|\leq1/2$ gives
$\|\log(e_n+c_X^k)\|_\infty\leq2\|c_X^k\|_\infty$.
Hence this logarithmic norm tends to zero uniformly as well.
Note that the bounds use only $b,s$, complementarity and $\Gamma\geq0$;
no positive lower bound on the nonzero entries of $\Gamma$ is needed.

\textup{(iii)} We now bound the last inner product 
in~\eqref{eq:ot-radial-identity}.
For $u>-1$, $u-\log(1+u)\geq0$, and part~\textup{(i)} gives
\begin{equation*}
\sum_j b_j\big[c_{X,j}^k-\log(1+c_{X,j}^k)\big]
=D_H(Y^{k+1};X^k).
\end{equation*}
Taking absolute values term by term and bounding each component
of $z^k-z_*-\log(e_n+c_X^k)$ by the sum of the two maximum norms gives
\begin{equation*}
\begin{aligned}
&\left|\big\langle c_X^k-\log(e_n+c_X^k),
 z^k-z_*-\log(e_n+c_X^k)\big\rangle_b\right|\\
&\quad\leq
\big(\|z^k-z_*\|_\infty+\|\log(e_n+c_X^k)\|_\infty\big)
D_H(Y^{k+1};X^k).
\end{aligned}
\end{equation*}
Apply $e^t-1-t\leq e^{|t|}t^2/2$ with
$t=\log(1+c_{X,j}^k)$ and sum with weights $b_j$. The same formula
for $D_H(Y^{k+1};X^k)$ yields
\begin{equation*}
D_H(Y^{k+1};X^k)
\leq\tfrac12e^{\|\log(e_n+c_X^k)\|_\infty}
\|\log(e_n+c_X^k)\|_b^2.
\end{equation*}
Since $2-s>0$, it follows that
\begin{equation*}
-\left(1-\frac{s}{2}\right)\|\log(e_n+c_X^k)\|_b^2
\leq-(2-s)e^{-\|\log(e_n+c_X^k)\|_\infty}D_H(Y^{k+1};X^k).
\end{equation*}

By part~\textup{(ii)}, we can choose $h>0$, independently of $\rho$,
so that $\mathcal E_P(Y^k,z^k;z_*)\leq h$ implies
$\|c_X^k\|_\infty\leq1/2$ and
\begin{equation*}
(2-s)e^{-\|\log(e_n+c_X^k)\|_\infty}
-\|z^k-z_*\|_\infty-\|\log(e_n+c_X^k)\|_\infty
\geq(2-s)/2.
\end{equation*}
Such a choice is possible because the left-hand side tends to
$2-s$ as the potential tends to zero.
The last two terms in~\eqref{eq:ot-radial-identity} are therefore
bounded above by $-(2-s)D_H(Y^{k+1};X^k)/2$.
Also $1+c_{X,j}^k\geq1/2$, so~\eqref{eq:ot-entropy-remainder} gives
\begin{equation*}
D_H(X^k;Y^k)-D_H((X^k)^\top e_m;b)
\geq\tfrac12\sum_jD_H(Y^{k+1}_{\cdot j};Y^k_{\cdot j})
=\tfrac12D_H(Y^{k+1};Y^k).
\end{equation*}
Substituting these two bounds into~\eqref{eq:ot-radial-identity}
proves~\eqref{eq:ot-entropy-descent} with the claimed choice of $h$.
\end{proof}

For $W\geq0$, write $\tau(W)\coleq\sum_{(i,j)\notin S}W_{ij}$.
Fix a dual optimal solution $(\alpha_*,\beta_*)$ satisfying
\eqref{eq:ot-strict-certificate}, with slack $\Gamma$.
The next proposition {states that 
the potential~\eqref{eq:ot-radial-energy}
enters
an arbitrarily small sublevel set.}
Its control of $\tau(Y^k)$ will also bound the marginal residuals.
\begin{proposition}[Entry and summable errors]\label{prop:ot-feasibility-sum}
Let $P_*$ be defined as in \eqref{eq:ot-selected}.
For every sufficiently small $h>0$, there exists
$\rho_0(C,a,b,X^0,s,h)>0$ such that, for each fixed $\rho\geq\rho_0$,
there is $N\geq1$ with
\begin{equation}\label{eq:ot-captured-potential}
\mathcal E_{P_*}(Y^k,z^k;\beta_*/\rho)\leq h
\qquad \forall\; k\geq N.
\end{equation}
Moreover,
\begin{equation}\label{eq:ot-feasibility-summable}
\begin{gathered}
\sum_{k=1}^{\infty}\big[D_H(Y^{k+1};X^k)+D_H(Y^{k+1};Y^k)\big]<\infty,\\
\sum_{k=1}^{\infty}\big[\tau(X^k)+\tau(Y^k)\big]<\infty.
\end{gathered}
\end{equation}
\end{proposition}
\begin{proof}
\textup{(i)} We first choose a path point with small entropy distance from $P_*$.
Fix $h>0$ below the threshold in Lemma~\ref{prop:ot-entropy-identity}.
By~\eqref{eq:ot-path-limit}, there is $T<\infty$ such that
$D_H(P_*;P(t))<h/4$ for every $t\geq T$.
For $\rho\geq1$, set $N\coleq\lceil T\rho\rceil+1$.
Then $T\leq(N-1)/\rho\leq T+1$ and
$N\leq\lfloor(T+1)\rho\rfloor+1$.
Proposition~\ref{prop:ot-iterate-projections}(i) and $Z^0=0$ give
\begin{equation}\label{eq:ot-path-identity}
\mathcal P(Y^k;a,b)=P((k-1)/\rho),\qquad k\geq1.
\end{equation}
Since $P_*\in\Pi$, applying~\eqref{eq:ot-projection-pythagoras} with
$W=Y^N$, $Q=P((N-1)/\rho)$ and $V=P_*$ yields
\begin{equation*}
D_H(P_*;Y^N)=D_H(P_*;P((N-1)/\rho))+D_H(P((N-1)/\rho);Y^N).
\end{equation*}
Thus it remains to control the second divergence and the multiplier at index $N$.

\textup{(ii)} We show that the actual iterate at this index has small potential.
The continuous positive path on $[0,T+1]$ has an entrywise lower bound
$\gamma>0$. Lemma~\ref{lem:ot-path-one-step}, applied on this interval,
gives a constant $C_{T+1}$ independent of $\rho$ such that
$\|Y^N-P((N-1)/\rho)\|_F\leq C_{T+1}/\rho$ for all sufficiently large $\rho$.
Increase the lower bound on $\rho$ so that $C_{T+1}/\rho\leq\gamma/2$.
Every entry of $Y^N$ is then at least $\gamma/2$.
Using $D_H(Q;Y)\leq\sum_{ij}(Q_{ij}-Y_{ij})^2/Y_{ij}$ gives
\begin{equation*}
D_H(P((N-1)/\rho);Y^N)
\leq\frac{2}{\gamma}\|P((N-1)/\rho)-Y^N\|_F^2
\leq\frac{2C_{T+1}^2}{\gamma\rho^2}.
\end{equation*}
Also $z^1=0$ and $N-1\leq\lfloor(T+1)\rho\rfloor$.
The second bound \eqref{eq:ot-finite-time-tracking} in Lemma~\ref{lem:ot-path-one-step} therefore implies
$\|z^N\|\leq\sum_{k=1}^{N-1}\|z^{k+1}-z^k\|\leq C_{T+1}/\rho$.
Consequently,
\begin{equation*}
\frac{\|z^N-\beta_*/\rho\|_b^2}{2s}
\leq\frac{\|b\|_\infty(C_{T+1}+\|\beta_*\|)^2}{2s\rho^2}.
\end{equation*}
Combining these bounds with part~\textup{(i)} gives a constant $L_1>0$,
independent of $\rho$, such that
\begin{equation}\label{eq-tmp2}
\mathcal E_{P_*}(Y^N,z^N;\beta_*/\rho)
\leq D_H(P_*;P((N-1)/\rho))+L_1/\rho^2<h/2
\end{equation}
for every $\rho\geq\rho_0$, after choosing $\rho_0$ so that $L_1/\rho_0^2<h/4$.
The choices are in the order $h$, $T$, then $\rho_0$.

\textup{(iii)} We next prove that the potential stays small and that the errors are summable.
At index $N$, the strict bound \eqref{eq-tmp2} in part~\textup{(ii)} allows us to apply
\eqref{eq:ot-entropy-descent}. Its three subtracted terms are nonnegative,
so the potential at $N+1$ is also below $h$. Induction proves
\eqref{eq:ot-captured-potential} and permits the same descent inequality
at every $k\geq N$.
Summing it from $N$ to $K$ cancels the intermediate potentials.
The last potential is nonnegative, so letting $K\to\infty$ gives
\begin{equation*}
\sum_{k=N}^\infty\left[
\rho^{-1}\langle\Gamma,X^k\rangle
+\frac12D_H(Y^{k+1};Y^k)
+\frac{2-s}{2}D_H(Y^{k+1};X^k)\right]
\leq\mathcal E_{P_*}(Y^N,z^N;\beta_*/\rho).
\end{equation*}
If $S^c\ne\varnothing$, then~\eqref{eq:ot-strict-certificate} gives
$\Gamma_{ij}>0$ on $S^c$, so $\min_{S^c}\Gamma_{ij}>0$. Since $\rho$ is fixed,
$\rho^{-1}\langle\Gamma,X^k\rangle\geq
\rho^{-1}(\min_{S^c}\Gamma_{ij})\tau(X^k)$,
which proves $\sum_{k=N}^\infty\tau(X^k)<\infty$.
If $S^c=\varnothing$, these terms vanish.
The threshold in Lemma~\ref{prop:ot-entropy-identity} also ensures
$\|c_X^k\|_\infty\leq1/2$ for $k\geq N$.
The column update $Y^{k+1}_{ij}=X^k_{ij}/(1+c_{X,j}^k)$ thus gives
$\tau(Y^{k+1})\leq2\tau(X^k)$.
Including the finite initial segment proves~\eqref{eq:ot-feasibility-summable}. 
This completes the proof.
\end{proof}

\subsection{KKT convergence and finite length}\label{sec:ot-kkt-convergence}
Proposition~\ref{prop:ot-feasibility-sum} gives
$\sum_k \tau(Y^k)<\infty$.
We now bound the marginal residuals near $P_*$ in terms of these entries. First, an elementary projection argument
bounds $\|Q-P_*\|_F$ by the marginal errors of $Q$ under the
conditions in the following lemma.
\begin{lemma}[Error bound for diagonal scalings of $P_*$]\label{lem:ot-support-recovery}
There are $\delta,L>0$ such that every positive diagonal
scaling $Q$ of $P_*$ on $S$ with $\|Q-P_*\|_F\leq\delta$ satisfies
\begin{equation}\label{eq:ot-support-recovery}
\|Q-P_*\|_F\leq L\big(\|Qe_n-a\|+\|Q^\top e_m-b\|\big).
\end{equation}
Here $Q$ is extended by zero outside the support set $S$.
\end{lemma}
\begin{proof}
For convenience, we let $P=P_*$ in this proof.
All Bregman distances in this proof are evaluated over the
entries indexed by $S$.
We first construct a feasible matrix that can be compared with $Q$.
On the Euclidean space of matrices supported in $S$, define the linear map
$AV\coleq(Ve_n,V^\top e_m)$, with the usual Euclidean norm on
$\mathbb R^m\times\mathbb R^n$, and let $A^\dagger$ be its Moore--Penrose inverse.
The correction $A^\dagger(AQ-AP)$ is supported in $S$.
Since $AQ-AP$ is in the range of $A$, the matrix
$\widehat Q\coleq Q-A^\dagger(AQ-AP)$ satisfies $A\widehat Q=AP$.
Moreover, $I-A^\dagger A$ is an orthogonal projector, so
\begin{equation*}
\|\widehat Q-P\|_F\leq\|Q-P\|_F,\qquad
\|\widehat Q-Q\|_F\leq\|A^\dagger\|\,\|AQ-AP\|.
\end{equation*}
Choose $\delta<\tfrac12\min_SP_{ij}$ and $\delta\leq1$.
If $\|Q-P\|_F\leq\delta$, these bounds show that both $Q$ and
$\widehat Q$ are positive on $S$. Thus $\widehat Q\in\Pi$ and
$\widehat Q|_{S^c}=0$.

We next use the entropy projection of $Q$ to bound $\|Q-P\|_F$.
Since $Q$ is a positive diagonal scaling of $P$, there exist vectors
$u\in\mathbb R^m$ and $v\in\mathbb R^n$ such that
$\log(P_{ij}/Q_{ij})=u_i+v_j$ on $S$.
For every $V\in\Pi$ supported in $S$,
\begin{equation*}
\sum_{(i,j)\in S}\log(P_{ij}/Q_{ij})(V_{ij}-P_{ij})
=u^\top(V-P)e_n+v^\top(V-P)^\top e_m=0.
\end{equation*}
This is the first-order optimality condition for $P$ as a minimizer of $\min\{D_H(V;Q)\mid V\in\Pi,\ V|_{S^c}=0\}$.
{The minimizing property of $P$} therefore gives $D_H(P;Q)\leq D_H(\widehat Q;Q)$.
Because $P_{ij}\leq1$ and $\delta\leq1$, every entry on the segment
between $P$ and $Q$ is at most two. The entropy second derivative
$1/x$ gives $D_H(P;Q)\geq\|P-Q\|_F^2/4$.
For the other divergence, $\log t\leq t-1$ and
$\min_SQ_{ij}\geq\tfrac12\min_SP_{ij}$ imply
\begin{equation*}
\frac14\|P-Q\|_F^2
\leq D_H(P;Q)\leq D_H(\widehat Q;Q)
\leq\frac{\|\widehat Q-Q\|_F^2}{\min_SQ_{ij}}
\leq\frac{2\|A^\dagger\|^2}{\min_SP_{ij}}\|AQ-AP\|^2.
\end{equation*}
Taking square roots and using
$\|AQ-AP\|\leq\|Qe_n-a\|+\|Q^\top e_m-b\|$
proves~\eqref{eq:ot-support-recovery}.
\end{proof}

The next lemma gives an upper bound for
$\|X^k-P_*\|_F+\|Y^k-P_*\|_F+\|\log(e_n+c_X^k)\|_b$.
It will give both summability in this subsection and the linear rate in
Section~\ref{sec:ot-strict-linear}.
\begin{lemma}[Primal and marginal error bound]\label{lem:ot-supported-estimates}
There are $h_1,K_1,K_2>0$ and $q\in(0,1)$, independent of $\rho$, such that
if~\eqref{eq:ot-captured-potential} holds with $h\leq h_1$, then, for all $k\geq N$,
\begin{equation}\label{eq:ot-supported-tail}
\|X^k-P_*\|_F+\|Y^k-P_*\|_F+\|\log(e_n+c_X^k)\|_b
\leq K_1q^{k-N}+K_2\sum_{j=N}^kq^{k-j}\tau(Y^j).
\end{equation}
\end{lemma}
\begin{proof}
Write $P\coleq P_*$. Throughout this proof, $k\geq N$ and
\eqref{eq:ot-captured-potential} holds with $h\leq h_1$.

\textup{(i)} We first bound the matrix errors by the marginal residuals
and the entries outside $S$. The inequality \eqref{eq-tmp3} 
in 
part~\textup{(ii)} of the proof of Lemma~\ref{prop:ot-entropy-identity}
bounds $\|X^k-P\|_F$ and $\|z^k-\beta_*/\rho\|_\infty$ by quantities
that tend to zero with $h_1$, uniformly in $\rho$.
Also $D_H(P;Y^k)\leq h_1$, and Pinsker's inequality gives
$\|Y^k-P\|_F\leq\|Y^k-P\|_1\leq\sqrt{2h_1}$.
We can therefore place both matrices $X^k,Y^k$ in a fixed neighborhood of $P$
by choosing $h_1$ before $\rho$.
By~\eqref{eq:ot-scaling-invariant} and
$C_{ij}=(\alpha_*)_i+(\beta_*)_j$ on $S$, the restriction $Y^k|_S$
is a positive diagonal scaling of $X^0|_S$. The row optimality
condition gives the same property for $X^k|_S$.
The entropy optimality condition in~\eqref{eq:ot-selected} also
expresses $P|_S$ as a positive diagonal scaling of $X^0|_S$.
Thus both iterate restrictions are positive diagonal scalings of $P$,
as required by Lemma~\ref{lem:ot-support-recovery}.

To bound the full matrix errors, we also estimate the entries outside $S$.
For sufficiently small $h_1$, closeness to $P$ ensures that
$\sum_{j:(i,j)\in S}Y^k_{ij}\geq a_i/2$ for every row $i$.
Since $\Gamma=0$ on $S$ and $\Gamma\geq0$ elsewhere,
\eqref{eq:ot-potential-update} gives
\begin{equation} \label{eq-tmp4}
\frac{X^k_{ij}}{Y^k_{ij}}
=\frac{e^{-\Gamma_{ij}/\rho+z_j^k-(\beta_*)_j/\rho}}
 {a_i^{-1}\sum_\ell Y^k_{i\ell}
 e^{-\Gamma_{i\ell}/\rho+z_\ell^k-(\beta_*)_\ell/\rho}}
\leq2e^{2\|z^k-\beta_*/\rho\|_\infty}\leq L_1 \quad \forall\; (i,j), 
\end{equation}
with $L_1$ independent of $\rho$. In particular,
$\tau(X^k)\leq L_1\tau(Y^k)$.
Removing the entries outside $S$ changes either marginal vector
by a vector of norm at most their sum. Thus
\begin{align*}
\|(Y^k|_S)e_n-a\|+\|(Y^k|_S)^\top e_m-b\|
&\leq\|a\|_\infty\|r_Y^k\|+2\tau(Y^k),\\
\|(X^k|_S)e_n-a\|+\|(X^k|_S)^\top e_m-b\|
&\leq\|b\|_\infty\|c_X^k\|+2\tau(X^k).
\end{align*}

Choose $h_1$ small enough to satisfy the neighborhood condition of
Lemma~\ref{lem:ot-support-recovery}. Apply that lemma to the two
restrictions, and then add the omitted entries. This gives $L_2>0$ such that
\begin{equation}\label{eq:ot-supported-reconstruction}
\|X^k-P\|_F+\|Y^k-P\|_F
\leq L_2\big(\|r_Y^k\|+\|c_X^k\|+\tau(Y^k)\big).
\end{equation}

\textup{(ii)} We now derive an estimate for the next pair of marginal residuals $(r^{k+1}_Y,c^{k+1}_X)$.

For this fixed iterate pair $(X^k,Y^k)$, let $H\coleq X^k|_S-P$ and
$G\coleq Y^k|_S-P$, extended by zero outside $S$.
These matrices depend on $k$ through $X^k$ and $Y^k$.
Note that $H= X^k-P - X^k|_{S^c}$ and $G = Y^k-P - Y^k|_{S^c}.$
By the residual definitions in
\eqref{eq:ot-marginal-residuals}, their marginal vectors satisfy
\begin{equation*}
\begin{alignedat}{2}
He_n&=-X^k|_{S^c}e_n,\qquad&
H^\top e_m&=b\odot c_X^k-(X^k|_{S^c})^\top e_m,\\
Ge_n&=a\odot r_Y^k-Y^k|_{S^c}e_n,\qquad&
G^\top e_m&=-(Y^k|_{S^c})^\top e_m.
\end{alignedat}
\end{equation*}
Here the restrictions to $S^c$ are extended by zero.
By \eqref{eq-tmp4},
the vectors $X^k|_{S^c}e_n$ and $(X^k|_{S^c})^\top e_m$ have norms
at most $L_1\tau(Y^k)$. The corresponding vectors formed from $Y^k$
have norms at most $\tau(Y^k)$.
On $S$, the entries of $P+G=Y^k|_S$ stay bounded away from zero, and
\begin{equation} \label{eq-tmp6}
\frac{(X^k_{ij})^2}{Y^k_{ij}} =  \frac{(P_{ij}+H_{ij})^2}{P_{ij}+G_{ij}}
=P_{ij}+2H_{ij}-G_{ij}
 +\frac{(H_{ij}-G_{ij})^2}{P_{ij}+G_{ij}}.
\end{equation}
Outside $S$, the ratio bound \eqref{eq-tmp4}  just proved gives
$\sum_{S^c}(X^k_{ij})^2/Y^k_{ij}\leq L_1^2\tau(Y^k)$.
On $S$, the remainder in the above displayed expansion is quadratic in $H-G$.

For the column update in~\eqref{eq:ot-two-primal-recurrence}, we have for small $\|c^k_X\|$ and $\|H\|_F$,
\begin{eqnarray*}
Y^{k+1}|_S=(P+H)\Diag((e_n+c_X^k)^{-1}) 
= P+H - P \Diag(c^k_X) + \Delta_Y,
\end{eqnarray*}
where $\|\Delta_Y\|_F=O(\|H\|_F^2+\|c_X^k\|^2)$.
Taking the row sums and dividing componentwise by $a$, we get
\begin{equation*}
\begin{aligned}
((Y^{k+1}|_S)e_n)\oslash a-e_m
&=(He_n-Pc_X^k)\oslash a+(\Delta_Ye_n)\oslash a\\
&=-M_Pc_X^k-(X^k|_{S^c}e_n)\oslash a+(\Delta_Ye_n)\oslash a,
\end{aligned}
\end{equation*}
where 
$\|(X^k|_{S^c}e_n)\oslash a\|=O(\tau(Y^k))$.
 Moreover, the column update gives
$\tau(Y^{k+1})\leq2\tau(X^k)\leq2L_1\tau(Y^k)$,
since $\|c_X^k\|_\infty\leq1/2$ in the chosen neighborhood.

For the row update in~\eqref{eq:ot-two-primal-recurrence}, expand
$(e_n+c_X^k)^{-(1+s)}=e_n-(1+s)c_X^k+\mathcal{O}(\|c_X^k\|^2)$. We get after using \eqref{eq-tmp6},
\begin{equation*}
\begin{aligned}
\big((X^k)^{\odot2}\oslash Y^k\big)|_S\Diag((e_n+c_X^k)^{-(1+s)})
&=P+2H-G-(1+s)P\Diag(c_X^k)\\
&\quad+O(\|H\|_F^2+\|G\|_F^2+\|c_X^k\|^2).
\end{aligned}
\end{equation*}
The row sums of the first-order term, 
$2H-G-(1+s)P\Diag(c_X^k)$, are
$-a\odot r_Y^k-(1+s)Pc_X^k$ up to an error bounded by $O(\tau(Y^k))$.
For differentiation at $P$, the algebraic formula for $\mathcal R_a$ is used
on the open set $\{W\in\R^{m\times n}:We_n>0\}$.
Since $Pe_n=a>0$, its denominators stay positive near $P$,
even when $P$ has zero entries, so the formula is smooth there.
The derivative of $\mathcal R_a$ at $P$, calculated in
Lemma~\ref{lem:ot-residual-one-step}, sends a direction $D$ to
$D-\Diag((De_n)\oslash a)P$.
Substituting $D=2H-G-(1+s)P\Diag(c_X^k)$ gives
\begin{equation*}
 X^{k+1}|_S = P + 2H-G-(1+s)P\Diag(c_X^k)
 +\Diag(r_Y^k+(1+s)M_Pc_X^k)P + \Delta_X
\end{equation*}
with the remainder $\Delta_X= 
O(\|H\|_F^2+\|G\|_F^2+\|c_X^k\|^2+\tau(Y^k))$.
Taking the column sums, dividing
by $b$, and substituting
the formulas for $H^\top e_m,G^\top e_m$ gives
\begin{eqnarray*}
\big((X^{k+1}|_S)^\top e_m\big)\oslash b - e_n&=&
2c_X^k-(1+s)c_X^k+N_P(r_Y^k+(1+s)M_Pc_X^k) + O(\|\Delta_X\|_F+\tau(Y^k))
\\
&=& N_Pr_Y^k+[(1-s)I+(1+s)N_PM_P]c_X^k + O(\|\Delta_X\|_F+\tau(Y^k)).
\end{eqnarray*}
Applying the bound in part~\textup{(i)} at iteration $k+1$ and using the estimate for $\tau(Y^{k+1})$ gives
\begin{equation*}
\tau(X^{k+1})\leq L_1\tau(Y^{k+1})\leq2L_1^2\tau(Y^k).
\end{equation*}

Thus 
\begin{eqnarray*}
  (r^{k+1}_Y,c^{k+1}_X) & =& \Big(((Y^{k+1}|_S)e_n)\oslash a -e_m, 
  ((X^{k+1}|_S)^\top e_m)\oslash b - e_n\Big)
  \\
  && + 
  \Big(((Y^{k+1}|_{S^c})e_n)\oslash a, 
  ((X^{k+1}|_{S^c})^\top e_m)\oslash b\Big) 
  \\
  &=& (-M_Pc^k_X, N_Pr_Y^k+[(1-s)I+(1+s)N_PM_P]c_X^k) + e^k,
\end{eqnarray*}
where $\|e^k\|\leq L_3'(\|\Delta_Y\|_F+\|\Delta_X\|_F+\tau(Y^k))$.
The constants in these remainder bounds are uniform on the fixed neighborhood chosen in part~\textup{(i)} and are independent of $\rho$.
By using 
\eqref{eq:ot-supported-reconstruction}
and the fact that $0\leq\tau(Y^k)\leq1$, we get
\begin{equation}\label{eq:ot-supported-expansion}
(r_Y^{k+1},c_X^{k+1})=B_P(r_Y^k,c_X^k)+e^k,\qquad
\|e^k\|\leq L_3\big(\|r_Y^k\|^2+\|c_X^k\|^2+\tau(Y^k)\big)
\end{equation}
for some $L_3>0$ independent of $\rho$.

\textup{(iii)} For each connected component, we subtract the
$a$-weighted mean from $r_Y^k$ and the $b$-weighted mean from $c_X^k$
so that the contraction
in Lemma~\ref{lem:ot-linear-contraction} applies.
The actual residuals need not satisfy the zero weighted sums in
\eqref{eq:ot-component-zero-sums}, but their failure to do so is controlled
by the entries outside $S$.
Let $I_d,J_d$ be the row and column index sets of the connected
components of the support graph of $P=P_*$, as in
Definition~\ref{def:ot-support-graph}. Lemma~\ref{lem:ot-marginal-sums} gives $\sum_{i\in I_d}a_i=\sum_{j\in J_d}b_j$.
Subtract these sums and use the exact column margins of $Y^k$ to obtain
\begin{equation*}
\sum_{i\in I_d}a_ir_{Y,i}^k
=\sum_{i\in I_d,j\notin J_d}Y^k_{ij}
 -\sum_{i\notin I_d,j\in J_d}Y^k_{ij}.
\end{equation*}
Both sets of entries are outside $S$, so $|\sum_{i\in I_d}a_ir_{Y,i}^k| \leq  \tau(Y^k)$. The same calculation using the exact row
margins of $X^k$ gives
$|\sum_{j\in J_d}b_jc_{X,j}^k|\leq\tau(X^k)$.

For the rest of this proof, 
define 
$d^k = (d^k_Y,d^k_X)$ by 
\begin{eqnarray*}
    (d^k_Y)_i &=& (r^k_Y)_i - \sum_{l\in I_d}a_lr_{Y,l}^k\Big/\sum_{l\in I_d}a_l, \quad \forall\; i\in I_d
    \\
    (d^k_X)_j &=& (c^k_X)_j - \sum_{l\in J_d}b_lc_{X,l}^k\Big/\sum_{l\in J_d}b_l, \quad \forall\; j\in J_d.
\end{eqnarray*}
The denominators are fixed and positive. Hence $d^k$ satisfies
\eqref{eq:ot-component-zero-sums} and, for some $L_4>0$,
$\|(r_Y^k,c_X^k)-d^k\|\leq L_4\tau(Y^k)$.

We now derive a recurrence for $d^k$.
Equation~\eqref{eq:ot-supported-expansion} gives
\begin{equation*}
d^{k+1}=B_Pd^k+\widetilde e^k,
\end{equation*}
where
\begin{equation*}
\widetilde e^k\coleq e^k+\big[d^{k+1}-(r_Y^{k+1},c_X^{k+1})\big]+B_P\big[(r_Y^k,c_X^k)-d^k\big].
\end{equation*}
The bound on the removed means and
$\tau(Y^{k+1})\leq2L_1\tau(Y^k)$ from part~\textup{(ii)} imply
\begin{equation*}
\|\widetilde e^k\|\leq\|e^k\|+L_4(2L_1+\|B_P\|)\tau(Y^k),
\end{equation*}
where $\|B_P\|$ is the Euclidean operator norm. Moreover,
\begin{equation*}
\|r_Y^k\|^2+\|c_X^k\|^2\leq2\|d^k\|^2+2L_4^2\tau(Y^k)^2.
\end{equation*}
Both $d^{k+1}$ and $B_Pd^k$ satisfy the equations in
\eqref{eq:ot-component-zero-sums}, so their difference $\widetilde e^k$ does as well.
Lemma~\ref{lem:ot-linear-contraction}(ii), with $K=\{P\}$,
gives fixed constants $m_P,\kappa_P>0$ such that
$m_P\|d\|\leq\|d\|_P\leq\kappa_P\|d\|$ for every $d$ satisfying
\eqref{eq:ot-component-zero-sums}.
Combining these bounds with~\eqref{eq:ot-supported-expansion} and
$\tau(Y^k)^2\leq\tau(Y^k)$ gives a constant $L_5>0$, independent of $\rho$, such that
\begin{equation*}
\|\widetilde e^k\|_P\leq\kappa_P\|\widetilde e^k\|\leq L_5\big(\|d^k\|_P^2+\tau(Y^k)\big).
\end{equation*}
Applying the contraction in Lemma~\ref{lem:ot-linear-contraction}(i) to $B_P(d^k)$, there exists $q_P\in (0,1)$ such that for all $k\geq N$,
\begin{equation*}
\|d^{k+1}\|_P \leq \|B_P(d^k)\|_P + \|\widetilde{e}^k\|_P \leq q_P\|d^k\|_P
 +L_5\big(\|d^k\|_P^2+\tau(Y^k)\big).
\end{equation*}

Fix $q\in(q_P,1)$.
By part~\textup{(i)} and the definitions of the marginal residuals and $d^k$,
there exists $L_6>0$, independent of $\rho$, $k$, and sufficiently small $h_1$, such that
\begin{equation*}
\|d^k\|_P\leq L_6\sqrt{h_1}\qquad\text{for all }k\geq N.
\end{equation*}
Reduce $h_1$, if necessary, so that $L_5L_6\sqrt{h_1}\leq q-q_P$.
Then $q_P+L_5\|d^k\|_P\leq q$ for all $k\geq N$, and hence
\begin{equation}\label{eq:ot-supported-forcing}
\|d^{k+1}\|_P\leq q\|d^k\|_P+L_5\tau(Y^k).
\end{equation}

\textup{(iv)} Finally, we recover the bound on the original iterates.
Since part~\textup{(iii)} applies at every $k\geq N$ under
\eqref{eq:ot-captured-potential}, iterating~\eqref{eq:ot-supported-forcing} gives
\begin{equation*}
\|d^k\|_P\leq q^{k-N}\|d^N\|_P
 +L_5\sum_{j=N}^{k-1}q^{k-1-j}\tau(Y^j),\qquad k\geq N.
\end{equation*}
The initial norm is uniformly bounded by the chosen neighborhood.
The bound on the removed means and $m_P\|d^k\|\leq\|d^k\|_P$ give
\begin{equation*}
\|r_Y^k\|+\|c_X^k\|
\leq\sqrt{2}\big(m_P^{-1}\|d^k\|_P+L_4\tau(Y^k)\big).
\end{equation*}
Substitute this inequality into~\eqref{eq:ot-supported-reconstruction}.
Since $q>0$, the sum up to $k-1$ is bounded by $q^{-1}$ times
the sum in~\eqref{eq:ot-supported-tail}, which also includes
$\tau(Y^k)$.
Finally, $\|c_X^k\|_\infty\leq1/2$ in this neighborhood implies
$\|\log(e_n+c_X^k)\|_b\leq2\|c_X^k\|_b$.
These bounds prove~\eqref{eq:ot-supported-tail} with constants
$K_1,K_2$ independent of $\rho$. 
\end{proof}

For the remaining results in Section~\ref{sec:ot-convergence}, consider
the sequence generated by Algorithm~\ref{alg:ot} with
\begin{equation}\label{eq:ot-convergence-regime}
X^0\in\operatorname{ri}\Pi,\qquad Z^0=0,\qquad s\in(0,2),\qquad
\rho\geq\rho_0(C,a,b,X^0,s,h),
\end{equation}
where $\rho$ is fixed, $h>0$ is fixed below the thresholds of
Lemmas~\ref{prop:ot-entropy-identity} and~\ref{lem:ot-supported-estimates},
and $\rho_0$ is given by Proposition~\ref{prop:ot-feasibility-sum}.
Then~\eqref{eq:ot-captured-potential} and~\eqref{eq:ot-feasibility-summable}
both hold.

\begin{proposition}[Summability and convergence of the iterates]\label{lem:ot-primal-convergence}
The sequence $\{(X^k,Y^k,Z^k)\}$ generated by Algorithm~\ref{alg:ot} under
\eqref{eq:ot-convergence-regime} satisfies
\begin{equation}\label{eq:ot-summable-residuals}
\sum_{k=1}^{\infty}\big(\|X^k-P_*\|_F+\|Y^k-P_*\|_F
 +\|z^{k+1}-z^k\|_b\big)<\infty.
\end{equation}
Consequently, $X^k,Y^k\to P_*$ and $z^k$ converges to a vector $z_\infty$.
\end{proposition}
\begin{proof}
We sum~\eqref{eq:ot-supported-tail} and then use the multiplier update.
Under~\eqref{eq:ot-convergence-regime}, Proposition~\ref{prop:ot-feasibility-sum}
gives $\sum_{j=N}^\infty\tau(Y^j)<\infty$.
All terms in the summation in~\eqref{eq:ot-supported-tail} are
nonnegative, so changing the order of summation below gives
\begin{equation*}
\sum_{k=N}^{\infty}\sum_{j=N}^kq^{k-j}\tau(Y^j)
=\sum_{j=N}^{\infty}\tau(Y^j)\sum_{k=j}^\infty q^{k-j}
=\frac1{1-q}\sum_{j=N}^{\infty}\tau(Y^j)<\infty.
\end{equation*}
The geometric term in~\eqref{eq:ot-supported-tail} is also summable.
Summing that estimate and substituting
$\|z^{k+1}-z^k\|_b=s\|\log(e_n+c_X^k)\|_b$ from
\eqref{eq:ot-residual-update} proves~\eqref{eq:ot-summable-residuals},
after including the finite initial segment.
The primal errors consequently tend to zero.
Moreover, for every $\ell>k$,
\begin{equation*}
\|z^\ell-z^k\|_b
\leq\sum_{j=k}^{\ell-1}\|z^{j+1}-z^j\|_b
\leq\sum_{j=k}^\infty\|z^{j+1}-z^j\|_b\longrightarrow0
\qquad(k\to\infty).
\end{equation*}
The last bound is independent of $\ell$.
Thus $\{z^k\}$ is a Cauchy sequence and has a limit $z_\infty$.
\end{proof}

The following proposition identifies the dual limit and completes the
finite-length result in Theorem~\ref{thm:ot-global}.

\begin{proposition}[KKT conditions at the limit and finite length]\label{lemma:ot-feasibility-kkt}
The sequence $\{(X^k,Y^k,Z^k)\}$ generated by Algorithm~\ref{alg:ot}
under~\eqref{eq:ot-convergence-regime}
converges to a KKT solution $(P_*,P_*,Z_\infty)$ of~\eqref{eq:ot-split}.
The sequence has finite length, i.e.,
$\sum_{k=1}^{\infty}\big(\|X^{k+1}-X^k\|_F+\|Y^{k+1}-Y^k\|_F
+\|Z^{k+1}-Z^k\|_F\big)<\infty$.
\end{proposition}
\begin{proof}
Write $P\coleq P_*$. Proposition~\ref{lem:ot-primal-convergence} gives
$X^k,Y^k\to P$ and $z^k\to z_\infty$.

\textup{(i)} We first identify the limits of the subproblem multipliers.
Each row of $P$ has a positive entry because $Pe_n=a>0$.
Thus every component of $(P\odot e^{-C/\rho})e^{z_\infty}$ is
positive. Exponentiating the row condition in~\eqref{eq:sub-kkt} and using
$X^ke_n=a$ gives
$e^{-r^k/\rho}\odot((Y^k\odot e^{-C/\rho})e^{z^k})=a$.
Solving for $r^k$ and taking limits yields
\begin{equation}\label{eq-alpha}
\begin{gathered}
r^k=\rho\log\!\left(((Y^k\odot e^{-C/\rho})e^{z^k})\oslash a\right)
\longrightarrow-\alpha_\infty,\\
\alpha_\infty\coleq\rho\log\!\left(a\oslash((P\odot e^{-C/\rho})e^{z_\infty})\right).
\end{gathered}
\end{equation}
Set $\beta_\infty\coleq\rho z_\infty$.
Since $X^k\to P\in\Pi$, the column residual $c_X^k$ tends to zero.
The column optimality condition in~\eqref{eq:sub-kkt} therefore gives
$c^k=\rho\log(e_n+c_X^k)-\rho z^k\to-\beta_\infty$.

\textup{(ii)} We next verify that these limits form an OT dual optimal solution.
For $(i,j)\in S$, positivity of $P_{ij}$ gives
$X^k_{ij}/Y^k_{ij}\to1$.
Substituting the limits of $r^k,z^k$ into the row optimality
condition gives $C_{ij}=(\alpha_\infty)_i+(\beta_\infty)_j$ on $S$.
For every entry, substituting the limits of $r^k,c^k$ into
\eqref{eq:ot-log-telescope} gives
\begin{equation}\label{eq:ot-actual-ratio}
\frac{Y^{k+1}_{ij}}{Y^k_{ij}}
\longrightarrow e^{-(\Gamma_\infty)_{ij}/\rho},
\qquad \Gamma_\infty\coleq C-\alpha_\infty e_n^\top-e_m\beta_\infty^\top.
\end{equation}
If $(\Gamma_\infty)_{ij}<0$ outside $S$, the limiting ratio exceeds
one. There would then be $\epsilon>0$ and an index $K$ such that
$Y^{k+1}_{ij}\geq(1+\epsilon)Y^k_{ij}$ for all $k\geq K$.
Iterating gives $Y^k_{ij}\geq(1+\epsilon)^{k-K}Y^K_{ij}$.
Since $Y^K_{ij}>0$, this contradicts $Y^k_{ij}\to P_{ij}=0$.
Consequently $\Gamma_\infty\geq0$ and $P\odot\Gamma_\infty=0$.
These are the KKT conditions for the OT problem \eqref{eq:gw}, and
$Z_\infty=-\rho e_mz_\infty^\top=-e_m\beta_\infty^\top$
gives the corresponding KKT solution of~\eqref{eq:ot-split}.

\textup{(iii)} Finally, we prove that the full sequence has finite length.
The summable primal errors bound the successive differences: for example,
$\|X^{k+1}-X^k\|_F\leq\|X^{k+1}-P\|_F+\|X^k-P\|_F$,
and the same inequality holds for $Y$.
The multiplier relation in Proposition~\ref{prop:ot-equivalent-updates}(ii) gives
\begin{equation*}
\|Z^{k+1}-Z^k\|_F
=\rho\sqrt m\,\|z^{k+1}-z^k\|
\leq\rho\sqrt{\frac{m}{\min_j b_j}}\,\|z^{k+1}-z^k\|_b.
\end{equation*}
Since $\rho$ is fixed, summing these bounds and applying
\eqref{eq:ot-summable-residuals} proves the finite length result.
\end{proof}

\subsection{Strict complementarity and linear convergence}\label{sec:ot-strict-linear}
The limiting dual solution need not coincide with the strictly
complementary dual optimal solution used in the descent estimate. We next prove that its own slack is
positive outside $S$. This will make $\tau(Y^k)$ decay geometrically,
so the same residual estimate~\eqref{eq:ot-supported-tail} gives a linear rate.

\begin{lemma}[Strict complementarity of the limit]\label{lem:ot-selected-strict}
For Algorithm~\ref{alg:ot} under~\eqref{eq:ot-convergence-regime}, the limiting slack
$\Gamma_\infty$ in~\eqref{eq:ot-actual-ratio} is positive on $S^c$.
\end{lemma}
\begin{proof}
 Write $P\coleq P_*$ and
$\mathcal E_k\coleq\mathcal E_P(Y^k,z^k;z_\infty)$.
Proposition~\ref{lemma:ot-feasibility-kkt} gives
$(Y^k,z^k)\to(P,z_\infty)$; since $P$ is positive on $S$,
$D_H(P;Y^k)\to0$ and hence $\mathcal E_k\to0$.

We argue by contradiction.
Suppose $S_0\coleq\{(i,j)\in S^c:(\Gamma_\infty)_{ij}=0\}$ is nonempty.
Set $S_+\coleq\{(i,j):(\Gamma_\infty)_{ij}>0\}$,
$\tau_k^0\coleq\sum_{S_0}Y^k_{ij}$ and
$\tau_k^+\coleq\sum_{S_+}Y^k_{ij}$.

\textup{(i)} We first compare the entries with zero and positive limiting slack.
On $S_0$, every ratio in~\eqref{eq:ot-actual-ratio} tends to one.
The ratio $\tau_{k+1}^0/\tau_k^0$ lies between the smallest and
largest of these entrywise ratios, so it also tends to one.
On the finite set $S_+$, the limiting ratios are strictly below
one. Choose $q_1\in(0,1)$ larger than all of them. Then, for all sufficiently
large $k$, we have $Y^{k+1}_{ij}\leq q_1Y^k_{ij}$ for all $(i,j)\in S_+$.
If $S_+$ is empty, choose any $q_1\in(0,1)$ instead.
Thus
\begin{equation*}
\tau_k^0>0,\qquad \tau_{k+1}^0/\tau_k^0\to1,\qquad
\tau_k^+=\mathcal{O}(q_1^k),\qquad \sum_k\tau_k^0<\infty,
\end{equation*}
where summability follows from~\eqref{eq:ot-feasibility-summable}.
In particular,
\begin{equation*}
\frac{q_1^{k+1}/(\tau_{k+1}^0)^2}{q_1^k/(\tau_k^0)^2}
=q_1\left(\frac{\tau_k^0}{\tau_{k+1}^0}\right)^2\longrightarrow q_1<1.
\end{equation*}
Hence $q_1^k/(\tau_k^0)^2\to0$ and
$\tau_k^+=o((\tau_k^0)^2)$.

\textup{(ii)} We next bound the primal errors and compare the potential with $\tau_k^0$.
Take $q$ from Lemma~\ref{lem:ot-supported-estimates} and fix
$\theta\in(q,1)$.
Enlarge $N$ so that $\tau_{k+1}^0\geq\theta\tau_k^0$ and
$\tau_k^+\leq\tau_k^0$ for every $k\geq N$.
The recurrence~\eqref{eq:ot-supported-forcing} can be started at this
later index, so~\eqref{eq:ot-supported-tail} remains valid there,
with its constants increased if necessary.
For $N\leq j\leq k$, iterating
$\tau_{\ell+1}^0\geq\theta\tau_\ell^0$ from $\ell=j$ to $k-1$ gives
$\tau_j^0\leq\theta^{-(k-j)}\tau_k^0$. Hence
\begin{equation*}
\sum_{j=N}^kq^{k-j}\tau_j^0
\leq\tau_k^0\sum_{j=N}^k(q/\theta)^{k-j}
\leq\frac{\tau_k^0}{1-q/\theta},\qquad
\frac{q^{k-N}}{\tau_k^0}
\leq\frac{(q/\theta)^{k-N}}{\tau_N^0}\longrightarrow0.
\end{equation*}
Since $\tau(Y^j)=\tau_j^0+\tau_j^+\leq2\tau_j^0$ for $j\geq N$,
substitution into~\eqref{eq:ot-supported-tail} yields
\begin{equation}\label{eq:ot-neutral-normal-bound}
\|X^k-P\|_F+\|Y^k-P\|_F+\|\log(e_n+c_X^k)\|_b
=\mathcal{O}(\tau_k^0).
\end{equation}

We also need a lower bound for the potential. Both $P$ and $Y^k$ have entry sum one, and
$\|Y^k|_S\|_1=1-\tau_k^0-\tau_k^+>0$.
The normalized restriction $Y^k|_S/\|Y^k|_S\|_1$ also has entry
sum one, so its relative entropy from $P|_S$ is nonnegative.
Using $-\log(1-u)\geq u$ gives
\begin{equation}\label{eq:ot-neutral-energy-lower}
\mathcal E_k\geq D_H(P;Y^k)
=\sum_{(i,j)\in S} P_{ij}(\log P_{ij} - \log Y^k_{ij})
=D_H\!\left(P|_S;\frac{Y^k|_S}{\|Y^k|_S\|_1}\right)
 -\log\|Y^k|_S\|_1
\geq\tau_k^0>0.
\end{equation}

\textup{(iii)} We now give an upper bound on how much the potential can decrease.
Apply~\eqref{eq:ot-radial-identity} with respect to the actual KKT solution
$(P,\alpha_\infty,\beta_\infty)$, with center
$\beta_\infty/\rho=z_\infty$. We will prove that 
each term of the upper bound for the potential difference is bounded by a constant
times $\tau_k^0\mathcal E_k$.

The row multiplier is given by
$r^k=\rho\log(((Y^k\odot e^{-C/\rho})e^{z^k})\oslash a)$,
as derived from~\eqref{eq:sub-kkt} in
Proposition~\ref{lemma:ot-feasibility-kkt}.
Its logarithm has a positive argument at $(P,z_\infty)$,
so it is continuously differentiable in a neighborhood of this point.
For the fixed $\rho$, its derivatives are bounded on a sufficiently
small closed neighborhood. Together with~\eqref{eq:ot-neutral-normal-bound} and the expression of $\alpha_\infty$ in \eqref{eq-alpha},
this gives
$\|r^k+\alpha_\infty\|=\mathcal{O}(\|z^k-z_\infty\|_b+\tau_k^0)$,
where the constants in this part may depend on the fixed $\rho$.
The row optimality condition~\eqref{eq:sub-kkt} can be written as
\begin{equation*}
\log\frac{X^k_{ij}}{Y^k_{ij}}
=-\frac{(\Gamma_\infty)_{ij}}{\rho}
 +z_j^k-(z_\infty)_j-\frac{r_i^k+(\alpha_\infty)_i}{\rho}.
\end{equation*}
On $S$, the entries of $X^k,Y^k$ tend to the positive entries of
$P$, so their logarithmic ratio is bounded by a constant times
$\|X^k-P\|_F+\|Y^k-P\|_F$.
On $S_0$, the slack term vanishes, and on $S_+$ it is fixed.
Substituting the preceding bounds therefore gives
\begin{equation}\label{eq-tmp5}
\log\frac{X^k_{ij}}{Y^k_{ij}}=
\begin{cases}
\mathcal{O}(\tau_k^0),&(i,j)\in S,\\
\mathcal{O}(\|z^k-z_\infty\|_b+\tau_k^0),&(i,j)\in S_0,\\
\mathcal{O}(1),&(i,j)\in S_+.
\end{cases}
\end{equation}
All logarithmic ratios thus lie in one bounded interval.
On this interval, $0\leq e^tt-e^t+1\leq L_1t^2$.
We also have $\mathcal E_k\geq\tau_k^0$ by
\eqref{eq:ot-neutral-energy-lower},
$\|z^k-z_\infty\|_b^2\leq2s\mathcal E_k$, and
$\tau_k^+=o((\tau_k^0)^2)$.
Applying the scalar bound with $t=\log(X^k_{ij}/Y^k_{ij})$, weighted
by $Y^k_{ij}$, and then using these three inequalities in \eqref{eq-tmp5}
yields
\begin{equation*}
\begin{aligned}
D_H(X^k;Y^k)
&\leq L_1\sum_{ij}Y^k_{ij}\left(\log\frac{X^k_{ij}}{Y^k_{ij}}\right)^2
\;\leq\; L_2\big((\tau_k^0)^2+\tau_k^0\|z^k-z_\infty\|_b^2+\tau_k^+\big)
\leq L_3\tau_k^0\mathcal E_k.
\end{aligned}
\end{equation*}
The bounded logarithmic ratios also bound $X^k_{ij}/Y^k_{ij}$.
Since $\Gamma_\infty$ is zero off $S_+$, it follows that
$\rho^{-1}\langle\Gamma_\infty,X^k\rangle
=\mathcal{O}(\tau_k^+)=\mathcal{O}(\tau_k^0\mathcal E_k)$.

It remains to bound the logarithmic terms in the potential difference.
For $c$ near zero, $|c|\leq2|\log(1+c)|$ and
$|c-\log(1+c)|\leq L_4c^2$.
By~\eqref{eq:ot-neutral-normal-bound}, each component of $c_X^k$
and its logarithm is $\mathcal{O}(\tau_k^0)$.
Also $z^k-z_\infty$ is bounded. These facts give
\begin{equation*}
\|\log(e_n+c_X^k)\|_b^2
+\big|\langle c_X^k-\log(e_n+c_X^k),
 z^k-z_\infty-\log(e_n+c_X^k)\rangle_b\big|
=\mathcal{O}((\tau_k^0)^2)=\mathcal{O}(\tau_k^0\mathcal E_k).
\end{equation*}
Equation~\eqref{eq:ot-entropy-remainder} also gives
$0\leq D_H(X^k;Y^k)-D_H((X^k)^\top e_m;b)\leq D_H(X^k;Y^k)$.
Moving the terms in~\eqref{eq:ot-radial-identity} to the other side
and using $0<s<2$, we therefore obtain
\begin{equation*}
\begin{aligned}
\mathcal E_k-\mathcal E_{k+1}
&\leq\rho^{-1}\langle\Gamma_\infty,X^k\rangle+D_H(X^k;Y^k)
 +(1-s/2)\|\log(e_n+c_X^k)\|_b^2\\
&\quad+\big|\langle c_X^k-\log(e_n+c_X^k),
 z^k-z_\infty-\log(e_n+c_X^k)\rangle_b\big|\\
&\leq L_5\tau_k^0\mathcal E_k
\end{aligned}
\end{equation*}
for some $L_5>0$ and every sufficiently large $k$.
Equivalently,
\begin{equation}\label{eq:ot-neutral-product-bound}
\mathcal E_{k+1}\geq(1-L_5\tau_k^0)\mathcal E_k.
\end{equation}

\textup{(iv)} Finally, we show that the inequality
\eqref{eq:ot-neutral-product-bound}
contradicts the convergence of the potential to zero.
Since $\sum_k\tau_k^0<\infty$, enlarge $N$ so that
$L_5\tau_k^0\leq1/2$ for $k\geq N$.
Iterating~\eqref{eq:ot-neutral-product-bound} and using
$\log(1-u)\geq-2u$ for $0\leq u\leq1/2$, we obtain
\begin{equation*}
\mathcal E_{k+1}\geq\mathcal E_N\prod_{j=N}^k(1-L_5\tau_j^0)
\geq\mathcal E_N\exp\!\left(-2L_5\sum_{j=N}^k\tau_j^0\right).
\end{equation*}
Here $\mathcal E_N>0$ by~\eqref{eq:ot-neutral-energy-lower}, and
$\sum_j\tau_j^0<\infty$. The right-hand side is consequently
bounded away from zero, contradicting $\mathcal E_k\to0$.
Thus $S_0=\varnothing$.
\end{proof}

\begin{proposition}[Linear convergence]\label{lem:ot-linear-length}
Suppose Algorithm~\ref{alg:ot} is run under~\eqref{eq:ot-convergence-regime}.
Then its sequence $\{(X^k,Y^k,Z^k)\}$ converges
$R$-linearly to $(P_*,P_*,Z_\infty)$.
\end{proposition}
\begin{proof}
We first obtain a geometric bound for the entries outside $S$.
If $S^c\ne\varnothing$, Lemma~\ref{lem:ot-selected-strict} makes
all limiting ratios in~\eqref{eq:ot-actual-ratio} strictly less
than one there. Choose $q_1\in(0,1)$ larger than their maximum.
Iterating the eventual entrywise ratio bounds gives
$\tau(Y^k)=\mathcal{O}(q_1^k)$. If $S^c=\varnothing$, this estimate holds
with any $q_1\in(0,1)$ because $\tau(Y^k)=0$.

We use this bound in~\eqref{eq:ot-supported-tail} to estimate the primal errors.
For $\zeta\in(\max\{q,q_1\},1)$, we have
\begin{equation*}
\sum_{j=N}^k q^{k-j}q_1^j
=\zeta^k\sum_{j=N}^k(q/\zeta)^{k-j}(q_1/\zeta)^j
\leq\frac{\zeta^k}{1-q/\zeta}.
\end{equation*}
The initial term $q^{k-N}$ is also $\mathcal{O}(\zeta^k)$, so
\eqref{eq:ot-supported-tail} gives
\begin{equation*}
\|X^k-P_*\|_F+\|Y^k-P_*\|_F+\|\log(e_n+c_X^k)\|_b=\mathcal{O}(\zeta^k).
\end{equation*}
Finally, $z^k\to z_\infty$ and~\eqref{eq:ot-residual-update}
express its error as a tail of the summable increments. Hence
\begin{equation*}
\|z^k-z_\infty\|_b
\leq s\sum_{j=k}^{\infty}\|\log(e_n+c_X^j)\|_b
=\mathcal{O}\!\left(\sum_{j=k}^{\infty}\zeta^j\right)=\mathcal{O}(\zeta^k).
\end{equation*}
The exact matrix multiplier relation also gives
\begin{equation*}
\|Z^k-Z_\infty\|_F
=\rho\sqrt m\,\|z^k-z_\infty\|
\leq\rho\sqrt{\frac{m}{\min_jb_j}}\,\|z^k-z_\infty\|_b=\mathcal{O}(\zeta^k).
\end{equation*}
Together with the primal bounds, this proves $R$-linear convergence
of the sequence.
\end{proof}

\begin{proof}[{\bf Proof of Theorem~\ref{thm:ot-global}}]
Choose $h$ below the thresholds in Lemmas~\ref{prop:ot-entropy-identity}
and~\ref{lem:ot-supported-estimates}, then choose $\rho_0$
as in Proposition~\ref{prop:ot-feasibility-sum}.
For every fixed $\rho\geq\rho_0$,
Proposition~\ref{lemma:ot-feasibility-kkt} proves finite length
in~\textup{(i)} and KKT convergence in~\textup{(ii)}.
Lemma~\ref{lem:ot-selected-strict} establishes strict complementarity,
completing~\textup{(ii)}.
Proposition~\ref{lem:ot-linear-length} proves~\textup{(iii)}.
\end{proof}

\section{Efficient implementations of BADMM}\label{sec:ot-implementation}
We reduce the computation time and memory required by BADMM
by updating vectors instead of matrices and using grid structure
to compute matrix-vector products.

\subsection{Implementation using scaling vectors}\label{sec:ot-fast-computation}
We implement Algorithm~\ref{alg:ot} using three vectors
$u^k,v^k,z^k$, without storing $X^k,Y^k,Z^k$ explicitly.
We assume $X^0>0$, $X^0e_n=a$ and $Z^0=0$, with fixed $\rho>0$.
Define the kernel matrix $G_\tau\coleq X^0\odot e^{-\tau C}$.
The positive vectors $u^k\in\R^m$ and $v^k\in\R^n$ multiply
the rows and columns of $G_{(k-1)/\rho}$, respectively, to represent
$Y^k$, following the notation in~\cite[Section~4.2]{yang2022bregman}.
The vector $z^k\in\R^n$ determines $Z^k$ as in~\eqref{eq:ot-recovery-z}.
For general $X^0$ and $C$, storing the input matrices and computing a
matrix--vector product still require $\mathcal{O}(mn)$ memory and work,
respectively.
Since $Y^1=\mathcal C_b(X^0)$, the initial factors are
$u^1=e_m$ and $v^1=b\oslash((X^0)^\top e_m)$.
The first multiplier update gives $z^1=s\log v^1$.
If $X^0\in\Pi$, these simplify to $v^1=e_n$ and $z^1=0$.
For $k\geq1$, compute
\begin{subequations}\label{eq:ot-factor-updates}
\begin{align}
u^{k+1}&=a\oslash\big(G_{k/\rho}(v^k\odot e^{z^k})\big),\label{eq:ot-factor-update-u}\\
v^{k+1}&=b\oslash\big(G_{k/\rho}^\top u^{k+1}\big),\label{eq:ot-factor-update-v}\\
z^{k+1}&=(1-s)z^k+s\log(v^{k+1}\oslash v^k).\label{eq:ot-factor-update-z}
\end{align}
\end{subequations}
The matrices in Algorithm~\ref{alg:ot} can be computed from
these vectors by
{
\begin{subequations}\label{eq:ot-row-factorization}
\begin{align}
X^k&=\Diag(u^{k+1})G_{k/\rho}\Diag(v^k\odot e^{z^k}),\label{eq:ot-recovery-x}\\
Y^k&=\Diag(u^k)G_{(k-1)/\rho}\Diag(v^k),\label{eq:ot-recovery-y}\\
Z^k&=-\rho e_m(z^k)^\top.\label{eq:ot-recovery-z}
\end{align}
\end{subequations}}
To check the stopping criterion, we compute the OT dual variables as follows:
\begin{equation}\label{eq:ot-scaling-duals}
\alpha^{k+1}\coleq\rho\log(u^{k+1}\oslash u^k),\qquad
\beta^{k+1}\coleq\rho\log(v^{k+1}\oslash v^k).
\end{equation}
The subproblem optimality conditions in~\eqref{eq:sub-kkt} give
$\alpha^{k+1}=-r^k$ and $\beta^{k+1}=-c^k$.
The stopping test pairs these dual variables with $Y^{k+1}$.
The implementation uses logarithms of the scaling factors for
numerical stability, as described in Section~\ref{sec:ot-kernel-truncation}.

\begin{remark}[Nonzero initial multipliers]
The vector updates also apply when the initial multiplier has the form
$Z^0=-\rho e_m(z^0)^\top$, with fixed $\rho$ and
$z^1=z^0+s\log v^1$.
For an arbitrary matrix $Z^0$, the multiplier update instead gives
$Z^k=(1-s)^kZ^0-\rho e_m(z^k)^\top$ for some $z^k$.
The term $(1-s)^kZ^0$ must then be included in the matrix-vector products,
so~\eqref{eq:ot-factor-updates} cannot be used without modification.
\end{remark}

\subsection{Structured matrix-vector products and stopping tests on grids}\label{sec:ot-structured-products}
We use the grid structure to speed up matrix-vector products and residual
evaluation without forming full $m\times n$ matrices. Suppose that the
source grid contains all pairs 
$\{(x^{(1)}_{i_1},x^{(2)}_{i_2})\mid 1\leq i_1\leq L_1^s, 1\leq i_2\leq L_2^s\}$
and the target grid contains all pairs $\{(y^{(1)}_{j_1},y^{(2)}_{j_2}) \mid 1\leq j_1\leq L_1^t, 1\leq j_2\leq L_2^t\}$.
The entries of $a$ and $b$
give the weights at the source and target grid points, respectively.
This includes transport between images with squared Euclidean costs.
With $X^0=ab^\top$, suppose that
$C_{(i_1,i_2),(j_1,j_2)}=D^{(1)}_{i_1j_1}+D^{(2)}_{i_2j_2}$,
where $D^{(1)}\in \R^{L_1^s\times L_1^t}$ and $D^{(2)}\in\R^{L_2^s\times L_2^t}$ contain the costs along the two coordinates.
The kernel $G_\tau$ then factors as
\begin{equation}\label{eq:ot-separable-kernel}
G_\tau=\Diag(a)
\big(G^{(2)}_\tau\otimes G^{(1)}_\tau\big)\Diag(b),
\qquad G^{(r)}_\tau\coleq\exp(-\tau D^{(r)}), \quad r=1,2.
\end{equation}
Arrange a vector on the target grid as a matrix $V\in \R^{L_1^t\times L_2^t}$, with the
first coordinate indexing rows and the second indexing columns.
Stacking the columns of $V$ gives $\operatorname{vec}(V)$, and
\begin{equation}\label{eq:ot-kronecker-product}
\big(G^{(2)}_\tau\otimes G^{(1)}_\tau\big)\operatorname{vec}(V)
=\operatorname{vec}\big(G^{(1)}_\tau V(G^{(2)}_\tau)^\top\big).
\end{equation}
Formula~\eqref{eq:ot-kronecker-product} replaces one large
matrix--vector product with two smaller matrix multiplications.
Multiplication by $\Diag(a)$ and $\Diag(b)$ only requires multiplying
vector entries by the weights $a$ and $b$.
Thus for two $L\times L$ grids, each kernel product costs $\mathcal{O}(L^3)$
instead of $\mathcal{O}(L^4)$. 
This method is standard in entropy-regularized
transport~\cite{peyre2019computational}.
Note that for a general $X^0$, the kernel need not have the form
in~\eqref{eq:ot-separable-kernel}.

The formula~\eqref{eq:ot-separable-kernel}
is applicable to squared Euclidean costs,
Manhattan costs, and other sums of coordinate costs. However, 
Euclidean distance without the square does not in general split this way.
If grid points with zero weights have been removed, we fill their positions
with zeros before multiplication and keep only the remaining positions afterward.

To check the stopping condition without forming $Y^k$,
we compute its row and column sums and objective value from the scaling vectors. Write
$Y^k=\Diag(u^k)G_{\tau_k}\Diag(v^k)$, where $\tau_k\coleq(k-1)/\rho$. Then
\begin{equation*}
Y^ke_n=u^k\odot(G_{\tau_k}v^k),\quad
(Y^k)^\top e_m=v^k\odot(G_{\tau_k}^\top u^k),\quad
\langle C,Y^k\rangle=(u^k)^\top(C\odot G_{\tau_k})v^k.
\end{equation*}
For the additive cost above, the weighted kernel satisfies
\begin{equation*}
C\odot G_\tau
=
\Diag(a)\big[(D^{(2)}\odot G^{(2)}_\tau)\otimes G^{(1)}_\tau
+G^{(2)}_\tau\otimes(D^{(1)}\odot G^{(1)}_\tau)\big]\Diag(b).
\end{equation*}
Each term is computed using two matrix multiplications, as
in~\eqref{eq:ot-kronecker-product}. The row and column sums and objective
value therefore require $\mathcal{O}(L^3)$ work and $\mathcal{O}(L^2)$
memory on two $L\times L$ grids. Since $Y^k\geq0$ and
$(Y^k)^\top e_m=b$, the primal feasibility residual
in~\eqref{eq:experiment-residuals} reduces to
$\eta_{\rm p}=\|Y^ke_n-a\|/(1+\|(a,b)\|)$.
Together with the corresponding dual objective
$a^\top\alpha^k+b^\top\beta^k$ from~\eqref{eq:ot-scaling-duals},
these quantities give the primal--dual gap.

For costs that split into two coordinate costs, the dual feasibility
residual in~\eqref{eq:experiment-residuals} is computed using the same
cost decomposition and sorting, without storing the full matrix
$T\coleq C-\alpha e_n^\top-e_m\beta^\top$. For general costs, the stopping
residuals require $\mathcal{O}(mn)$ work but can still be computed without
storing $Y^k$ or $T$.

\subsection{Stable computation of matrix-vector products}\label{sec:ot-kernel-truncation}
As $k$ increases, entries of
$G_{k/\rho}=X^0\odot e^{-kC/\rho}$ can become extremely small,
while $u^k$ and $v^k$ can become large. To avoid forming these values,
we store logarithms of the scaling vectors.
For $X^0=ab^\top$ used in our experiments, we store
$\log(a\odot u^k)$ and $\log(b\odot v^k)$, initialized as
$\log a$ and $\log b$. The fixed weights $a,b$ do not change
the ratios used in~\eqref{eq:ot-scaling-duals}.
The vector $z^k$ is stored directly. Subtracting the same scalar from
all its entries keeps them from becoming large and leaves the primal
iterates and the OT dual slack unchanged.

Matrix-vector products are computed using the log-sum-exp identity.
For $h_j\in\R\cup\{-\infty\}$ with at least one finite entry,
\begin{equation*}
\log\sum_j e^{h_j}
=\max_jh_j+\log\sum_j\exp(h_j-\max_\ell h_\ell).
\end{equation*}
The exponents on the right are nonpositive, so their exponentials
cannot overflow. This computes the logarithm of the sum without
first evaluating the possibly very large or small values $e^{h_j}$.
Logarithmic stabilization of transport scaling is discussed in
\cite[Section~4.4]{peyre2019computational} and
\cite{schmitzer2019stabilized}.

Now suppose that, along each coordinate, the source and target
points use the same equally spaced locations. For squared Euclidean costs
and $\tau>0$, the entries of either $G_\tau^{(1)}$ or $G_\tau^{(2)}$
in~\eqref{eq:ot-separable-kernel} have the form
$\exp(-\gamma(p-q)^2)$. Here $\gamma>0$ equals $\tau$ times the squared
spacing between adjacent points, divided by the factor used to normalize $C$.
For a vector with entries $e^{h_q}$, the $p$-th output of multiplication
by either matrix is $\sum_{q=1}^L\exp(-\gamma(p-q)^2)e^{h_q}$.
Its logarithm is therefore
\begin{equation}\label{eq:ot-log-contraction}
F_p(h)\coleq\log\sum_{q=1}^L\exp\big(h_q-\gamma(p-q)^2\big),\qquad 1\leq p\leq L.
\end{equation}
We use this formula for the row and column multiplications
in~\eqref{eq:ot-kronecker-product}. Zero input entries, including zeros
inserted at removed grid points, are represented by $h_q=-\infty$.
If all input entries are zero, $F_p(h)=-\infty$ for every $p$.
Otherwise, set $M_p\coleq\max_q\{h_q-\gamma(p-q)^2\}$ and apply
the log-sum-exp identity to obtain
\begin{equation*}
F_p(h)=M_p+\log\sum_{q=1}^L
\exp\big(h_q-\gamma(p-q)^2-M_p\big).
\end{equation*}
To compute all the required maxima efficiently, observe that
$-M_p=\min_q\{\gamma(p-q)^2-h_q\}$.
This is the minimization computed by the squared-distance transform
algorithm of Felzenszwalb and Huttenlocher
\cite[Section~2.1]{felzenszwalb2012distance}, which gives all $M_p$
in $\mathcal{O}(L)$ work. We retain all terms in the exponential sums,
as in the reported experiments. The complete kernel product on two
$L\times L$ grids requires $\mathcal{O}(L^3)$ work.

\let\otsectionfiveclarify\otclarify
\renewcommand{\otclarify}[1]{#1}
\section{Numerical experiments}\label{sec:ot-numerical}
We evaluate the computational time, memory requirements and
solution accuracy of BADMM on CPU and GPU platforms. We consider a
generic-cost-matrix setting, in which the methods receive only $(a,b,C)$
or the corresponding linear program, and a two-dimensional-kernel setting
that additionally uses the grid structure described in 
Section \ref{sec:ot-structured-products}. These tests assess the
performance of BADMM and the computational gains from the structured
implementation in Section~\ref{sec:ot-implementation}.

\medskip
\par\noindent\tb{Tested problems.}
We use images 1001 and 1002 from each of the ten DOTmark
classes~\cite{schrieber2016dotmark} at resolutions $L=64,128,256,512$.
The generic-cost-matrix tests use resolutions 64 and 128, with additional
GPU BADMM tests at resolution 256. The two-dimensional-kernel tests extend
to resolution 512. We normalize each image by its total pixel intensity.
As in~\cite{hou2024sparse}, we remove pixels with zero mass. We also
discard normalized entries at or below $10^{-14}$ and renormalize the
remaining entries. Let $a\in\R_{++}^m$ and $b\in\R_{++}^n$ be the
resulting probability vectors, where $m$ and $n$ are the numbers
of retained pixels in the two images.
For their pixel positions $x_i,y_j\in\{1,\ldots,L\}^2$, the
unnormalized transport cost is $\|x_i-y_j\|^2$.
We divide this cost matrix by its maximum entry to obtain $C$.
{\color{black}
Since DOTmark provides images only up to $512\times512$, we use DOMDEC images
\texttt{g-008-2048} and \texttt{g-009-2048} from the DOMDEC examples
repository~\cite{bonafini2021domain} to test larger problems.
We test the original images and the $1024\times1024$ images
obtained by summing pixel intensities within nonoverlapping $2\times2$
blocks. We apply the same normalization and cost scaling as above.
\par}

\medskip
\par\noindent\tb{Compared methods.}
{\color{black}
Methods in the generic-cost-matrix group receive only $(a,b,C)$
or the corresponding linear program. Methods in the two-dimensional-kernel
 group may additionally use grid coordinates and the squared Euclidean
cost structure.
On the CPU, the generic-cost-matrix group contains BADMM,
SqSN~\cite{hou2024sparse}, HPR~\cite{chen2025hprlp},
PDLP~\cite{applegate2025pdlp},
Gurobi's barrier and PDHG methods,
HiGHS-Bar (HiPO)~\cite{zanetti2025hipo}, POT-EMD~\cite{POT},
CPLEX-Net, and CPLEX-Bar. The two-dimensional-kernel group contains
BADMM and HOT~\cite{zhang2025hot}.
We call PDLP through OR-Tools.
The suffixes Bar and Net denote barrier and network simplex methods,
respectively, and PDHG denotes the primal--dual hybrid gradient method.
We do not include IPOT~\cite{xie2020ipot} in the comparison,
as its default settings did not attain the required accuracy
and its convergence was slower than that of the other benchmark methods
in our preliminary tests.

On the GPU, the generic-cost-matrix group contains BADMM,
PDOT~\cite{lu2024pdot}, Wang's BADMM
(denoted by W-BADMM)~\cite{wang2014bregman},
cuPDLPx~\cite{lu2025cupdlpx}, the C implementation
HPR-LP-C of HPR-LP~\cite{chen2025hprlp},
Gurobi-PDHG, and the PDLP and barrier methods of
cuOpt. 
To avoid confusion, we recall that 
W-BADMM uses a linear dual update, whereas the BADMM studied here uses
the nonlinear dual update in Algorithm~\ref{alg:ot}.
The two-dimensional-kernel group contains
BADMM, HOT~\cite{zhang2025hot}, HALO~\cite{xia2026halo}, and HELLO~\cite{xia2026hello}.
The codes HOT, HALO and HELLO used in our experiments
require grid information or point coordinates and are therefore not
directly applicable when only an arbitrary cost matrix is supplied.
\par}

\medskip
\par\noindent\tb{Stopping conditions.}
Let $(P,(\alpha,\beta,\Gamma))$ be a returned approximate primal--dual
solution pair for~\eqref{eq:gw}, where $P$ is the primal matrix, $\alpha$ and $\beta$ are
the multipliers for the marginal constraints, and $\Gamma$ is the dual slack.
Set $T\coleq C-\alpha e_n^\top-e_m\beta^\top$ and
$M_-\coleq\max\{-M,0\}$ for a matrix $M$.
We measure the relative primal infeasibility, relative dual infeasibility
and relative primal--dual gap on the input OT problem by
\begin{equation}\label{eq:experiment-residuals}
\begin{aligned}
 \eta_{\rm p}&\coleq\max\left\{
 \frac{\|(Pe_n-a,P^\top e_m-b)\|}{1+\|(a,b)\|},
 \frac{\|P_-\|_F}{1+\|P\|_F}\right\},\\
 \eta_{\rm d}&\coleq
 \frac{\max\{\|T_-\|_F,\|\Gamma_-\|_F,\|\Gamma-T\|_F\}}
 {1+\|C\|_F},\quad
 \eta_{\rm g}\coleq
 \frac{|\langle C,P\rangle-a^\top\alpha-b^\top\beta|}
 {1+|\langle C,P\rangle|+|a^\top\alpha+b^\top\beta|}.
\end{aligned}
\end{equation}
If a solver does not return a separate dual slack, we set $\Gamma=T$.
We call a returned approximate primal--dual solution accurate if
$\eta_{\max}\coleq\max\{\eta_{\rm p},\eta_{\rm d},\eta_{\rm g}\}
\leq10^{-5}$.
{\color{black}
We independently evaluate every returned solution using the
residuals in~\eqref{eq:experiment-residuals} in double precision (FP64)
on the input OT problem. We use each solver's default stopping criteria
and set the stopping tolerances to $10^{-5}$ where available, with the
exceptions described below. With a primal feasibility tolerance of
$10^{-5}$, HiGHS declared some feasible test problems as infeasible during
presolve. We therefore retain its default tolerance of $10^{-7}$.
The default stopping criteria of HOT and W-BADMM do not always attain
the required accuracy. We additionally require
$\eta_{\max}\leq10^{-5}$ for CPU HOT and W-BADMM. CPU HOT retains its native checking frequency of 100
iterations, and W-BADMM retains its native single-precision updates
while evaluating the additional test in FP64. GPU HOT retains its default stopping criterion.
POT-EMD retains its default stopping test, since its interface
does not expose a stopping tolerance. Its returned solution is evaluated
using the same residual criterion \eqref{eq:experiment-residuals}.
\par}

\medskip
\par\noindent\tb{BADMM implementation and parameters.}\label{sec:ot-adaptive-implementation}
{\color{black}
On both CPU and GPU platforms, the generic-cost-matrix and
two-dimensional-kernel BADMM implementations use the scaling-vector representation
in Section~\ref{sec:ot-fast-computation}. The generic-cost-matrix
implementation evaluates matrix-vector products from the supplied $C$.
The two-dimensional-kernel implementation uses the structured products in
Section~\ref{sec:ot-structured-products} and the stable
evaluation described in Section~\ref{sec:ot-kernel-truncation}.
The reported experiments use stabilized matrix-vector products without truncating small exponential terms.
On each platform, the two implementations use the same
outer iteration, penalty updates and stopping condition, but evaluate
the matrix-vector products differently.
The CPU implementations use FP64.
The GPU implementations evaluate
update exponentials in single precision (FP32). They store the logarithms
of the scaling vectors and the dual variables in FP64, and use FP64
to sum the terms in matrix-vector products and evaluate the stopping residuals.
}
We initialize $X^0=ab^\top$, $Z^0=0$, $\rho_0=0.1$ and $s=1.5$.
{\color{black}
We check the stopping condition at the first iteration, every five
iterations and at termination, and consider penalty updates every five
iterations. At each penalty update, set
}
\begin{equation}\label{eq:ot-adaptive-penalty}
\rho_{k+1}=\begin{cases}
 \min\{0.3,1.2\rho_k\},&\eta_{\rm p}>5\max\{\eta_{\rm d},\eta_{\rm g}\},\\
 \max\{0.003,\rho_k/1.2\},&\eta_{\rm p}<\max\{\eta_{\rm d},\eta_{\rm g}\}/5,\\
 \rho_k,&\text{otherwise}.
\end{cases}
\end{equation}
After a penalty change,
we rescale $z$ by $\rho_k/\rho_{k+1}$ to preserve $Z$.
For numerical stabilization, the two-dimensional-kernel CPU implementation
and both GPU implementations subtract $\max_j z_j$ from every component of
$z$ after each multiplier update.
The convergence results in Section~\ref{sec:ot-convergence}
concern a sufficiently large fixed penalty. In the experiments, we use
the heuristic adaptive penalty rule~\eqref{eq:ot-adaptive-penalty}
for better computational performance, as is common in
ADMM~\cite{boyd2011distributed,he2000alternating,xu2017adaptive}.

\medskip
\par\noindent\tb{Implementation and timings.}
{\color{black}
CPU experiments are run on a machine with an Intel Xeon E5-2680 v3 processor at 2.50 GHz using one
computational thread. We use Julia 1.13.1. The CPU LP packages are Gurobi 13.0.3~\cite{gurobimanual},
CPLEX 22.2.0.1~\cite{ibmcplexmanual}, HiGHS 1.15.1 and OR-Tools 9.15.6755.
POT-EMD uses POT 0.9.7.post1.
CPU experiments also have a 75 GiB memory limit.

GPU experiments are run on a machine with one NVIDIA RTX 6000 Ada Generation GPU with 48 GiB
of memory, on an Intel Xeon w7-2595X host. They use up to eight host
threads and one BLAS thread. GPU Gurobi and cuOpt use versions 13.0.3
and 26.8.0~\cite{cuoptsoftware}, respectively.
Input reading, warm-up and independent residual evaluation
after termination are excluded from the reported times.

For all experiments, we set the time limit to 3600 seconds.
\par}

\medskip
\par\noindent\tb{Table notation.}
{\color{black}
``Acc./Run'' is the number of returned solutions satisfying
$\eta_{\max}\leq10^{-5}$, with the residuals defined
in~\eqref{eq:experiment-residuals}, divided by the number of attempted instances. TL and ML denote a time limit and a memory limit, respectively.
A dagger indicates that $\eta_{\max}$, or its reported mean,
exceeds $10^{-5}$.
CPU memory is the peak sampled resident memory during each run,
including input data and runtime memory. GPU tables separate host
resident memory from device memory. 
For TL and ML, the status is reported in place of a numeric time.
}

\providecommand{\otnum}[3]{\mbox{#1e\makebox[0.65em][c]{\ensuremath{#2}}#3}}
\newcommand{\OTCPUGeneralTable}{%
\begingroup
\begin{footnotesize}
\begin{longtable}[c]{llrrrrr}
\caption{\textcolor{black}{CPU results in the generic-cost-matrix setting at $L=64$.}\label{tab:ot-cpu-general}} \\
\toprule
$L$ & Solver & Time (s) & Memory (GiB) & $\eta_{\max}$ & Acc./Run & TL/ML \\
\midrule
\endfirsthead
\multicolumn{7}{c}{Table \thetable\ continued}\\
\toprule
$L$ & Solver & Time (s) & Memory (GiB) & $\eta_{\max}$ & Acc./Run & TL/ML \\
\midrule
\endhead
\midrule
\multicolumn{7}{r}{Continued on next page}\\
\endfoot
\bottomrule
\endlastfoot
64 & BADMM & 25.023 & 0.635 & \otnum{9.68}{-}{06} & 1/1 & 0/0 \\
64 & POT-EMD & 7.520 & 0.607 & \otnum{1.90}{-}{13} & 1/1 & 0/0 \\
64 & CPLEX-Net & 55.890 & 3.505 & \otnum{1.11}{-}{18} & 1/1 & 0/0 \\
64 & SqSN & 62.766 & 2.808 & \otnum{7.74}{-}{06} & 1/1 & 0/0 \\
64 & Gurobi-Bar & 336.696 & 11.301 & \otnum{4.43}{-}{06} & 1/1 & 0/0 \\
64 & CPLEX-Bar & 655.793\rlap{$^{\dagger}$} & 6.561 & \otnum{2.28}{-}{05} & 0/1 & 0/0 \\
64 & HPR & TL & \textendash & \textendash & 0/1 & 1/0 \\
64 & PDLP & TL & \textendash & \textendash & 0/1 & 1/0 \\
64 & Gurobi-PDHG & TL & \textendash & \textendash & 0/1 & 1/0 \\
64 & HiGHS-Bar & TL & \textendash & \textendash & 0/1 & 1/0 \\
\end{longtable}
\end{footnotesize}
\endgroup}

\newcommand{\OTCPUStructuredTable}{%
\begingroup
\begin{footnotesize}
\begin{longtable}[c]{llrrrrr}
\caption{\textcolor{black}{CPU results in the two-dimensional-kernel setting.}\label{tab:ot-cpu-structured}} \\
\toprule
$L$ & Solver & Time (s) & Memory (GiB) & $\eta_{\max}$ & Acc./Run & TL/ML \\
\midrule
\endfirsthead
\multicolumn{7}{c}{Table \thetable\ continued}\\
\toprule
$L$ & Solver & Time (s) & Memory (GiB) & $\eta_{\max}$ & Acc./Run & TL/ML \\
\midrule
\endhead
\midrule
\multicolumn{7}{r}{Continued on next page}\\
\endfoot
\bottomrule
\endlastfoot
64 & BADMM & 1.328 & 0.610 & \otnum{9.07}{-}{06} & 10/10 & 0/0 \\
64 & HOT & 36.803 & 0.281 & \otnum{7.13}{-}{06} & 10/10 & 0/0 \\[5pt]
128 & BADMM & 9.769 & 0.635 & \otnum{8.73}{-}{06} & 10/10 & 0/0 \\
128 & HOT & 631.115 & 0.669 & \otnum{8.61}{-}{06} & 9/10 & 1/0 \\[5pt]
256 & BADMM & 89.237 & 0.656 & \otnum{9.01}{-}{06} & 10/10 & 0/0 \\
256 & HOT & TL & \textendash & \textendash & 0/10 & 10/0 \\[5pt]
512 & BADMM & 708.348 & 0.716 & \otnum{9.49}{-}{06} & 10/10 & 0/0 \\
512 & HOT & TL & \textendash & \textendash & 0/10 & 10/0 \\
\end{longtable}
\end{footnotesize}
\endgroup}

\newcommand{\OTGPUGeneralTable}{%
\begingroup
\begin{footnotesize}
\begin{longtable}[c]{llrrrrrr}
\caption{\textcolor{black}{GPU results in the generic-cost-matrix setting at $L=64$.}\label{tab:ot-gpu-general}} \\
\toprule
$L$ & Solver & Time (s) & Host (GiB) & GPU (GiB) & $\eta_{\max}$ & Acc./Run & TL/ML \\
\midrule
\endfirsthead
\multicolumn{8}{c}{Table \thetable\ continued}\\
\toprule
$L$ & Solver & Time (s) & Host (GiB) & GPU (GiB) & $\eta_{\max}$ & Acc./Run & TL/ML \\
\midrule
\endhead
\midrule
\multicolumn{8}{r}{Continued on next page}\\
\endfoot
\bottomrule
\endlastfoot
64 & BADMM & 0.127 & 1.828 & 1.241 & \otnum{9.68}{-}{06} & 1/1 & 0/0 \\
64 & PDOT & 6.241 & 2.273 & 25.040 & \otnum{6.39}{-}{06} & 1/1 & 0/0 \\
64 & HPR-LP-C & 27.553\rlap{$^{\dagger}$} & 2.315 & 7.030 & \otnum{6.51}{-}{02} & 0/1 & 0/0 \\
64 & cuOpt-Bar & 35.472\rlap{$^{\dagger}$} & 11.926 & 13.702 & \otnum{2.38}{-}{03} & 0/1 & 0/0 \\
64 & cuPDLPx & 76.231\rlap{$^{\dagger}$} & 5.447 & 3.720 & \otnum{1.66}{-}{05} & 0/1 & 0/0 \\
64 & Gurobi-PDHG & 106.283\rlap{$^{\dagger}$} & 13.738 & 2.888 & \otnum{8.57}{-}{04} & 0/1 & 0/0 \\
64 & cuOpt-PDLP & 188.049\rlap{$^{\dagger}$} & 5.299 & 8.071 & \otnum{1.54}{-}{05} & 0/1 & 0/0 \\
64 & W-BADMM & TL & \textendash & \textendash & \textendash & 0/1 & 1/0 \\
\end{longtable}
\end{footnotesize}
\endgroup}

\newcommand{\OTGPUStructuredTable}{%
\begingroup
\begin{footnotesize}
\begin{longtable}[c]{llrrrrrr}
\caption{\textcolor{black}{GPU results in the two-dimensional-kernel setting.}\label{tab:ot-gpu-structured}} \\
\toprule
$L$ & Solver & Time (s) & Host (GiB) & GPU (GiB) & $\eta_{\max}$ & Acc./Run & TL/ML \\
\midrule
\endfirsthead
\multicolumn{8}{c}{Table \thetable\ continued}\\
\toprule
$L$ & Solver & Time (s) & Host (GiB) & GPU (GiB) & $\eta_{\max}$ & Acc./Run & TL/ML \\
\midrule
\endhead
\midrule
\multicolumn{8}{r}{Continued on next page}\\
\endfoot
\bottomrule
\endlastfoot
256 & BADMM & 0.379 & 1.864 & 1.145 & \otnum{9.01}{-}{06} & 10/10 & 0/0 \\
256 & HOT & 7.209\rlap{$^{\dagger}$} & 1.355 & 3.802 & \otnum{1.53}{-}{03} & 0/10 & 0/0 \\
256 & HALO & 2.164\rlap{$^{\dagger}$} & 1.165 & 1.342 & \otnum{1.95}{-}{05} & 3/10 & 0/0 \\
256 & HELLO & 8.454 & 1.499 & 2.120 & \otnum{9.82}{-}{06} & 10/10 & 0/0 \\[5pt]
512 & BADMM & 1.872 & 1.991 & 1.210 & \otnum{9.49}{-}{06} & 10/10 & 0/0 \\
512 & HOT & 57.768\rlap{$^{\dagger}$} & 4.887 & 23.070 & \otnum{1.30}{-}{03} & 0/10 & 0/0 \\
512 & HALO & 11.682\rlap{$^{\dagger}$} & 1.923 & 2.556 & \otnum{1.38}{-}{05} & 6/10 & 0/0 \\
512 & HELLO & 158.148 & 2.694 & 2.729 & \otnum{7.71}{-}{06} & 10/10 & 0/0 \\
\addlinespace[5pt]
1024 & BADMM & 11.253 & 2.377 & 1.491 & \otnum{9.57}{-}{06} & 1/1 & 0/0 \\
1024 & HOT & ML & \textendash & \textendash & \textendash & \textendash & \textendash \\
1024 & HALO & 145.648 & 5.847 & 5.841 & \otnum{1.68}{-}{06} & 1/1 & 0/0 \\
1024 & HELLO & TL & \textendash & \textendash & \textendash & 0/1 & 1/0 \\[5pt]
2048 & BADMM & 98.683 & 3.076 & 2.585 & \otnum{9.59}{-}{06} & 1/1 & 0/0 \\
2048 & HOT & ML & \textendash & \textendash & \textendash & \textendash & \textendash \\
2048 & HALO & 2346.401 & 22.187 & 16.544 & \otnum{1.47}{-}{07} & 1/1 & 0/0 \\
2048 & HELLO & TL & \textendash & \textendash & \textendash & \textendash & \textendash \\
\end{longtable}
\end{footnotesize}
\endgroup}

\newcommand{\OTCPUSelectedTable}{%
\begingroup
\begin{footnotesize}
\begin{longtable}[c]{llrrrrr}
\caption{CPU results in the generic-cost-matrix setting.\label{tab:ot-cpu-selected}} \\
\toprule
$L$ & Solver & Time (s) & Memory (GiB) & $\eta_{\max}$ & Acc./Run & TL/ML \\
\midrule
\endfirsthead
\multicolumn{7}{c}{Table \thetable\ continued}\\
\toprule
$L$ & Solver & Time (s) & Memory (GiB) & $\eta_{\max}$ & Acc./Run & TL/ML \\
\midrule
\endhead
\bottomrule
\endfoot
64 & BADMM & 30.669 & 0.615 & \otnum{9.07}{-}{06} & 10/10 & 0/0 \\
64 & POT-EMD & 16.120 & 0.530 & \otnum{8.34}{-}{13} & 10/10 & 0/0 \\
64 & SqSN & 50.792 & 2.463 & \otnum{7.62}{-}{06} & 9/10 & 0/0 \\
64 & CPLEX-Net & 47.399 & 3.047 & \otnum{2.53}{-}{17} & 10/10 & 0/0 \\[5pt]
128 & BADMM & 528.449 & 2.263 & \otnum{8.73}{-}{06} & 10/10 & 0/0 \\
128 & POT-EMD & 1099.711 & 7.650 & \otnum{1.35}{-}{11} & 10/10 & 0/0 \\
128 & SqSN & 1554.204 & 33.673 & \otnum{8.33}{-}{06} & 10/10 & 0/0 \\
128 & CPLEX-Net & 2136.888 & 46.867 & \otnum{2.70}{-}{17} & 8/10 & 2/0 \\
\end{longtable}
\end{footnotesize}
\endgroup}

\newcommand{\OTGPUSelectedTable}{%
\begingroup
\begin{footnotesize}
\begin{longtable}[c]{llrrrrrr}
\caption{GPU results in the generic-cost-matrix setting.\label{tab:ot-gpu-selected}} \\
\toprule
$L$ & Solver & Time (s) & Host (GiB) & GPU (GiB) & $\eta_{\max}$ & Acc./Run & TL/ML \\
\midrule
\endfirsthead
\multicolumn{8}{c}{Table \thetable\ continued}\\
\toprule
$L$ & Solver & Time (s) & Host (GiB) & GPU (GiB) & $\eta_{\max}$ & Acc./Run & TL/ML \\
\midrule
\endhead
\bottomrule
\endfoot
64 & BADMM & 0.131 & 1.807 & 1.220 & \otnum{9.07}{-}{06} & 10/10 & 0/0 \\
64 & PDOT & 219.652 & 2.174 & 26.915 & \otnum{7.31}{-}{06} & 10/10 & 0/0 \\[5pt]
128 & BADMM & 2.018 & 3.513 & 2.982 & \otnum{8.73}{-}{06} & 10/10 & 0/0 \\
128 & PDOT & 642.218 & 5.866 & 37.376 & \otnum{8.15}{-}{06} & 7/10 & 3/0 \\[5pt]
256 & BADMM & 27.006 & 29.868 & 31.320 & \otnum{9.01}{-}{06} & 10/10 & 0/0 \\
256 & PDOT & ML & \textendash & \textendash & \textendash & \textendash & \textendash \\
\end{longtable}
\end{footnotesize}
\endgroup}

{\color{black}
\subsection{CPU experiments}
We first compare all methods in the generic-cost-matrix setting
at $L=64$ on the WhiteNoise instance in
Table~\ref{tab:ot-cpu-general}. BADMM
requires less computational time and memory than CPLEX-Net and SqSN
while meeting the accuracy criterion. POT-EMD has the shortest
computational time on this instance. The remaining methods either
require substantially more time or fail to meet the accuracy criterion within the
time limit.

\OTCPUGeneralTable

Based on these results, we compare BADMM with POT-EMD, SqSN and
CPLEX-Net on more DOTmark instances in Table~\ref{tab:ot-cpu-selected}.
For the DOTmark summary tables, we report the mean values of
computational time, sampled peak memory and $\eta_{\max}$ over returned
solutions, excluding TL and ML. Returned solutions that do not meet
the accuracy criterion are included. The advantage of BADMM becomes
more pronounced at the higher resolution in
Table~\ref{tab:ot-cpu-selected}. At $L=128$, BADMM meets the accuracy
criterion on all ten instances, with about half the mean computational
time and less than one third of the mean memory required by POT-EMD.
BADMM also has lower mean computational time and memory requirements
than SqSN and CPLEX-Net.

\OTCPUSelectedTable

The structured matrix-vector implementation 
in Section~\ref{sec:ot-structured-products}
substantially reduces the
computational time of BADMM in Table~\ref{tab:ot-cpu-structured}.
Compared with the generic-cost-matrix implementation on the same ten
instances, the mean computational time is reduced by factors of
approximately 23 and 54 at resolutions 64 and 128, respectively.
These gains are obtained by evaluating the matrix-vector products using the
separable representation in~\eqref{eq:ot-separable-kernel}.
BADMM meets the accuracy criterion on all reported instances through
resolution 512. It also requires less mean computational time than HOT,
which also exploits the separable squared Euclidean cost on the
two-dimensional grid.
At resolution 64, both methods meet the accuracy criterion on all ten
instances, and BADMM is approximately 28 times faster than HOT in terms
of mean computational time. HOT reaches the time limit on one instance
at resolution 128 and on all ten instances at resolutions 256 and 512.

\OTCPUStructuredTable

\subsection{GPU experiments}
We first compare all methods in the generic-cost-matrix setting
at $L=64$ on the WhiteNoise instance in Table~\ref{tab:ot-gpu-general}.
BADMM has the shortest computational time among the methods
that meet the accuracy criterion.
Compared with PDOT, BADMM's computational time and GPU memory requirement
are reduced by factors of approximately 49 and 20, respectively.
BADMM and PDOT are the only methods that meet the criterion on this
WhiteNoise instance at $L=64$.

\OTGPUGeneralTable

Based on these results, we compare BADMM with PDOT on more
DOTmark instances in Table~\ref{tab:ot-gpu-selected}. The results show
a substantial computational advantage for BADMM. At resolution 64,
BADMM requires over three orders of magnitude less mean
computational time and less than 5\% of the mean GPU memory
required by PDOT, with both methods meeting the accuracy criterion on
all ten instances. At resolution 128, BADMM meets the criterion on all
ten instances, whereas PDOT reaches the time limit on three instances.
BADMM also meets the criterion on all ten instances at resolution 256.
Comparing Tables~\ref{tab:ot-cpu-selected} and~\ref{tab:ot-gpu-selected},
the GPU implementation of BADMM reduces the mean computational time
at resolution 128 by a factor of approximately 260 on the reported
CPU and GPU configurations.

\OTGPUSelectedTable

Table~\ref{tab:ot-gpu-structured} compares BADMM with HOT, HALO and HELLO  using the two-dimensional kernel
setting on a GPU platform.
BADMM combines low computational time and GPU memory
requirements with consistent attainment of the accuracy criterion in
Table~\ref{tab:ot-gpu-structured}. BADMM and HELLO meet the criterion on
all ten DOTmark instances at resolutions 256 and 512. Compared with
HELLO, BADMM reduces the mean computational time by factors of
approximately 22 and 84, respectively, and also requires less GPU
memory. Thus, its computational advantage is larger at the higher
tested resolution. HALO meets the criterion on only some instances,
while HOT does not meet it on these tests.

The DOMDEC results for the resolutions at $L=1024$ and
$L=2048$ further demonstrate the ability of the
structured BADMM implementation to handle larger problems. BADMM meets
the accuracy criterion at resolutions 1024 and 2048 with low GPU memory
requirements. Its computational times are approximately $1/13$ and
$1/24$ of those of HALO, respectively. Both methods meet the criterion,
with HALO attaining smaller residuals. HOT exceeds the available GPU
memory.

\OTGPUStructuredTable

The detailed results, including iteration counts and the
three residuals, are given in Appendix~\ref{app:ot-computational-details}.
}

\let\otclarify\otsectionfiveclarify

\section{Conclusions}\label{sec:ot-conclusion}
We have established convergence of entropy Bregman ADMM for
 general discrete OT problems. With a strictly positive feasible primal
initialization, a zero initial multiplier, any fixed multiplier stepsize
$s\in(0,2)$, and a sufficiently large fixed penalty, the full primal--dual
sequence converges $R$-linearly. This result does not require a unique
or strictly positive OT optimal solution. Moreover, the entropy Bregman
projection characterization determines the primal limit selected by the
initial point.

The CPU and GPU experiments demonstrate the computational
efficiency of BADMM on the tested problems. For separable costs on
two-dimensional grids, the structured implementation substantially
reduces computational time and memory requirements and enables the
solution of larger problems. Further convergence analysis of adaptive
penalty updates and sparse implementations is left for future work.

\section*{Declaration of generative AI and AI-assisted technologies}
During the preparation of this work, the authors used OpenAI Codex
to assist with drafting and revising the manuscript, checking mathematical
arguments, and developing numerical code. The authors fully reviewed and edited
the resulting material and take full responsibility for the content of
this work.

\begingroup
\expandafter\def\csname otnewbib@xia2026halo\endcsname{1}
\expandafter\def\csname otnewbib@xia2026hello\endcsname{1}
\expandafter\def\csname otnewbib@wangcudabadmm\endcsname{1}
\expandafter\def\csname otnewbib@cupdlpxsoftware\endcsname{1}
\expandafter\def\csname otnewbib@cuoptsoftware\endcsname{1}
\expandafter\def\csname otnewbib@domdecexamples\endcsname{1}

\expandafter\def\csname otbibcorrected@khoo2025bregman\endcsname{1}
\expandafter\def\csname otbibcorrected@sharify2013scaling\endcsname{1}
\expandafter\def\csname otbibcorrected@chizat2024annealed\endcsname{1}
\expandafter\def\csname otbibcorrected@chen2024spurious\endcsname{1}
\expandafter\def\csname otbibcorrected@ding2026nonkkt\endcsname{1}
\expandafter\def\csname otbibcorrected@gao2026sinkhorn\endcsname{1}
\expandafter\def\csname otbibcorrected@lu2024pdot\endcsname{1}
\expandafter\def\csname otbibcorrected@sun2026multiscale\endcsname{1}
\expandafter\def\csname otbibcorrected@ma2025bdrs\endcsname{1}
\expandafter\def\csname otbibcorrected@xia2026hello\endcsname{1}
\expandafter\def\csname otbibcorrected@lu2025cupdlpx\endcsname{1}
\expandafter\def\csname otbibcorrected@zanetti2025hipo\endcsname{1}
\expandafter\def\csname otbibcorrected@ma2026scaling\endcsname{1}
\expandafter\def\csname otbibcorrected@ma2026failure\endcsname{1}

\expandafter\def\csname otbibcorrected@gurobimanual\endcsname{1}
\expandafter\def\csname otbibcorrected@ibmcplexmanual\endcsname{1}
\expandafter\def\csname otbibcorrected@cuoptsoftware\endcsname{1}

\expandafter\def\csname otbibcorrected@bonafini2021domain\endcsname{1}

\let\otoriginalbibitem\bibitem
\renewcommand{\bibitem}[1]{%
  \color{black}
  \otoriginalbibitem{#1}}
\bibliographystyle{plain}
\bibliography{references}

\endgroup

\appendix
\section{Support identification}\label{sec:ot-support-identification}
The BADMM iterates remain positive even when the optimal
solutions are sparse. 
To reduce the number of variables,
we derive a test that identifies variables equal
to zero in every optimal solution. 
Fixing these variables at zero preserves
the full optimal solution set. 
Variable fixing and screening are also
used in sparse quadratic programming~\cite{hou2026scsdp,tan2025screening}.

We first construct a feasible primal--dual pair from $(Y^k,Z^k)$.
Since $(Y^k)^\top e_m=b$, we obtain a feasible matrix $V^k$ by correcting
the row marginal and adding a rank-one matrix. The rounding construction
in~\cite[Algorithm~2 and Lemma~7]{altschuler2017nearlinear} reduces to
\begin{equation}\label{eq:ot-repair}
\begin{aligned}
W_{ij}&\coleq Y^k_{ij}\min\{1,a_i/(Y^ke_n)_i\},\\
\Delta a&\coleq a-We_n,\qquad \Delta b\coleq b-W^\top e_m,
\qquad \tau_k\coleq e_m^\top\Delta a,\\
V^k&\coleq
\begin{cases}W+\Delta a(\Delta b)^\top/\tau_k,&\tau_k>0,\\
W,&\tau_k=0.
\end{cases}
\end{aligned}
\end{equation}
Since $0\leq W\leq Y^k$ and $We_n\leq a$, we have
$\Delta a,\Delta b\geq0$. Their sums are equal because
$e_m^\top a=e_n^\top b=1$, so
$e_m^\top\Delta a=e_n^\top\Delta b=\tau_k$.
For $\tau_k>0$, the added rank-one matrix has row and column sums
$\Delta a$ and $\Delta b$, respectively. If $\tau_k=0$, both
corrections vanish. Hence $V^k\in\Pi$ in either case.

It is readily shown that $\|Y^ke_n-a\|_1 = e_m^\top\Delta a + e_m^\top (Y^k-W)e_n=\tau_k + e_n^\top ((Y^k)^\top e_m - W^\top e_m)= \tau_k + e_n^\top (b-W^\top e_m)=2\tau_k.$ 
The nonnegative matrices $Y^k-W$ and $V^k-W$ each have entry sum
$\tau_k=\|Y^ke_n-a\|_1/2$, giving
$\|V^k-Y^k\|_1\leq\|Y^ke_n-a\|_1$.
We obtain $\widehat\beta^k$ from the column means of $-Z^k$,
consistent with the limiting relation $Z_\infty=-e_m\beta_\infty^\top$.
We then define $\widehat\alpha^k$ by row minima so that the pair is
 dual feasible:
\begin{equation}\label{eq:ot-support-labels}
\begin{gathered}
\widehat\beta_j^k\coleq-\frac1m\sum_iZ^k_{ij},\qquad
\widehat\alpha_i^k\coleq\min_j(C_{ij}-\widehat\beta_j^k),\\
\Gamma^k_{ij}\coleq C_{ij}-\widehat\alpha_i^k-\widehat\beta_j^k,\qquad
\widehat G_k\coleq\langle C,V^k\rangle
 -a^\top\widehat\alpha^k-b^\top\widehat\beta^k.
\end{gathered}
\end{equation}
The row minima ensure $\Gamma^k\geq0$, and weak duality gives
$\widehat G_k\geq0$. 

Let $\delta_{\rm v}>0$ be a lower bound on every positive entry of
every vertex of $\Pi$. The next algorithm retains the variables not
excluded by the dual slack and gap.
\begin{algorithm}[H]
\caption{Support identification from a Bregman ADMM iterate}\label{alg:ot-identification}
\textbf{Input:} Current iterates $(Y^k,Z^k)$ and a valid bound $\delta_{\rm v}>0$.
\begin{enumerate}[leftmargin=*,itemsep=0pt,topsep=4pt]
\item Form $V^k\in\Pi$ by~\eqref{eq:ot-repair}.
\item Compute the dual feasible pair, slack, and gap
in~\eqref{eq:ot-support-labels}.
\item Return
\begin{equation}\label{eq:ot-retained-set}
T_k\coleq\{(i,j):\Gamma^k_{ij}\leq\widehat G_k/\delta_{\rm v}\}.
\end{equation}
\end{enumerate}
\end{algorithm}
Here $T_k$ is the set of indices retained by the test.
The following proposition shows that variables outside $T_k$ can be
fixed at zero without removing any optimal solution.
Given a valid bound $\delta_{\rm v}>0$, each identification
check requires $\mathcal{O}(mn)$ operations and may be applied at
selected iterations.

We next describe how to choose $\delta_{\rm v}$.
If the marginals are integer multiples of $1/d$, use $\delta_{\rm v}\coleq1/d$.
For uniform marginals, $d\coleq\operatorname{lcm}(m,n)$ is valid.
For general marginals, a valid choice is
\begin{equation}\label{eq:ot-vertex-mass}
\delta_{\rm v}\coleq
\min_{\substack{I\subseteq\{1,\ldots,m\},\ J\subseteq\{1,\ldots,n\}\\
 \sum_{i\in I}a_i\ne\sum_{j\in J}b_j}}
 \left|\sum_{i\in I}a_i-\sum_{j\in J}b_j\right|.
\end{equation}
This expression establishes the existence of a positive
bound, but its direct evaluation may be expensive.

We prove safe exclusion for OT and, using strict complementarity,
finite identification of $S$.
\begin{proposition}[Safe exclusion and finite support identification]
Let $V\in\Pi$, $\alpha_i+\beta_j\leq C_{ij}$, and define
$G\coleq\langle C,V\rangle-a^\top\alpha-b^\top\beta$ and
$\Gamma_{ij}\coleq C_{ij}-\alpha_i-\beta_j$.
Let $\delta_{\rm v}>0$ be a lower bound on every positive
entry of every vertex of $\Pi$. The following statements hold.
\begin{enumerate}[label=(\roman*),ref=(\roman*),leftmargin=*]
\item (Safe exclusion.) We have
\begin{equation}\label{eq:ot-safe-exclusion}
\Gamma_{ij}>G/\delta_{\rm v}\quad\Longrightarrow\quad
P_{ij}=0\quad\text{for every }P\in\mathcal O.
\end{equation}
Consequently, the sets $T_k$ in~\eqref{eq:ot-retained-set}
satisfy $S\subseteq T_k$ at every iteration.
\item (Finite identification.) For the sequence generated
by Algorithm~\ref{alg:ot} under the assumptions of
Theorem~\ref{thm:ot-global}, $T_k=S$ for all sufficiently large $k$.
\end{enumerate}
\end{proposition}
\begin{proof}
\indent\textup{(Positive vertex entries.)} We first verify
 the lower bounds used in~\eqref{eq:ot-safe-exclusion}.
For a vertex $U$ of $\Pi$, its bipartite support graph in
Definition~\ref{def:ot-support-graph} is a forest
\cite[Lemma~2.9]{deloera2014transportation}.
Fix an edge $(i,j)$ with $U_{ij}>0$ and delete it from its tree.
Let $I,J$ be the row and column index sets of either resulting part.
The deleted edge is the only positive entry connecting this part
to its complement. Subtracting the sums of the marginal equations
on $I,J$ cancels every internal entry and leaves either $U_{ij}$
or $-U_{ij}$. Hence
\begin{equation*}
U_{ij}=\left|\sum_{p\in I}a_p-\sum_{q\in J}b_q\right|>0.
\end{equation*}
The minimum in~\eqref{eq:ot-vertex-mass} is over a nonempty finite
set of positive differences, so it is positive and bounds every
positive vertex entry from below.
If the marginals are integer multiples of $1/d$, each such
difference is a positive integer multiple of $1/d$, which proves
the stated rational bound.

\textup{(Safe exclusion.)} We prove (i) by contradiction.
Suppose $\Gamma_{ij}>G/\delta_{\rm v}$ but an optimal solution
is positive at $(i,j)$.
The optimal set is a face of the compact polytope $\Pi$ and is the
convex hull of its vertices. Therefore it has an optimal vertex
$U$ with $U_{ij}>0$ and hence $U_{ij}\geq\delta_{\rm v}$.
Dual feasibility gives $\Gamma\geq0$, while optimality of $U$ and
feasibility of $V$ give $\langle C,U\rangle\leq\langle C,V\rangle$.
Using the marginal constraints of $U$ in the middle equality,
we obtain
\begin{equation*}
\delta_{\rm v}\Gamma_{ij}\leq\langle\Gamma,U\rangle
=\langle C,U\rangle-a^\top\alpha-b^\top\beta
\leq\langle C,V\rangle-a^\top\alpha-b^\top\beta=G.
\end{equation*}
This contradicts $\Gamma_{ij}>G/\delta_{\rm v}$ and
proves~\eqref{eq:ot-safe-exclusion}.
Applying~\eqref{eq:ot-safe-exclusion} to $V^k$ from~\eqref{eq:ot-repair}
and $(\widehat\alpha^k,\widehat\beta^k)$ from~\eqref{eq:ot-support-labels} gives $S\subseteq T_k$ for every $k$.

\textup{(Finite identification.)} To prove (ii), consider
 the sequence generated by Algorithm~\ref{alg:ot} under the assumptions
of Theorem~\ref{thm:ot-global}. The associated primal--dual variables
converge to the strictly complementary solution
$(P_*,\alpha_\infty,\beta_\infty)$.
The repair bound following~\eqref{eq:ot-repair} gives
\begin{equation*}
\|V^k-P_*\|_1\leq\|V^k-Y^k\|_1+\|Y^k-P_*\|_1
\leq\|Y^ke_n-a\|_1+\|Y^k-P_*\|_1\longrightarrow0.
\end{equation*}
Since $Z^k\to-e_m\beta_\infty^\top$, taking column means in
\eqref{eq:ot-support-labels} gives $\widehat\beta^k\to\beta_\infty$.
Dual feasibility implies $C_{ij}-(\beta_\infty)_j\geq(\alpha_\infty)_i$.
Every row of $P_*$ contains a positive entry, where complementarity
makes this inequality an equality. Thus
$\min_j(C_{ij}-(\beta_\infty)_j)=(\alpha_\infty)_i$.
Continuity of the minimum of finitely many functions now gives
$\widehat\alpha^k\to\alpha_\infty$, and the remaining definitions
in~\eqref{eq:ot-support-labels} give
$\Gamma^k\to\Gamma_\infty$ and $\widehat G_k\to0$.
If $S^c\ne\varnothing$, strict complementarity gives
$\gamma\coleq\min_{S^c}(\Gamma_\infty)_{ij}>0$.
For all sufficiently large $k$, we have
$\Gamma^k_{ij}>\gamma/2$ on $S^c$ and
$\widehat G_k/\delta_{\rm v}<\gamma/2$.
Definition~\eqref{eq:ot-retained-set} then excludes every entry
outside $S$. Combining this with $S\subseteq T_k$ proves $T_k=S$.
If $S^c=\varnothing$, the inclusion already gives the conclusion.
\end{proof}

\begin{remark}[Variable fixing and sparse computation]
Fixing variables outside $T_k$ at zero preserves all optimal
solutions and reduces the number of primal variables from $mn$ to
$|T_k|$.
\end{remark}

\section{Detailed computational results}\label{app:ot-computational-details}
Tables~\ref{tab:ot-detail-cpu-general}--\ref{tab:ot-detail-gpu-structured}
supplement the results in Section~\ref{sec:ot-numerical} with iteration
counts and the three residuals. For the DOTmark summary rows, each
quantity is averaged over the same instances as the corresponding
main-table row, excluding TL and ML. The all-method comparisons and
DOMDEC rows report individual instances. ``Iter'' denotes the iteration
count reported by each solver or its mean. Iteration counts from
different methods are not directly comparable. The dagger, TL and ML
have the meanings specified in Section~\ref{sec:ot-numerical}.
Unavailable entries are marked by a dash.

\providecommand{\otnum}[3]{\mbox{#1e\makebox[0.65em][c]{\ensuremath{#2}}#3}}
\begingroup
\setlength{\LTpre}{6pt}
\setlength{\LTpost}{6pt}
\begin{footnotesize}
\begin{longtable}[c]{llrrrrr}
\caption{Detailed CPU results in the generic-cost-matrix setting at $L=64$.\label{tab:ot-detail-cpu-general}} \\
\toprule
$L$ & Solver & Time (s) & Iter & $\eta_{\rm p}$ & $\eta_{\rm d}$ & $\eta_{\rm g}$ \\
\midrule
\endfirsthead
\multicolumn{7}{c}{Table \thetable\ continued}\\
\toprule
$L$ & Solver & Time (s) & Iter & $\eta_{\rm p}$ & $\eta_{\rm d}$ & $\eta_{\rm g}$ \\
\midrule
\endhead
\midrule
\multicolumn{7}{r}{Continued on next page}\\
\endfoot
\bottomrule
\endlastfoot
64 & BADMM & 25.023 & 95 & \otnum{2.47}{-}{06} & \otnum{9.68}{-}{06} & \otnum{2.09}{-}{06} \\
64 & POT-EMD & 7.520 & \textendash & \otnum{1.30}{-}{17} & \otnum{2.32}{-}{14} & \otnum{1.90}{-}{13} \\
64 & CPLEX-Net & 55.890 & 990466 & \otnum{1.11}{-}{18} & \otnum{4.45}{-}{20} & \otnum{1.08}{-}{19} \\
64 & SqSN & 62.766 & 44 & \otnum{4.86}{-}{06} & \otnum{1.04}{-}{06} & \otnum{7.74}{-}{06} \\
64 & Gurobi-Bar & 336.696 & 13 & \otnum{6.00}{-}{14} & \otnum{0.00}{+}{00} & \otnum{4.43}{-}{06} \\
64 & CPLEX-Bar & 655.793\rlap{$^{\dagger}$} & 23 & \otnum{6.86}{-}{07} & \otnum{0.00}{+}{00} & \otnum{2.28}{-}{05} \\
64 & HPR & TL & \textendash & \textendash & \textendash & \textendash \\
64 & PDLP & TL & \textendash & \textendash & \textendash & \textendash \\
64 & Gurobi-PDHG & TL & \textendash & \textendash & \textendash & \textendash \\
64 & HiGHS-Bar & TL & \textendash & \textendash & \textendash & \textendash \\
\end{longtable}
\end{footnotesize}

\begin{footnotesize}
\begin{longtable}[c]{llrrrrr}
\caption{Detailed CPU results for selected methods in the generic-cost-matrix setting.\label{tab:ot-detail-cpu-selected}} \\
\toprule
$L$ & Solver & Time (s) & Iter & $\eta_{\rm p}$ & $\eta_{\rm d}$ & $\eta_{\rm g}$ \\
\midrule
\endfirsthead
\multicolumn{7}{c}{Table \thetable\ continued}\\
\toprule
$L$ & Solver & Time (s) & Iter & $\eta_{\rm p}$ & $\eta_{\rm d}$ & $\eta_{\rm g}$ \\
\midrule
\endhead
\midrule
\multicolumn{7}{r}{Continued on next page}\\
\endfoot
\bottomrule
\endlastfoot
64 & BADMM & 30.669 & 136.0 & \otnum{5.18}{-}{06} & \otnum{8.50}{-}{06} & \otnum{4.31}{-}{06} \\
64 & POT-EMD & 16.120 & \textendash & \otnum{5.42}{-}{17} & \otnum{2.84}{-}{13} & \otnum{8.34}{-}{13} \\
64 & SqSN & 50.792 & 39.7 & \otnum{4.88}{-}{06} & \otnum{9.22}{-}{07} & \otnum{6.47}{-}{06} \\
64 & CPLEX-Net & 47.399 & 829844.1 & \otnum{1.92}{-}{17} & \otnum{7.34}{-}{19} & \otnum{1.34}{-}{17} \\[5pt]
128 & BADMM & 528.449 & 128.0 & \otnum{2.22}{-}{06} & \otnum{7.24}{-}{06} & \otnum{4.25}{-}{06} \\
128 & POT-EMD & 1099.711 & \textendash & \otnum{5.01}{-}{17} & \otnum{2.71}{-}{12} & \otnum{1.30}{-}{11} \\
128 & SqSN & 1554.204 & 70.8 & \otnum{1.11}{-}{06} & \otnum{6.07}{-}{07} & \otnum{8.33}{-}{06} \\
128 & CPLEX-Net & 2136.888 & 10168026.1 & \otnum{2.83}{-}{18} & \otnum{3.68}{-}{19} & \otnum{2.66}{-}{17} \\
\end{longtable}
\end{footnotesize}

\begin{footnotesize}
\begin{longtable}[c]{llrrrrr}
\caption{Detailed CPU results in the two-dimensional-kernel setting.\label{tab:ot-detail-cpu-structured}} \\
\toprule
$L$ & Solver & Time (s) & Iter & $\eta_{\rm p}$ & $\eta_{\rm d}$ & $\eta_{\rm g}$ \\
\midrule
\endfirsthead
\multicolumn{7}{c}{Table \thetable\ continued}\\
\toprule
$L$ & Solver & Time (s) & Iter & $\eta_{\rm p}$ & $\eta_{\rm d}$ & $\eta_{\rm g}$ \\
\midrule
\endhead
\midrule
\multicolumn{7}{r}{Continued on next page}\\
\endfoot
\bottomrule
\endlastfoot
64 & BADMM & 1.328 & 136.0 & \otnum{5.18}{-}{06} & \otnum{8.50}{-}{06} & \otnum{4.31}{-}{06} \\
64 & HOT & 36.803 & 2440.0 & \otnum{3.60}{-}{06} & \otnum{5.51}{-}{08} & \otnum{5.63}{-}{06} \\[5pt]
128 & BADMM & 9.769 & 128.0 & \otnum{2.22}{-}{06} & \otnum{7.24}{-}{06} & \otnum{4.25}{-}{06} \\
128 & HOT & 631.115 & 3988.9 & \otnum{3.06}{-}{06} & \otnum{1.71}{-}{08} & \otnum{8.61}{-}{06} \\[5pt]
256 & BADMM & 89.237 & 106.5 & \otnum{1.20}{-}{06} & \otnum{7.54}{-}{06} & \otnum{4.99}{-}{06} \\
256 & HOT & TL & \textendash & \textendash & \textendash & \textendash \\[5pt]
512 & BADMM & 708.348 & 98.0 & \otnum{1.04}{-}{06} & \otnum{8.88}{-}{06} & \otnum{3.85}{-}{06} \\
512 & HOT & TL & \textendash & \textendash & \textendash & \textendash \\
\end{longtable}
\end{footnotesize}

\begin{footnotesize}
\begin{longtable}[c]{llrrrrr}
\caption{Detailed GPU results in the generic-cost-matrix setting at $L=64$.\label{tab:ot-detail-gpu-general}} \\
\toprule
$L$ & Solver & Time (s) & Iter & $\eta_{\rm p}$ & $\eta_{\rm d}$ & $\eta_{\rm g}$ \\
\midrule
\endfirsthead
\multicolumn{7}{c}{Table \thetable\ continued}\\
\toprule
$L$ & Solver & Time (s) & Iter & $\eta_{\rm p}$ & $\eta_{\rm d}$ & $\eta_{\rm g}$ \\
\midrule
\endhead
\midrule
\multicolumn{7}{r}{Continued on next page}\\
\endfoot
\bottomrule
\endlastfoot
64 & BADMM & 0.127 & 95 & \otnum{2.47}{-}{06} & \otnum{9.68}{-}{06} & \otnum{2.09}{-}{06} \\
64 & PDOT & 6.241 & 1024 & \otnum{6.39}{-}{06} & \otnum{3.80}{-}{07} & \otnum{1.02}{-}{06} \\
64 & HPR-LP-C & 27.553\rlap{$^{\dagger}$} & 13650 & \otnum{2.09}{-}{05} & \otnum{6.51}{-}{02} & \otnum{7.29}{-}{06} \\
64 & cuOpt-Bar & 35.472\rlap{$^{\dagger}$} & 46 & \otnum{3.01}{-}{06} & \otnum{4.36}{-}{20} & \otnum{2.38}{-}{03} \\
64 & cuPDLPx & 76.231\rlap{$^{\dagger}$} & 19400 & \otnum{1.66}{-}{05} & \otnum{4.12}{-}{08} & \otnum{1.16}{-}{05} \\
64 & Gurobi-PDHG & 106.283\rlap{$^{\dagger}$} & 19914 & \otnum{8.57}{-}{04} & \otnum{2.90}{-}{07} & \otnum{7.94}{-}{07} \\
64 & cuOpt-PDLP & 188.049\rlap{$^{\dagger}$} & 49600 & \otnum{1.54}{-}{05} & \otnum{3.19}{-}{06} & \otnum{5.36}{-}{06} \\
64 & W-BADMM & TL & \textendash & \textendash & \textendash & \textendash \\
\end{longtable}
\end{footnotesize}

\begin{footnotesize}
\begin{longtable}[c]{llrrrrr}
\caption{Detailed GPU results for selected methods in the generic-cost-matrix setting.\label{tab:ot-detail-gpu-selected}} \\
\toprule
$L$ & Solver & Time (s) & Iter & $\eta_{\rm p}$ & $\eta_{\rm d}$ & $\eta_{\rm g}$ \\
\midrule
\endfirsthead
\multicolumn{7}{c}{Table \thetable\ continued}\\
\toprule
$L$ & Solver & Time (s) & Iter & $\eta_{\rm p}$ & $\eta_{\rm d}$ & $\eta_{\rm g}$ \\
\midrule
\endhead
\midrule
\multicolumn{7}{r}{Continued on next page}\\
\endfoot
\bottomrule
\endlastfoot
64 & BADMM & 0.131 & 136.0 & \otnum{5.18}{-}{06} & \otnum{8.50}{-}{06} & \otnum{4.31}{-}{06} \\
64 & PDOT & 219.652 & 44595.2 & \otnum{2.70}{-}{06} & \otnum{2.28}{-}{06} & \otnum{5.96}{-}{06} \\[5pt]
128 & BADMM & 2.018 & 128.0 & \otnum{2.22}{-}{06} & \otnum{7.24}{-}{06} & \otnum{4.25}{-}{06} \\
128 & PDOT & 642.218 & 10093.7 & \otnum{4.73}{-}{06} & \otnum{1.99}{-}{06} & \otnum{7.54}{-}{06} \\[5pt]
256 & BADMM & 27.006 & 106.5 & \otnum{1.20}{-}{06} & \otnum{7.54}{-}{06} & \otnum{4.99}{-}{06} \\
256 & PDOT & ML & \textendash & \textendash & \textendash & \textendash \\
\end{longtable}
\end{footnotesize}

\begin{footnotesize}
\begin{longtable}[c]{llrrrrr}
\caption{Detailed GPU results in the two-dimensional-kernel setting.\label{tab:ot-detail-gpu-structured}} \\
\toprule
$L$ & Solver & Time (s) & Iter & $\eta_{\rm p}$ & $\eta_{\rm d}$ & $\eta_{\rm g}$ \\
\midrule
\endfirsthead
\multicolumn{7}{c}{Table \thetable\ continued}\\
\toprule
$L$ & Solver & Time (s) & Iter & $\eta_{\rm p}$ & $\eta_{\rm d}$ & $\eta_{\rm g}$ \\
\midrule
\endhead
\midrule
\multicolumn{7}{r}{Continued on next page}\\
\endfoot
\bottomrule
\endlastfoot
256 & BADMM & 0.379 & 106.5 & \otnum{1.20}{-}{06} & \otnum{7.54}{-}{06} & \otnum{4.99}{-}{06} \\
256 & HOT & 7.209\rlap{$^{\dagger}$} & 510.0 & \otnum{1.69}{-}{04} & \otnum{7.21}{-}{07} & \otnum{1.53}{-}{03} \\
256 & HALO & 2.164\rlap{$^{\dagger}$} & \textendash & \otnum{1.91}{-}{05} & \otnum{3.67}{-}{07} & \otnum{3.30}{-}{06} \\
256 & HELLO & 8.454 & \textendash & \otnum{9.14}{-}{06} & \otnum{1.87}{-}{06} & \otnum{2.43}{-}{06} \\[5pt]
512 & BADMM & 1.872 & 98.0 & \otnum{1.04}{-}{06} & \otnum{8.88}{-}{06} & \otnum{3.85}{-}{06} \\
512 & HOT & 57.768\rlap{$^{\dagger}$} & 490.0 & \otnum{8.05}{-}{05} & \otnum{4.13}{-}{07} & \otnum{1.30}{-}{03} \\
512 & HALO & 11.682\rlap{$^{\dagger}$} & \textendash & \otnum{1.31}{-}{05} & \otnum{4.74}{-}{07} & \otnum{3.26}{-}{06} \\
512 & HELLO & 158.148 & \textendash & \otnum{3.77}{-}{06} & \otnum{2.08}{-}{06} & \otnum{5.03}{-}{06} \\
\addlinespace[5pt]
1024 & BADMM & 11.253 & 95 & \otnum{3.55}{-}{07} & \otnum{9.57}{-}{06} & \otnum{4.77}{-}{06} \\
1024 & HOT & ML & \textendash & \textendash & \textendash & \textendash \\
1024 & HALO & 145.648 & \textendash & \otnum{9.64}{-}{07} & \otnum{2.30}{-}{09} & \otnum{1.68}{-}{06} \\
1024 & HELLO & TL & \textendash & \textendash & \textendash & \textendash \\[5pt]
2048 & BADMM & 98.683 & 95 & \otnum{1.77}{-}{07} & \otnum{9.59}{-}{06} & \otnum{4.75}{-}{06} \\
2048 & HOT & ML & \textendash & \textendash & \textendash & \textendash \\
2048 & HALO & 2346.401 & \textendash & \otnum{1.47}{-}{07} & \otnum{1.24}{-}{09} & \otnum{1.30}{-}{07} \\
2048 & HELLO & TL & \textendash & \textendash & \textendash & \textendash \\
\end{longtable}
\end{footnotesize}

\endgroup

\end{document}